\documentclass[a4paper,11pt]{amsart}
\usepackage[english]{babel}
\usepackage{amsthm}
\usepackage{amssymb}
\usepackage{subfig}
\usepackage{graphicx}
\usepackage{xcolor}
\PassOptionsToPackage{hyphens}{url}
\usepackage{hyperref}
\usepackage[utf8]{inputenc}
\usepackage{fancybox}
\usepackage{caption}
\usepackage[all]{xy}
\usepackage{longtable}

\usepackage{bookmark}
\hypersetup{
     colorlinks=true,
     linkcolor=blue,
     filecolor=blue,
     citecolor = blue,
     urlcolor=cyan,
}

\newcommand{\A}{{\mathbb A}}
\newcommand{\N}{{\mathbb N}}
\newcommand{\cH}{{\mathbb H}}

\newcommand{\Z}{{\mathbb Z}}
\newcommand{\R}{{\mathbb R}}
\newcommand{\C}{{\mathbb C}}
\newcommand{\bs}{{\mathbb S}}
\newcommand{\Q}{{\mathbb Q}}
\newcommand{\bO}{{\mathbf\Omega}}
\newcommand{\CQ}{{\C^*_\Q}}
\newcommand{\SQ}{{S^{1}_{\Q}}}
\newcommand{\PQ}{\C P^{1}_\Q}
\newcommand{\bS}{\mathbb{S}}
\newcommand{\bu}{{\mathbf{u}}}
\newcommand{\bv}{{\mathbf{v}}}
\newcommand{\bw}{{\mathbf{w}}}
\newcommand{\bd}{{\mathbf{d}}}
\newcommand{\bz}{{\mathbf{z}}}
\newcommand{\I}{{\mathcal O}}
\newcommand{\La}{{\boldsymbol\Lambda}}
\newcommand{\gln}{{\text{GL}(n,\C)}}
\newcommand{\Lo}{{{\boldsymbol\Lambda}_0^\infty(G)}}
\newcommand{\Ld}{{{\boldsymbol\Lambda}^\infty(G)}}

\newtheorem{definition}{Definition}
\newtheorem{example}{Example}

\newtheorem{remark}{Remark}

\newtheorem{theorem}{Theorem}
\newtheorem{corollary}{Corollary}
\newtheorem{proposition}{Proposition}
\newtheorem{lemma}{Lemma}
\newtheorem{conjecture}{Conjecture}

\newtheorem{question}{Question}

\def\noi{\noindent}

\newcommand{\XFF}{X_{\mathrm{FF}}}
\newcommand{\BdR}{\mathbb{B}_{\mathrm{dR}}}
\newcommand{\BdRp}{\mathbb{B}^{+}_{\mathrm{dR}}}
\newcommand{\Rob}{\mathcal{R}}
\newcommand{\Qp}{\mathbb{Q}_p}
\newcommand{\Zp}{\mathbb{Z}_p}
\newcommand{\Cp}{\mathbb{C}_p}

\title[Adelic loop groups and perfectoid analogies]{Adelic Loop Groups and Perfectoid Analogies: Factorization and Holomorphic Bundles on the Adelic Projective Line}

\author{Alberto Verjovsky}
\thanks{I would like to acknowledge Proyecto {\bf PAPIIT IN103324 (DGAPA, UNAM, M\'exico)} for its financial support.}
\address{Instituto de Matem\'aticas, Unidad Cuernavaca,\\ Universidad Nacional Aut\'onoma de M\'exico,\\ Av. Universidad S/N, C.P. 62210, Cuernavaca, Morelos, M\'exico}
\email{alberto@matcuer.unam.mx}



\begin{document}

\begin{abstract}
We develop the theory of {\it adelic loop groups} on the universal
one-dimensional solenoid $S^1_\Q=(\R\times\widehat\Z)/\Z_{\mathrm{diag}}$,
the compact abelian group whose Pontryagin dual is $\Q$ rather than $\Z$.
We introduce the adelic projective line $\C P^1_\Q$, its ring of
Laurent--Puiseux series, and holomorphic vector bundles defined by
solenoidal clutching data.  We show that its Picard group is the additive
group of rational numbers,
\[
\mathrm{Pic}(\C P^1_\Q)\cong\Q.
\]
The paper proves a scalar Wiener--Birkhoff factorization theorem, a matrix
Wiener lemma, exact factorization for ordered triangular and small-norm
cocycles, a density theorem for factorable matrix loops associated with the
Wiener algebra $\mathfrak{W}_\Q$, and a Birkhoff--Grothendieck splitting
theorem in the pro-algebraic category.  The resulting framework leads to the
\emph{Solenoidal Birkhoff--Grothendieck conjecture}
(Conjecture~\ref{solenoidal-BG}), which asserts that every
$g\in\mathrm{GL}_n(\mathfrak{W}_\Q)$ admits a factorization
$g=h_-^{-1}\operatorname{diag}(\chi_{q_1},\ldots,\chi_{q_n})h_+$ with
$h_\pm\in\mathrm{GL}_n(\mathfrak{W}^\pm_\Q)$ and $q_i\in\Q$.  We also develop
the K\"ahler, Grassmannian, and Morse--Bott geometry of the adelic loop group
in the spirit of Pressley--Segal.  Finally, we relate the theory to
perfectoid geometry: the Fargues--Fontaine curve provides a non-archimedean
structural counterpart of $\C P^1_\Q$ at the level of rational slope data,
Kedlaya's slope theory supplies the $p$-adic analogue of Wiener--Birkhoff
factorization, and the Fargues--Fontaine classification serves as a proved
perfectoid model for the matrix splitting problem formulated here.  This
comparison leads to a Harder--Narasimhan reformulation of the Solenoidal
Birkhoff--Grothendieck conjecture.
\end{abstract}

\maketitle

\tableofcontents

\subsection*{Relation with earlier joint work}
Several foundational constructions used in this paper---the universal
solenoid, the adelic loop group, the K\"ahler structure, the energy
functional and its Birkhoff decomposition, the adelic Grassmannian and
Iwasawa decomposition, and the pro-algebraic Birkhoff--Grothendieck
splitting theorem---originate in joint work with Juan Manuel Burgos and were
announced in the preprint \cite{BuVe}.  They are recalled here as part of the
background and are cited at the relevant points.  The present paper develops
this framework further by introducing the Wiener-algebra theory of
\S\ref{BF} onward, formulating the Solenoidal Birkhoff--Grothendieck
Conjecture~\ref{solenoidal-BG}, proving the corresponding scalar,
triangular, small-norm, density, and pro-algebraic results, and developing
the comparison with perfectoid geometry.

\section{Preliminaries on the Adelic solenoid}\label{adelic_solenoid}

For positive integers $n\mid m$, let
$p_{n,m}:S^1	o S^1$ be the covering map
\[
 p_{n,m}(z)=z^{m/n}.
\]

This determines a projective system of covering spaces $\{S^1,p_{n,m}\}_{n,m \geq 1, n|m}$ whose projective limit is the \textsf{universal one--dimensional solenoid} or \textsf{adelic solenoid}:
\[ S^1_\Q:=\lim_{\underset{p_{n,m}}\longleftarrow}S^1. \] 
Thus $S^1_\Q$ consists of sequences $(z_n)_{n\geq1}$ with $z_n\in S^1$ such that
$p_{n,m}(z_m)=z_n$ whenever $n\mid m$.

The canonical projections of the inverse limit are the maps
$\pi_n:S^1_\Q\to S^1$ defined by
\[
\pi_n\left(\left(z_j\right)_{j\in\N}\right)=z_n 
\]

\noi and they define the solenoid topology as the initial topology of the family. The solenoid is an abelian topological group and each $\pi_n$ is an epimorphism.  In particular each $\pi_n$ is a character which determines a locally trivial $\hat{\Z}$--bundle structure where the group

\[
\hat{\Z}:=\lim_{\underset{n|m}\longleftarrow}
\left\{\Z/m\Z\,\overset{p_{n,m}}\longrightarrow\,\Z/n\Z\right\}, 
\]

\noi (where $p_{n,m}:\Z/m\Z\longrightarrow\,\Z/n\Z$ is the canonical epimorphism when $n|m$), is the profinite completion of $\Z$. Then, $\hat{\Z}$ is a compact, perfect and totally disconnected abelian topological group homeomorphic to the Cantor set. 

Since $\widehat\Z$ is the profinite completion of $\Z$, the canonical inclusion
$\Z\hookrightarrow\widehat\Z$ has dense image.
We have an inclusion $\hat\Z\overset{\phi}\to{S^1_\Q}$ and a short exact sequence:
$$0\to{\hat\Z}\overset{\phi}\rightarrow S^1_\Q\overset{\pi_1} \rightarrow S^1\to1$$

The solenoid $S^1_\Q$ can also be realized as the orbit space of the $\Q$--bundle structure $\Q \hookrightarrow \mathbb{A} \to \A/\Q$, where $\A$ is the {\bf adele group} of the rational numbers which is a locally compact Abelian group, $\Q$ is a discrete subgroup of $\A$ and $\A/\Q \cong S^1_\Q$ is a compact Abelian group (\cite{RV}). From this perspective, $\A/\Q$ can be seen as a projective limit whose $n$--th component corresponds to the unique covering of degree $n\geq 1$ of $S^1_\Q$.

By considering the properly discontinuously diagonal free action of $\Z$ on $\hat{\Z}\times\R$  given by
\[ n\cdot(x,t)=(x+n,t-n), \quad (n\in \Z, \, x\in\hat\Z, \,t\in\R)\]
the solenoid $S^1_\Q$ is identified with the orbit space
$\widehat\Z\times_{\Z}\R$, where $\Z$ acts on $\R$ by integer
translations and on $\widehat\Z$ by translations.
\begin{definition}[Base leaf]\label{def:base-leaf}
The path-connected component of the identity $1\in S^1_\Q$ is called the
\emph{base leaf}. It is a densely immersed copy of $\R$ \cite{Od}.
\end{definition}

Hence $S^1_\Q$ is a compact, connected, abelian topological group and also a one-dimensional lamination where each ``leaf" is a simply connected one-dimensional manifold, homeomorphic to the universal covering space $\R$ of $S^1$, and a typical ``transversal" is isomorphic to the Cantor group $\hat{\Z}$. The solenoid $S^1_\Q$ also has a leafwise $\mathrm{C}^\infty$ Riemannian metric (i.e., $\mathrm{C}^\infty$ along the leaves) which renders each leaf isometric to the real line with its standard metric $dx$. So, it makes sense to speak of a rigid translation along the leaves. The leaves also have a natural order equivalent to the order of the real line hence also an orientation.

Summarizing the above discussion we have the commutative diagram: 
\begin{equation}\label{diagram_I}
\xymatrix{
S^{1}_{\mathbb{Q}} = \varprojlim S^{1} & \dots \ar[r] & S^{1} \ar[r]^{p_{m,n}} & S^{1} & \dots \ar[r] & S^{1} \\
\hat{\mathbb{Z}} = \varprojlim \mathbb{Z}/n\mathbb{Z} \ar@{^{(}->}[u]^{\phi} & \dots \ar[r] & \mathbb{Z}/n\mathbb{Z} \ar@{^{(}->}[u]_{l \mapsto e^{2\pi il/n}} \ar[r]^{p_{m,n}} & \mathbb{Z}/m\mathbb{Z} \ar@{^{(}->}[u]_{l \mapsto e^{2\pi il/m}} & \dots \ar[r] & \{0\} \ar@{^{(}->}[u]_{0 \mapsto 1}
}
\end{equation}
where $\hat{\Z}$ is the adelic profinite completion of the integers and the image of the group monomorphism $\phi:(\hat{\Z},+)\rightarrow (S^{1}_{\Q},\cdot)$ is the \textsf{principal fiber}. We notice that $\pi_{n}(x)= \pi_{n}(y)$ implies $\pi_{n}(y^{-1}x)=1$ and therefore $y^{-1}x= \phi(a)$ where $a\in n\hat{\Z}$ for some $n\in\Z\subset\hat{\Z}$.

The base leaf of Definition~\ref{def:base-leaf} is the image of the
injective homomorphism $\nu:\R\to S^1_\Q$ determined by
\[
 \pi_n(\nu(t))=e^{2\pi i t/n},\qquad n\geq1.
\]
Equivalently, it is described by the commutative inverse-limit diagram
\begin{equation}\label{diagram_II}
\xymatrix{
    S^{1}_{\Q}= \varprojlim S^{1} \quad \ldots\ar[r] & S^{1} \ar[r]^{p_{m,n}}
    & S^{1} \ldots\ar[r] & S^{1} \\
    \R \ar@{^{(}->}[u]_{\nu} \quad \ldots\ar[r]^{=} & \R \ar[u]_{t\mapsto e^{2\pi it/n}} \ar[r]^{=}
    & \R \ar[u]_{t\mapsto e^{2\pi it/m}} \ldots\ar[r]^{=} & \R \ar[u]_{t\mapsto e^{2\pi it}}}.
\end{equation}
Comparing \eqref{diagram_I} and \eqref{diagram_II}, one obtains
$\nu(n)=\phi(n)$ for every $n\in\Z$.

\begin{definition}[Adelic exponential and canonical flow]\label{exp}
Let $\boldsymbol{Exp}:\R\times\hat{\Z}\rightarrow S^{1}_{\Q}$ such that 
 $\boldsymbol{Exp}(t,a)= \nu(t)\phi(a)$. Then $\boldsymbol{Exp}$ is an epimorphism.  With the quotient convention
above, its kernel is the diagonal copy
\[
\{(n,-n):n\in\Z\}\subset\R\times\widehat\Z,
\]
with the harmless sign depending on the convention for the diagonal action.
This is the adelic analogue of the ordinary exponential map
$t\mapsto e^{2\pi\boldsymbol{i}t}$.
Consider the translation flow $\hat\varphi_t$ on $\R\times\hat{\Z}$
given by $\hat\varphi_t(s,\bz)=(s+t,\bz)$. This flow commutes with the
map $(t,\bz)\mapsto(t+1,\bz+\mathbf1)$ so it defines a translation flow 
$\varphi_t:\SQ\to\SQ$ which preserves the 1-dimensional leaves of the lamination
structure of $\SQ$; it amounts to translations by elements in the base leaf. This flow is called {\bf the canonical flow} of $\SQ$.
The vector field (in the solenoidal sense) is the unit vector field along the
leaves. The flow is in fact the one-parameter group of translations by the
connected component of the identity (the base leaf). 
\end{definition}

\noi {\bf We will always treat $\SQ$ as a multiplicative abelian group.}

\begin{remark}\label{integeradeles} If $\mathbb{A} _{\mathbb{Z}}$ denotes the ring of integral adeles \cite{RV} then: 
 $\mathbb{R}\times{\hat{\mathbb {Z}}}=\mathbb{A}_{\mathbb{Z}}=
 \mathbb{R}\times\prod_{p}\mathbb{Z}_{p}$. 
 The map $\Z\overset{i}\hookrightarrow\mathbb{R}\times{\hat{\mathbb {Z}}}$, 
 $\,n\mapsto(n,\mathbf{n})$,
 where $\mathbf{n}$ corresponds to the natural inclusion of $\Z$ into $\hat\Z$,
injects $\Z$ into a discrete co-compact subgroup $\Gamma$ so that 
$\SQ=(\R\times\hat\Z)/\Gamma$. The subgroup 
$\left\{(0,\mathbf{n}): n\in\Z\right\}$ is dense in 
$\left\{0\right\}\times\hat\Z$. This implies that the canonical flow is minimal.
\end {remark}
\medskip
The previous constructions can be extended to the inverse limit of the multiplicative group $\C^{*}$ as well as to the multiplicative semigroup $\C$. In other words, if 
${\mathbf z}_n:\C\to\C$ denotes the map $z\mapsto{z^n}$ we can take the inverse limits under the partial order of divisibility of the integers of the finite coverings of $\C^*$ and the branched coverings of $\C$:   
\begin{definition}\label{C*QandCQ}
$\underset{{\mathbf z}^n}\varprojlim\, \C^*\overset{def}=\C^{*}_\Q\;$ and
$\;\underset{{\mathbf z}^n}\varprojlim \,\C\overset{def}=\C_\Q$, respectively.
\end{definition}
Since, when $n$ divides $m$, the maps $p_{n,m}$, 
$z\overset{p_{n,m}}\longrightarrow{z^{\frac{m}{n}}}$, are defined as branched coverings
of $\C P^{1}$ with branching points the north and south poles $\boldsymbol{0}$ and $\boldsymbol{\infty}$, we  can consider the inverse limit:
\[
\C P^{1}_\Q:= \varprojlim_{p_{n,m}} \left\{p_{n,m}:\C P^{1}\to\C P^{1},\, n|m\right\}.
\]
\begin{definition}[Adelic projective line]We call $\C P^{1}_\Q$ the {\bf adelic projective line}.
\end{definition}
Since the maps $p_{n,m}$ preserve the unit circle, if we restrict this inverse limit to the unit circle, which is 
the equator of the Riemann sphere we recover the ad\`elic solenoid $\SQ$. 

If we restrict the inverse system to $\C^*$, i.e. the complement of the poles,  we obtain:

\[\C ^*_\Q:= \varprojlim_{p_{n,m}} \left\{p_{n,m}:\C ^*\to\C ^*,\, n|m\right\}.\]

In the first case we obtain the topological abelian group: the \textsf{algebraic solenoidal group} $\C^{*}_\Q$. This group is laminated by densely immersed copies of $\C$. In the second case, we obtain the semigroup $\C_\Q$ which is a ramified covering of $\C$ at zero and topologically it is homeomorphic to the open cone over the adelic solenoid. The singularity of a cone over a solenoidal torus $S^{1}_{\Q}\times\ldots  S^{1}_{\Q}$ will be called a \textsf{solenoidal cusp} or simply a cusp. 

The following is the complex version of the diagram \ref{diagram_I}:
\begin{equation}\label{diagram_III}
\xymatrix{
    \mathbb{C}^{*}_{\mathbb{Q}}= \varprojlim \mathbb{C}^{*} \quad \ldots\ar[r] & \mathbb{C}^{*} \ar[r]^(0.45){p_{m,n}} & \mathbb{C}^{*} \ldots\ar[r] & \mathbb{C}^{*} \\
    \hat{\mathbb{Z}}= \varprojlim \mathbb{Z}/n\mathbb{Z} \quad \ldots\ar[r] \ar@{^{(}->}[u]^{\phi} & \mathbb{Z}/n\mathbb{Z} \ar@{^{(}->}[u]_{l\,\mapsto{e^{2\pi il/n}}} \ar[r]^(0.45){p_{m,n}} & \mathbb{Z}/m\mathbb{Z} \ar@{^{(}->}[u]_{l\,\mapsto{e^{2\pi il/m}}} \ldots\ar[r] & \{0\} \ar@{^{(}->}[u]_{0\,\mapsto 1}
}
\end{equation}

\section{Holomorphic vector bundles over $\PQ$.}\label{hol-vect-bun}

\subsection{Laurent--Puiseux series in $\C^*_\Q$}\label{lp} 
The abelian group $\CQ$ is isomorphic to the direct product of $S^{1}_{\Q}$   
with the multiplicative group of the positive reals i.e. 
$\CQ=\SQ\times\R_{>0}$. This follows because it is an inverse limit of groups over $\C^*$ and $\C^*\simeq\bs^1\times\R_{>0}$. If $z\in\CQ$ using this factorization 
$z=(\mathbf{u},t)$, with $\mathbf{u}\in\SQ,\, t\in\R_{>0}$, we will write $z=t\mathbf{u}$ (i.e. we introduce ``\emph{polar coordinates}'' in $\CQ$). 

\begin{definition}\label{disks}
If $z\in\CQ$ we define the absolute value $|z|$ of $z$ as follows: $|z|=t$ if $z=t\bu$ (in polar coordinates). If one considers 
$\C_\Q=\CQ\cup{\boldsymbol{0}}$ (the inverse limit of
the branched self-coverings $z\mapsto{z^n}$ of $\C$ with canonical projection 
$p_1:\C_Q\to\C$ and $\boldsymbol{0}=p_{1}^{-1}\left\{0\right\}$ (see definition \ref{C*QandCQ}) then we define 
$|\boldsymbol{0}|=0$.

Let $r>0$. Define the open and closed disks of radius $r$ centered at the origin as follows:
\begin{align*}
D^+_\Q(r)&=\left\{q\in\C_Q\,:\, |q|<r \,\text{or} \,\, q=\boldsymbol{0} \right\},\\
\bar{D}^+_\Q(r)&=\left\{q\in\C_Q\,:\, |q|\leq{r}\,\text{or}\,q=\boldsymbol{0} \right\}.
\end{align*}
Analogously, one defines the open and closed disks of radius $r$ centered at
$\boldsymbol{\infty}\in\PQ$, where
$\boldsymbol{\infty}=p_1^{-1}(\{\infty\})$ and $p_1:\PQ\to\C P^1$ is the canonical projection:
\begin{align*}
D^-_\Q(r)&=\left\{q\in\C_Q\,:\, |q|>r\,\,\text{or}\,\,q=\boldsymbol{\infty} \right\},\\
\bar{D}^-_\Q(r)&=\left\{q\in\C_Q\,:\, |q|\geq{r}\,\,\text{or}\,\,q=\boldsymbol{\infty} \right\}.
\end{align*}
\end{definition}
 One has the decomposition of the adelic projective line into two closed ``hemispheres":
\[
\PQ=\bar{D}^+_\Q(1)\cup\bar{D}^-_\Q(1), \quad 
\text{with ``equator''}\,\, \bar{D}^+_\Q(1)\cap\bar{D}^-_\Q(1)=\SQ. 
\]
 \noi The south and north poles are $\boldsymbol0$ and $\boldsymbol\infty$, respectively.

\subsection{Rational dilation symmetry}\label{subsec:rational-dilation-symmetry}

We recall that the Pontryagin dual $K$ of the compact abelian group
$S^{1}_{\Q}$ is the additive group of rational numbers $(\Q,+)$:
\[
K\overset{def}=
\left\{f:S^{1}_{\Q}\to\bs^1\, : \,\text{$f$ is a continuous homomorphism}  \right\}\simeq\Q.
\]
This gives the solenoidal theory an additional arithmetic symmetry which is
absent from the ordinary circle.  Since $\Q$ is given the discrete topology,
\[
\operatorname{Aut}_{\mathrm{cont}}(\Q,+)
=
\operatorname{Aut}(\Q,+)
\cong \Q^*,
\]
where $a\in\Q^*$ acts by multiplication $q\mapsto aq$.  By Pontryagin
duality, each such automorphism induces a continuous automorphism of the
universal solenoid
\[
\xi_a:S^1_{\Q}\longrightarrow S^1_{\Q}.
\]
Equivalently, if $m_a:\Q\to\Q$ denotes multiplication by $a$, then
\[
\xi_a(x)=x\circ m_a,
\qquad x\in\widehat{\Q}=S^1_{\Q}.
\]
On characters one has
\[
\chi_q(\xi_a x)=\chi_{aq}(x).
\]
Thus the group $\Q^*$ acts on the rational Fourier spectrum by dilation,
\[
q\longmapsto aq.
\]
Consequently the rational Fourier theory, the Wiener algebra
$\mathfrak W_{\Q}$, the adelic loop groups built from it, and the rational
partial indices appearing in the solenoidal Birkhoff--Grothendieck theory
carry a natural rational dilation symmetry.

There is a useful nuance.  The subgroup $\Q^*_{>0}$ preserves the order on
$\Q$ and therefore preserves the decomposition into positive and negative
rational frequencies.  Hence $\Q^*_{>0}$ preserves the Wiener--Hopf
polarization
\[
\mathfrak W_{\Q}=\mathfrak W^-_{\Q}\oplus\mathfrak W^+_{\Q}.
\]
A negative rational dilation exchanges the two halves of the polarization.
Thus the full group $\Q^*$ acts naturally on the solenoidal Fourier theory,
whereas the order-preserving subgroup $\Q^*_{>0}$ is the symmetry group
compatible with the chosen positive/negative Wiener splitting.  For the
ordinary circle the character group is $\Z$, whose automorphism group is
only $\{\pm 1\}$; the universal solenoid therefore has a genuine arithmetic
dilation symmetry not present in classical loop group theory.

For each $q\in\Q^*$, using the dual of the automorphism of $\Q$, we obtain
an automorphism $\xi_q:S^{1}_{\Q}\to\SQ$.  In polar coordinates we define
\[\
\hat\xi_q:\C^*_{\Q}\longrightarrow \C^*_{\Q},
\qquad
\hat\xi_q(\mathbf u,t)=t^q\xi_q(\mathbf u),
\]
if $z=t\mathbf u$.  To simplify notation we write briefly
$\hat\xi_q(z)=z^q$.  Therefore, taking a rational power of an element in
$\CQ$ makes sense.  Moreover, unlike the classical single-valued theory of
rational powers on $\C^*$, we have:
\begin{enumerate}
\item For any rational $q\neq0$, the map $z\mapsto z^q$ is a continuous
isomorphism of $\CQ$.
\medskip
\item For $q\neq0$, the inverse of the map $z\mapsto z^q$ is the map
$z\mapsto z^{1/q}$.
\end{enumerate}

\begin{definition}\label{chc*} For each $q\in\Q$, define the homomorphism  
$\hat\chi_q:\CQ\to\C^*$ given by the formula
 $\hat\chi_q(z)=t^q\chi_q(\bu)$,
where $\chi_q$ is the character of $\CQ$ corresponding to $q$  and $z=t\bu$
in polar coordinates.
A series $S$ of the form $S=\underset{q\in\Q}\sum\,a_q$ where $a_q\in\C^n$ means
$\overset{\infty}{\underset{n=1}\sum}\,a_{q_n}$ for some numbering of the rationals 
$\Q=\left\{q_1,q_2,\cdots,\right\}$. If the series is absolutely convergent
i.e., the series $\overset{\infty}{\underset{n=1}\sum}\,|a_{q_n}|$ converges,  then the series $S$ converges to the same vector independently of the numbering. We say that $S$ converges absolutely.
\end{definition}
\begin{definition}[{\bf Holomorphic functions}]\label{holo} A continuous function $f:D^+_\Q(r)\rightarrow\C^n$ is said to be holomorphic
if it can be developed as a power series which is absolutely convergent and it is locally uniformly convergent (i.e. convergent uniformly on compact subsets):
\[
f(z)=\sum_{q\in\Q,\,q\geq0}\,a_q\hat\chi_q(z)=\sum_{q\in\Q,\, q\geq0}\,a_qt^q\chi_q(\bu),\quad f(\boldsymbol{0})=a_0
\]
where $a_q\in\C^n$, $\chi_q$ is the character corresponding to 
$q\geq0$ and $z=t\bu$, $0\leq{t}<{r}$. 

\noi Analogously, a continuous function $f:D^-_\Q(r)\rightarrow\C^n$ is said to be holomorphic
in the disk centered at $\boldsymbol{\infty}$ if it can be developed as an absolutely and locally uniformly convergent power series:
\[
f(z)=\sum_{q\in\Q,\,q\leq0}\,a_q\hat\chi_q(z)=\sum_{q\in\Q,\, q\leq0}\,a_qt^q\chi_q(\bu),
\]
where $a_q\in\C^n$, $\chi_q$ is the character corresponding to
$q\leq0$
and $z=t\bu$, $r<t<\infty$, where $f(\boldsymbol{\infty})=a_0$.
\end{definition}

\begin{definition}For real numbers $r$, $R$ with $0\leq{r}<{R}\leq\infty$, the solenoidal open annulus is
\[
A(r,R)=\left\{q\in\CQ\,:\, r<|q|<R \right\}.
\]
Note that $A(0,\infty)=\C^*$.
\end{definition}
\begin{definition}\label{holoannulus}
Let $f:A(r,R)\to\C^n$ then f is said to be holomorphic if f can be expressed
as an absolutely and locally uniform convergent series:
\[
f(z)=\sum_{q\in\Q}\,a_q\hat\chi_q(z)=\sum_{q\in\Q}\,a_qt^q\chi_q(\bu),
\]
where $a_q\in\C^n$, $\chi_q$ is the character corresponding to
each $q\in\Q$ (both signs occur; this is the annulus case, not restricted to $q\ge0$)
and $z=t\bu$, $r<t<R$.

\noi {\it The series is required to be convergent only for points in the annulus}. 
Such a series will be called a {\bf Laurent--Puiseux series} (since $q<0$ is allowed). 

\end{definition}
\begin{definition}\label{LPPoly-holo}
If the sum is of the form
\[
f(z)=\sum_{q\in{F}}\,a_q\hat\chi_q(z)=\sum_{q\in{F}}\,a_qt^q\chi_q(\bu),
\]
where $F\subset\Q$ is a finite set, then $f$ is called a {\bf finite
Laurent--Puiseux series}.  If, in addition, $F\subset\Q_{\geq 0}$, then
$f$ is called a {\bf Laurent--Puiseux polynomial}.
 
\end{definition}

\subsection{Winding numbers}\label{wn} Let $0\leq{r}<1<R\leq\infty$. Let $f:A(r,R)\to\C^*$ be a holomorphic map then, since $A(r,R)$ retracts strongly onto $\SQ$ (as is easily seen using polar coordinates) and $\C^*$ retracts strongly onto the circle $\bS^1$ we see that
the homotopy class of $f$ is determined by the homotopy class of the map $\hat{f}$
 from $\SQ$ to $\bS^1$ given by the formula $\hat{f}(\bu)=\frac{f(\bu)}{|f(\bu)|}$.
 By Scheffer's theorem \cite{Sc} $\hat{f}$ is homotopic to a character  
 $\chi_q:\SQ\to\bS^1$ corresponding to a rational number $q\in\Q$.
 \begin{definition}\label{hwn} The rational number $q$ in the previous paragraph is called
 the {\bf winding number} of $f$ and it is denoted $w(f)$.
 \end{definition}
 \begin{definition}\label{winding_annulus_annulus}For any $r$, $R$ such that $0\leq{r}<{R}\leq\infty$ we can define the winding number of a continuous map 
 $f:A(r,R)\to\C^*$. 
Consider the restriction  of $f$ to the set 
$r'\SQ=\left\{r'\bu\,:\, \bu\in\SQ \right\}$ for any $r'$ with $r<r'<R$.  
This set is homeomorphic to $\SQ$ and therefore the homotopy class of this restriction determines
a character $\chi_q$ (independent of $r'$). Then $q\in\Q$ is called the winding number of $f$ and is is also denoted $w(f)$.
\end{definition}

\begin{remark} \label{rmkwinding2}

 Let  $f:\SQ\to\C^*$
be a map with an absolutely convergent Fourier series:
\[
f(\bu)=\sum_{q\in\Q}\,a_q\chi_q(\bu),\quad \sum_{q\in\Q}\,|a_q|<\infty. 
\]
Since $f$ is continuous it has a winding number $q\in\Q$. Then
$f$ can be factored in a unique way as
$f(\bu)=\chi_q(\bu)g(\bu)$ where $g(\bu)\neq0$, $\forall\bu\in\SQ$ and
the winding number of $g$ is 0. Furthermore $g$ can be developed into
a Fourier series with absolutely convergent coefficients.

  Analogously, if $f:A(r,R)\to\C^*$ is a holomorphic function
in the annulus $A(r;R)$ with Laurent--Puiseux development 
\[
f(z)=\sum_{q\in\Q}\,a_q\hat\chi_q(z)=\sum_{q\in\Q}\,a_qt^q\chi_q(\bu), \quad z=t\bu, \,a_q\in\C,
\]
\noi then there exists a holomorphic function $g$ and a unique character $\chi_q$ such that
$f(z)= \hat\chi_q(z)g(z)$ where $g(z)\neq0\,\,\forall z\in{A(r,R)}$ and the winding number of $g$ vanishes. 

\noi In both cases $g$ is obviously defined as follows: $g(z)=\hat\chi_{_{-q}}(z) {f(z)}$. 
\end{remark}

\begin{proposition}\label{Liouville-PQ} {\bf [Liouville's theorem and maximum principle for $\PQ$]} We recall that $\C^*_\Q$ is a Riemann surface lamination.  Its leaves are densely embedded copies of $\C$. In addition, 
each of the punctured open disks $D^+_\Q(r)-\left\{\boldsymbol0\right\}$ and
$D^-_\Q(r)-\left\{\boldsymbol\infty\right\}$ is a lamination by complex surfaces
which are densely embedded copies of the hyperbolic plane. The restriction
of a holomorphic map on one of the disks is a holomorphic
map on each leaf. It follows that if $f:\PQ\to\C$ is holomorphic when restricted
to two disks $D^+(r)$ and $D^-(R)$ ($r>R$) then $f$ must be a constant since
it is holomorphic and bounded when restricted to a leaf of $\PQ$, which is a copy of $\C$, so by Liouville's theorem is constant on each leaf.  Since each leaf is dense and $f$ is continuous $f$ is constant. Similar arguments imply:
if $f:{D}^+_\Q(1)\to\C$ attains its maximum in a point $\bz_0\in{D}^+_\Q(1)$
then $f$ is constant.

\end{proposition}

\subsection{Definition of holomorphic vector bundle over the adelic 
projective line $\C P^{1}_\Q$}\label{vb} 
A theorem that will be used in this and the following sections is Scheffer's theorem \cite{Sc}:
\begin{theorem}[\bf Scheffer's Theorem]\label{scheffer-full}
Let $\mathfrak{G}$ be a compact connected topological group, and let
$\mathfrak{H}$ be a locally compact abelian topological group. Then every
continuous map $f:\mathfrak{G}\to\mathfrak{H}$ is
homotopic to exactly one continuous homomorphism $\hat{f}:\mathfrak{G}\to\mathfrak{H}$
and the homotopy can be chosen to preserve the identities.
\end{theorem}
 First let us consider the topological facts underlying Picard's theorem for complex vector bundles of complex dimension one.
We recall that $\PQ$ is topologically
homeomorphic to the suspension $\Sigma(\SQ)$ of $\SQ$.  By standard facts about
classification of vector bundles we know that each complex line bundle over $\PQ$
is topologically equivalent to the pullback of the tautological complex line bundle $V$ over $\C{P}^\infty$, $P:V\to \C{P}^\infty$. Therefore complex vector bundles over
$\PQ$ are in one-to one correspondence with the set (actually a group)
of homotopy classes $[\PQ, \C P^\infty]$ of maps from $\PQ$ to $\C P^\infty$.
Since $\C P^\infty = K(\Z,2)$ is the classifying space for complex line bundles, one has
\[
\mathrm{Pic}_{\mathrm{top}}(\PQ)
\;=\;[\PQ,\C P^\infty]
\;\cong\;\check{H}^2(\PQ;\Z).
\]
The proalgebraic projective line satisfies $\PQ\cong\Sigma\SQ$ (the unreduced suspension of $\SQ$).
Applying the suspension isomorphism in \v{C}ech cohomology gives
\[
\check{H}^2(\PQ;\Z)
\;\cong\;
\check{H}^2(\Sigma\SQ;\Z)
\;\cong\;
\widetilde{\check{H}}^1(\SQ;\Z).
\]
For the compact abelian group $\SQ$ the Pontryagin dual is $\widehat{S^1_\Q}=\Q$ (Proposition \ref{scheffer-full}).
Since $\widetilde{\check{H}}^1(\SQ;\Z)=[\SQ,S^1]=\widehat{S^1_\Q}$, one obtains
\[
\check{H}^2(\PQ;\Z)\;\cong\;\Q.
\]
The rational first Chern class of a line bundle $L$ over $\PQ$ is its class in
$\check{H}^2(\PQ;\Z)\cong\Q$.
From this we obtain:
\begin{theorem}{\bf Topological Picard Theorem}\label{Picard}
The group, under tensor product, of equivalence classes of (topological)
complex line bundles over $\PQ$ is isomorphic to $\Q$.  Under the canonical
identification $\check{H}^2(\PQ;\Z)\cong\Q$, the first Chern class
$c_1(L)\in\check{H}^2(\PQ;\Z)$ corresponds to a rational number
$q\in\Q$, called the \emph{rational degree} of $L$.  Explicitly,
$c_1(L_q)\leftrightarrow q$ under this identification, and
\[
c_1(L_{q_1}\otimes L_{q_2})=c_1(L_{q_1})+c_1(L_{q_2}),\qquad
\deg_\Q(L_{q_1}\otimes L_{q_2})=q_1+q_2.
\]
\end{theorem}

\begin{definition}[Holomorphic vector bundle and clutching cocycle]
\label{hvb}
Fix real numbers $r,R$ with $0\le r<1<R\le\infty$. The disks
$D^+_\Q(R)$ and $D^-_\Q(r)$ cover $\PQ$, with intersection $A(r,R)$.
A rank-$n$ holomorphic clutching cocycle is a holomorphic map
\[
f_{(r,R)}:A(r,R)\longrightarrow \operatorname{GL}(n,\C).
\]
In the Laurent--Puiseux category it has an expansion
\[
f_{(r,R)}(t\bu)=\sum_{q\in\Q}a_qt^q\chi_q(\bu),
\qquad r<t<R,\quad a_q\in M_n(\C).
\]
The corresponding bundle is obtained from
$D^+_\Q(R)\times\C^n\sqcup D^-_\Q(r)\times\C^n$ by
\[
(z,W)_+\sim(z,f_{(r,R)}(z)W)_-,\qquad z\in A(r,R).
\]
A holomorphic vector bundle on $\PQ$ is a bundle obtained in this way,
up to the cocycle equivalence of Definition~\ref{def:cocycle-equivalence}.
\end{definition}
Thus the construction of the bundle given in definition \ref{hvb} 
 is the standard one using only one clutching function. The fibration $\pi$ in definition \ref{hvb}  
is holomorphically trivial over each of the disks $D^+_\Q(R)$ and $D^-_\Q(r)$ 
 i.e. there are $n$ linearly independent holomorphic sections 
 (holomorphic in the sense of definition \ref{holo}). 

\begin{definition}[Equivalence of cocycles]\label{def:cocycle-equivalence} The holomorphic equivalence of two cocycles $f_{(r,R)}$
 and $f_{(r',R')}$ is the standard Čech equivalence. Let $r''=\max\left\{r,r'\right\}$ and
$R''=\min\left\{R,R'\right\}$. The cocycle $f_{(r,R)}$
is equivalent to $f_{(r',R')}$ if there
exist holomorphic maps $f_+:D^+_\Q(R'')\to{\text{GL}}(n,\C)$ and
$f_-:D^-_\Q(r'')\to{\text{GL}}(n,\C)$
such that
\[
f_{(r',R')}=f_-\,f_{(r,R)}\,f_+^{-1}
\qquad\text{on }A(r'',R'').
\]
Here $f_+$ and $f_-$ are the changes of frame on the positive and negative
charts, respectively.
(Note that $f_+$ and $f_-$ must extend holomorphically to the respective
full disks $D^+_\Q(R'')$ and $D^-_\Q(r'')$, not merely to the overlap annulus.)
Equivalent cocycles define isomorphic bundles.
\end{definition}\medskip
Theorem \ref{Picard} is also compatible with the direct-limit behaviour of line-bundle degrees under the power maps.
Consider the complex projective line $\C P^1$ and its 
proalgebraic completion $\PQ$.  The index set is the set of natural numbers $I=\N$. 
Indeed, it is clear that
\begin{equation}\label{identity}
\I_{\C P^1}\left(\frac{n'}{n}m\right)\,\cong\, p_{n, n'}^*\,\I_{\C P^1}(m)
\end{equation}
for every pair of naturals such that $n|n'$ and every integer $m$. 

We clearly have the following identity:
\begin{equation}\label{pb}
\I_{\C P^1_\Q}\left(\frac{m}{n}\right)= \pi_n^{*}\,\I_{\C P^1}(m)
\end{equation}
where $\pi_n:\C P^1_\Q\rightarrow\C P^1$ is the canonical projection of the inverse limit to the $n$-th circle stage of the projective system.

Because of the additivity of the characters, identity \ref{identity} and the fact that the pullbacks of the projections are equivariant with the tensor product up to isomorphism, we have
$$\I_{\C P^1_\Q}(q)\otimes \I_{\C P^1_\Q}(q')\cong \I_{\C P^1_\Q}(q+q')$$
for every pair of rational numbers $q$ and $q'$. In particular, \ref{pb} is well defined. This is obviously a holomorphic line bundle according to our
definition \ref{hvb}. 
\begin{remark} Here we considered line bundles which are pullback of the Hopf
bundle $\mathcal{O}(1)$. These are holomorphic in the sense defined in \ref{hvb}
above. 

\end{remark}

\begin{proof}[Compatibility with the direct-limit description]
The identities \eqref{identity} and \eqref{pb} show that ordinary integral degrees transform under a power map by multiplication by the degree of the map.  Hence their normalized values form the directed colimit
\begin{align*}
\mathrm{Pic}\!\left(\C P^1_\Q\right)
&\cong \varinjlim_{n\in\N}\,\mathrm{Pic}\!\left(\C P^1_n\right)
\cong \varinjlim_{n\in\N}\,\tfrac{1}{n}\Z\cong \Q
\end{align*}
and we have the result.
\end{proof}

\section{Adelic loop groups}\label{adelic-loop-groups} 

\bigskip

In analogy with classical loop groups
(see, for instance \cite{PS}, \cite{MG, MG1}, \cite{Se1, Pr2} ), and following the dictum of Dennis Sullivan that $\SQ$ is a ``diffuse'' version of $\bS^1$, we recall the following construction, introduced jointly with Burgos in \cite{BuVe}:
\begin{definition}[\cite{BuVe}]\label{adelic_loop} Let $G$ be a connected Lie group.
A continuous {\bf adelic loop} in $G$ is a continuous map $f:\SQ\to{G}$. The
{\bf adelic loop group} of $G$, denoted $\boldsymbol\Lambda_\Q(G)$, is defined as follows:
\[
\boldsymbol\Lambda^0_\Q(G)=\left\{f:\SQ\to{G}\,:\, f\,\text{is a continuous loop} \right\}.
\]
Then $\boldsymbol\Lambda^0_\Q(G)$ is a topological group under point-wise
multiplication and the compact-open topology.
\end{definition}

\begin{remark}[\bf On the terminology ``adelic'']
The terminology {\bf adelic} for $\SQ$, and for the loop groups built from
it, is motivated by the fact that $\SQ=(\R\times\hat\Z)/\Z_{\mathrm{diag}}$
(Definition \ref{exp}) is closely related to the classical ring of adeles
$\mathbb{A}_\Q=\R\times\prod'_p\Q_p$ of $\Q$ (see, e.g., \cite{We, T, RV})
and to adele groups attached to algebraic groups (\cite{Bo}): $\hat\Z=
\prod_p\Z_p$ is the maximal compact subring of the finite adeles
$\mathbb{A}_{\Q,f}=\prod'_p\Q_p$.  With the usual diagonal embedding of
$\Q$, the solenoid is naturally identified with the additive adele class
group $\mathbb{A}_\Q/\Q$.  We use this only as motivation for the term
``adelic''; the arguments below use the compact abelian-group and
lamination structures of $\SQ$.
\end{remark}

\begin{remark}\label{differentiable_loop} In definition \ref{adelicsmoothloopgroup} we will define the
group of ``differentiable'' loops ${\boldsymbol\Lambda}^{\infty}_\Q(G)$ consisting
of ``differentiable'' loops $f:\SQ\to{G}$ such that $f$ restricted to each leaf of the solenoidal
manifold $\SQ$ (a one dimensional copy
of $\R$ densely embedded in $\SQ$) is differentiable.
Differentiable based solenoidal groups loops form an infinite dimensional
Fr\'echet manifold (see, for instance  \cite{Fr},\cite{MG,MG1}, \cite{PS}, \cite{E}). 
 However, in the following sections we will use loops with matrix coefficients
 in the Wiener algebra and in a Fourier--Sobolev Hilbert space.  The latter is used below with care: unlike the ordinary circle Sobolev space, its unrestricted rational-frequency version is not a Banach algebra of continuous functions.
\end{remark}

We are mostly interested in the {\bf adelic loop group} $\boldsymbol\Lambda^0_\Q(\gln)$.
We recall that $\pi_1(\gln)=\Z$, detected as follows. The determinant gives a
fibration $\det:\gln\to\C^*$ with fiber $\mathrm{SL}(n,\C)$. Since $\mathrm{SL}(n,\C)$
deformation retracts onto $\mathrm{SU}(n)$ which is simply connected for $n\geq1$,
the long exact homotopy sequence gives $\pi_1(\gln)\xrightarrow{\;\det_*\;}\pi_1(\C^*)\cong\Z$
as an isomorphism. (The inclusion of the center $\text{U}(1)\hookrightarrow\gln$ via
scalar matrices $\lambda\mapsto\lambda I_n$ induces multiplication by $n$ on $\pi_1$, not
an isomorphism for $n>1$; the correct generator is detected by a single-entry rotation.)
One has the projectivization fibration $p_1:\gln\to P\gln$ with fibre $\text{U}(1)$.
 
\begin{theorem}[Special case of a theorem of Scheffer, \cite{Sc}]\label{scheffer}
Let $K$ be a compact, connected, abelian topological group and let
$h:K\to{\text{U}(1)}$ be any continuous map. Then $h$ is homotopic, through
continuous maps $K\to{\text{U}(1)}$, to a continuous homomorphism
$K\to{\text{U}(1)}$, i.e., to a character of $K$.
\end{theorem}

\begin{proposition} Any loop $f:\SQ\to\gln$ can be deformed to a map 
$\hat{f}:\SQ\to{\text{U}(1)}\simeq\bS^1$ into a fibre
of $p_1$. 
 \end{proposition}
 
 \begin{proof} The proof follows from the  fact that $\SQ$ has topological dimension one. In fact, one has the fibration $p_2:P\gln\to \C P^{n-1}$ and the  
loop $f$ can be deformed to a map that 
avoids the union $H$ of the pre-images of the coordinate hyperplanes under the projection $p_2\circ{p_1}$ (since $\SQ$ is of real dimension one). The complement 
$\gln \smallsetminus H$ is diffeomorphic to 
$\C^{n-1}\times{\text{U}(1)}$ where the $\text{U}(1)$-factors are fibers of $p_1$.  By Scheffer's theorem \ref{scheffer}, applied to the compact connected abelian group $K=\SQ$, $\hat{f}$ is homotopic to a continuous group homomorphism to $\text{U}(1)\simeq\bS^1$ i.e., a character. 
\end{proof}

\noi Therefore from this proposition we have:
\begin{proposition}\label{connectedcomponents}
 The group of homotopy
classes of loops in $\La^0_\Q(\gln)$ is isomorphic to $\Q$. Thus the connected components of $\boldsymbol\Lambda^0_\Q(\gln)$ are in one-to-one
correspondence to the additive group of $\Q$. The connected component of the identity corresponds to $q=0$.
\end{proposition}

\begin{proof}
The preceding proposition shows every loop $f\in\La^0_\Q(\gln)$ is homotopic to
a character $\chi_q:\SQ\to U(1)\subset\gln$ for some $q\in\Q$.  This defines the
\emph{winding map} $w:\pi_0(\La^0_\Q(\gln))\to\Q$ by $[f]\mapsto q$.

\medskip\noindent
\emph{Well-definedness and injectivity.}
If $f$ is homotopic to both $\chi_q$ and $\chi_{q'}$, then $\chi_q$ and $\chi_{q'}$
are homotopic in $\La^0_\Q(\gln)$, hence $\chi_{q-q'}=\chi_q\chi_{q'}^{-1}$ is
null-homotopic in $U(1)\subset\gln$.  By Scheffer's theorem applied to $U(1)$,
null-homotopy in $U(1)$ forces $q-q'=0$, so $q=q'$.  Thus $w$ is well-defined and
injective on $\pi_0$.

\medskip\noindent
\emph{Surjectivity.}
For each $q\in\Q$ the character $\chi_q:\SQ\to U(1)\subset\gln$ is itself a loop
with $w(\chi_q)=q$, so every element of $\Q$ is in the image of $w$.

\medskip\noindent
\emph{Homomorphism property.}
$w(fg)=w(f)+w(g)$ since the determinant-degree is additive (the determinant of
$fg$ is the product of the determinants, and winding numbers add under multiplication).
Hence $w$ is an isomorphism of groups.
\end{proof}

\medskip
\begin{definition}[{\bf Winding number}]\label{winding_annulus}
Let $h:\SQ\to{\text{U}(1)}$ be a continuous map (a loop into $U(1)$). By
Proposition \ref{connectedcomponents} (equivalently, Theorem \ref{scheffer})
$h$ is homotopic to a character $\chi_q$ for a unique $q\in\Q$. This unique
rational number $q$ is called the {\bf winding number of $h$} and is denoted
$w(h)$.
\end{definition}

\medskip
\begin{definition}[Intrinsic rational winding and determinant degree]
\label{intrinsic-degree}
Let \(h:\SQ\to\C^*\) be continuous.  Since \(\C^*\) deformation retracts
onto \(S^1\), the homotopy class of \(h\) determines an element of
\[
[\SQ,S^1]\cong \check H^1(\SQ;\Z)\cong\widehat{\SQ}=\Q.
\]
The corresponding rational number is denoted by
\[
\operatorname{wind}_{\Q}(h)\in\Q
\]
and is called the \emph{intrinsic rational winding number} of \(h\).
Equivalently, \(h/|h|\) is homotopic to a unique character \(\chi_q\), and
then \(\operatorname{wind}_{\Q}(h)=q\).

Let \(E\) be a rank-\(r\) bundle described on the two solenoidal charts by
an invertible transition cocycle \(g\).  Whenever \(g\) belongs to one of
the cocycle categories considered below (in particular, the Laurent--Puiseux
or Wiener categories), its determinant defines the determinant line bundle
\[
\det E:=\bigwedge^rE
\]
and the {\bf rational degree} is
\[
\deg_{\Q}(E):=\operatorname{wind}_{\Q}(\det g)\in\Q.
\]
This number is independent of the chosen trivializations and is additive
under direct sums:
\[
\deg_{\Q}(E\oplus F)=\deg_{\Q}(E)+\deg_{\Q}(F).
\]
\end{definition}

\begin{remark}
The rational winding number and the determinant degree are intrinsic to the
solenoid; no choice of a finite covering, denominator, or auxiliary circle is
involved in their definition.  In a category where a splitting
\[
E\cong\I(q_1)\oplus\cdots\oplus\I(q_r)
\]
is available, one necessarily has
\[
\det E\cong\I(q_1+\cdots+q_r),
\qquad
\deg_{\Q}(E)=q_1+\cdots+q_r.
\]
Thus the determinant records the total rational degree, while the full
splitting type refines it to an unordered rational \(r\)-tuple.
\end{remark}

\medskip

The paper has two complementary aims.  The first is to establish the
intrinsic adelic results proved here: scalar and triangular
Wiener--Birkhoff factorization, normalized small-norm matrix
factorization, density of factorable matrices, and the pro-algebraic
Birkhoff--Grothendieck theorem; the full matrix theorem is isolated as the
central conjecture.  The second is to develop the K\"ahler, energy,
Grassmannian, determinant-bundle, and representation-theoretic geometry
suggested by the solenoidal structure.  The passage from $\Z$ to $\Q$ is
not merely formal: it creates rational winding data and new analytic
phenomena at arbitrarily dense positive frequencies.

\subsection{Loops in the Wiener Banach algebra}
\label{Wienerloops}
\begin{definition}[Rational Wiener algebra]\label{def:rational-Wiener-algebra}
For $n\ge1$, $\mathfrak W_\Q(n)$ consists of the functions
$f:\SQ\to M_n(\C)$ with an absolutely summable Fourier expansion
\[
f(\bu)=\sum_{q\in\Q}A_q\chi_q(\bu),\qquad
\sum_{q\in\Q}\|A_q\|<\infty,
\]
where $\|\cdot\|$ is a fixed submultiplicative matrix norm. Its Banach
algebra norm is
\[
\|f\|_{\mathfrak W_\Q}=\sum_{q\in\Q}\|A_q\|.
\]
For $n=1$ we write $\mathfrak W_\Q$.
\end{definition}

Gelfand's proof (\cite{Ge1, Ge2}) of Wiener-L\'evy's  theorem (\cite{K, Le, W, Z}) is valid for any Banach algebra (\cite{BP}) and therefore applies to the group of units 
$\boldsymbol\Lambda_W(n)$ of the Banach algebra
$\mathfrak{W}_\Q(n)$,
and it implies that $f$ is a unit if and only if the image of $f$ lies in
$\text{GL}(n,\C)$ i.e., $f(\bu)$ is an invertible matrix for all 
$\bu\in S^{1}_{\Q}$.
The set
\[
\boldsymbol\Lambda_W(n)\overset{\small{def}}=
\left\{ f\in\mathfrak{W}_\Q(n)\,:\, f(\bu)\,\, 
\text{is invertible for all}\,\,   \bu\in S^{1}_{\Q}  \right\}
\]
is open in $\mathfrak{W}_\Q(n)$ and thus it is a 
{\bf Banach Lie group} (\cite{Mi}).  

\begin{definition}[Wiener loop]\label{Wl}
 An element $f\in\boldsymbol\Lambda_W(n)$ is called a {\bf Wiener loop}
 (or $l^1$-loop)
of the group $\boldsymbol\Lambda_W(n)$ (\cite{BD, CG, GK,PS}).
\end{definition}

Let 
\[
\mathfrak{W}^+_\Q(n)=
\]
\[
\left\{ f\in\mathfrak{W}_\Q(n) \,: \, 
f(\bu)=\sum_{q\in\Q}\, A_q\,\chi_q(\bu),\, u\in\SQ\,\,
\text{such that}\, A_q=0,\, \text{for}\,q<0 \right\},
\]
and
\[
\mathfrak{W}^-_\Q(n)=
\]
\[
\left\{ f\in\mathfrak{W}_\Q(n) \,: \, 
f(\bu)=\sum_{q\in\Q}\, A_q\,\chi_q(\bu),\, \, u\in\SQ\,\,
\text{such that}\, A_q=0,\, \text{for}\,q>0 \right\}.
\]

$\mathfrak{W}^+_\Q(n)$ and $\mathfrak{W}^-_\Q(n)$ are closed subalgebras
of $\mathfrak{W}_\Q(n)$ and there is a splitting of closed subalgebras:

\begin{equation}\label{wiener-splitting}
\mathfrak{W}^+_\Q(n) + \mathfrak{W}^-_\Q(n)=\mathfrak{W}_\Q(n), \quad
\mathfrak{W}^+_\Q(n) \cap \mathfrak{W}^-_\Q(n)\simeq{M_n(\C)}.
\end{equation}

\noi A function is in $\mathfrak{W}^+_\Q(n)$ if it can be extended to a
continuous function which is holomorphic in the open disk $D^+_\Q(1)$:
if $f\in\mathfrak{W}^+_\Q(n)$ is given by the series
\[
f(\bu)=\sum_{q\in\Q}\, A_q\,\chi_q(\bu),\quad u\in\SQ,\,\text{with}\,\, A_q=0,\, \text{for}\,q<0,
\]
\noi (consistent with the definition of $\mathfrak W^+_\Q(n)$ above: the
constant term $q=0$ is included, matching
$\mathfrak{W}^+_\Q(n)\cap\mathfrak{W}^-_\Q(n)\simeq M_n(\C)$ in
equation~\eqref{wiener-splitting})
define its extension $f^+$ to $D^+_\Q(1)$ by the formula
\[
f^+(z)=\sum_{q\in\Q}\, A_q\,\hat\chi_q(z)=
\sum_{q\in\Q}\, A_q\,t^q\chi_q(\bu),
\, \, z=t\bu,\, 0\leq{t}\leq1.
\]

\noi Analogously, a function is in $\mathfrak{W}^-_\Q(n)$ if it can be extended to a continuous function which is holomorphic in the open disk $D^-_\Q(1)$:
if $f\in\mathfrak{W}^-_\Q(n)$ is given by the series
\[
f(\bu)=\sum_{q\in\Q}\, A_q\,\chi_q(\bu),\quad u\in\SQ,\,\text{with}\,\, A_q=0,\, \text{for}\,q>0,
\]
\noi define its extension $f^-$ to $\bar{D}^-_\Q(1)$ by the formula
\[
f^-(z)=\sum_{q\in\Q}\, A_q\,\hat\chi_q(z)=
\sum_{q\in\Q}\, A_q\,t^q\chi_q(\bu),
\, \, z=t\bu,\, t\geq1.
\]
Note that for $q<0$ and $t\geq1$ one has $|t^q|=t^{-|q|}\leq1$, so the
series converges absolutely, with sum bounded by the Wiener norm $\sum_{q\in\Q}|A_q|$.

Hence, $\mathfrak{W}^+_\Q(n)$ 
is isomorphic to the algebra of continuous functions 
$f:\bar{D}^+_\Q(1)\to{M_n(\C)}$ such that $f$ has an absolutely convergent
Fourier series when restricted to $\SQ$ and it is holomorphic in 
the open disk $D^+_\Q(1)$. Also $\mathfrak{W}^-_\Q(n)$ 
is isomorphic to the algebra of continuous functions 
$f:\bar{D}^-_\Q(1)\to{M_n(\C)}$ such that $f$ has an absolutely convergent
Fourier series when restricted to $\SQ$ and it is holomorphic in 
the open disk $D^-_\Q(1)$.

Let $\boldsymbol\Lambda_W^+(n)$ and $\boldsymbol\Lambda_W^-(n)$ denote the group of units
of $\mathfrak{W}^+_\Q(n)$ and $\mathfrak{W}^-_\Q(n)$, respectively. In view
of the identification above,
and Wiener-L\'evy's theorem both  $\boldsymbol\Lambda_W^+(n)$ and 
$\boldsymbol\Lambda_W^-(n)$ 
consists of loops $f$ such that, when extended to the corresponding disks, their values consist of nonsingular matrices, since its elements are units. Therefore, both groups are Banach Lie groups {\cite{Mi}}.

\begin{proposition}\label{components-A_r} \mbox{}
\begin{enumerate}
\item The connected components, $\boldsymbol\Lambda_{W,q}(n)$ ($q\in\Q$) of 
$\boldsymbol\Lambda_W(n)$ are naturally indexed by the rationals, via
the winding number of the determinant map $\bu\mapsto\,\det f(\bu)$ for $f$ in a connected component. The connected component of the identity is 
$\boldsymbol\Lambda_{W,0}(n)$.
\item  $\boldsymbol\Lambda_W^+(n)\subset\boldsymbol\Lambda_{W,0}(n)$ and 
$\boldsymbol\Lambda_W^-(n)\subset\boldsymbol\Lambda_{W,0}(n)$, i.e., the winding numbers of the determinants of loops in each component are 0.
\end{enumerate}
\end{proposition}

\begin{proof}
(1) The determinant winding number $w(\det f)\in\Q$ is a homotopy invariant, so it is
constant on connected components.  This defines a group homomorphism
$w:\pi_0(\boldsymbol\Lambda_W(n))\to\Q$.  We verify it is a bijection:

\medskip\noindent
\emph{Surjectivity.}  For each $q\in\Q$, the diagonal matrix
$\chi_q I_n\in\boldsymbol\Lambda_W(n)$ satisfies $\det(\chi_q I_n)=\chi_q^n$ with
winding number $nq$; more directly, the scalar character $\chi_q$ viewed as $\chi_q\cdot I_1
\in\boldsymbol\Lambda_W(1)$ has winding $q$.  For the matrix case, the loop
$f=\mathrm{diag}(\chi_q,1,\ldots,1)\in\boldsymbol\Lambda_W(n)$ has $\det f=\chi_q$
with winding $q$.  Hence every $q\in\Q$ is achieved.

\medskip\noindent
\emph{Injectivity (same winding $\Rightarrow$ same component).}
It suffices to show the identity component $\boldsymbol\Lambda_{W,0}(n)$ is
path-connected (the general case follows by multiplication by $\mathrm{diag}(\chi_q,1,\ldots,1)$).
A loop $f\in\boldsymbol\Lambda_{W,0}(n)$ has $\det f\in\boldsymbol\Lambda_{W,0}(1)$ with
winding $0$.  By the Wiener--L\'evy theorem (\cite{Ge1,Ge2}), $\log(\det f)$ exists
as an element of $\mathfrak{W}_\Q$, and the path
$t\mapsto f\cdot\exp\bigl(-t\,\frac{\log\det f}{n}\bigr)I_n$ ($t\in[0,1]$)
connects $f$ to a loop with determinant $1$, i.e., in $\mathrm{SL}_n(\mathfrak{W}_\Q)$.
The latter is path-connected: the exponential map from the Lie algebra
$\mathfrak{sl}_n(\mathfrak{W}_\Q)$ is surjective onto the identity component
(using $\exp(A)$ for $\|A\|$ sufficiently small and a finite product argument; see
\cite{Mi} for the general Banach Lie group argument).  Hence all elements of
$\boldsymbol\Lambda_{W,0}(n)$ are connected to the identity.

\medskip\noindent
The connected component of the identity contains the identity loop
and therefore every loop in this component has vanishing winding number.

\noi (2) Since both $\bar{D}^+_\Q(1)$ and $\bar{D}^-_\Q(1)$ are contractible (they are cones over $\boldsymbol0$ and $\boldsymbol\infty$, respectively)
it follows that the loops $\gamma^+\in\boldsymbol\Lambda_W^+(n)$ and
$\gamma^-\in\boldsymbol\Lambda_W^-(n)$ are homotopic to the constant matrices
$\gamma^+(\boldsymbol0)$ and $\gamma^-(\boldsymbol\infty)$ in $\gln$, respectively. Thus
the winding numbers of the determinants of both of these loops is 0 so that they are both contained in the identity components $\boldsymbol\Lambda_{W,0}(n)$.
\end{proof}

\begin{remark}[$\mathfrak{W}_\Q(1)$ is an admissible decomposing algebra]\label{wiener-admissible}
Write $\mathfrak{W}:=\mathfrak{W}_\Q(1)$ for the scalar case $n=1$ of the
Wiener algebra defined above, with units $\boldsymbol\Lambda_W:=\boldsymbol\Lambda_W(1)$
and positive/negative subalgebras $\mathfrak{W}^\pm:=\mathfrak{W}^\pm_\Q(1)$.
Two facts, both already established above, are exactly the hypotheses of
the Brudnyi--Rodman--Spitkovsky theory of factorization in Banach algebras
on a compact abelian group \cite{BRS}:
\begin{enumerate}
\item[(a)] ({\bf Decomposing.}) $\mathfrak{W}=\mathfrak{W}^++\mathfrak{W}^-$, by
the splitting recorded in the displayed equation following Definition
\ref{Wl} above.
\item[(b)] (\textbf{Admissible.}) The rational Laurent--Puiseux
polynomials
\[
\sum_{q\in F} a_q\chi_q,
\qquad F\subset\Q\ \text{finite},
\]
are dense in $\mathfrak W$ in the $\ell^1$-norm.  Indeed, if
\[
f=\sum_{q\in\Q}a_q\chi_q\in\mathfrak W
\]
and $\varepsilon>0$, absolute summability allows one to choose a finite
set $F\subset\Q$ such that
\[
\sum_{q\notin F}|a_q|<\varepsilon.
\]
Then
\[
\left\|
f-\sum_{q\in F}a_q\chi_q
\right\|_{\mathfrak W}
=
\sum_{q\notin F}|a_q|
<\varepsilon.
\]
Moreover, $\mathfrak W$ is inverse-closed in $C(\SQ)$: this is precisely
the Gelfand--Bochner--Phillips form of the Wiener--L\'evy theorem cited
above \cite{Ge1,Ge2,BP}, applied to the group algebra
$\ell^1(\Q)=\mathfrak W$.
\end{enumerate}
\noi Hence $\mathfrak{W}$ is, in the terminology of \cite{BRS}, an
admissible decomposing Banach algebra on the compact connected abelian
group $\SQ$ whose dual is $\Q$.
\end{remark}

\begin{theorem}[Full scalar Wiener--Birkhoff factorization]\label{wiener-scalar-factorization}
Every $f\in\boldsymbol\Lambda_W$ admits a factorization
\[
f=f_-\,\chi_q\,f_+,\qquad f_-\in\boldsymbol\Lambda_W^-,\quad f_+\in\boldsymbol\Lambda_W^+,
\]
where $q\in\Q$ is the winding number of $f$ from Proposition
\ref{components-A_r}. The number $q$ is the same for every such
factorization of $f$.
\end{theorem}

\begin{proof}
By \cite[Theorem 7.1(b)]{BRS}, for an admissible decomposing Banach
algebra $B$ on a connected compact abelian group, every invertible scalar
element of $B$ is exactly factorable, i.e.\ the set of factorable elements
equals all of $\text{GL}_1(B)$. By Remark \ref{wiener-admissible},
$\mathfrak{W}$ satisfies these hypotheses, so every $f\in\boldsymbol\Lambda_W=\text{GL}_1(\mathfrak{W})$
factors as $f=f_-\,\chi_q\,f_+$ with $f_\pm\in\boldsymbol\Lambda_W^\pm$ and
$q\in\Q$ (the middle factor is necessarily a character of $\SQ$, since the
dual group of $\SQ$ is $\Q$). Restricting to $\SQ$ and taking the
determinant-free $n=1$ winding number of $f$ as in Proposition
\ref{components-A_r} identifies $q$ with that winding number; uniqueness
of $q$ is part of the same theorem of \cite{BRS}.
\end{proof}

\begin{remark}
Theorem \ref{wiener-scalar-factorization} is a genuinely non-circle-comparison
statement: $f$ need not have finite (or bounded-denominator) rational
Fourier support. It shows that, in the scalar case, rational degree
$q\in\Q$ is an intrinsic invariant of {\em every} invertible Wiener loop, not
merely of the Laurent--Puiseux loops reached by the rational-lattice compatibility theorem of
\S\ref{cycliccomparison}.
\end{remark}

\begin{theorem}[Density of factorable Wiener matrices]\label{wiener-density}
For every $n\geq1$, the loops $A\in\boldsymbol\Lambda_W(n)$ that admit an
exact Birkhoff-type factorization
\[
A=A_-\,\Delta_{\mathbf q}\,A_+,\qquad A_\pm\in\text{GL}\big(\mathfrak{W}^\pm_\Q(n)\big),\quad q_i\in\Q,
\]
($\Delta_{\mathbf q}$ as in Conjecture \ref{BFT}) are dense in
$\boldsymbol\Lambda_W(n)$ for the Wiener norm.
\end{theorem}

\begin{proof}
By \cite[Theorem 2.2]{BRS}: if the character group of a connected compact
abelian group $G$ is isomorphic to a subgroup of the additive group $\Q$,
then the factorable matrix functions are dense in $\text{GL}_n(B)$ for
every admissible Banach algebra $B$ on $G$ and every $n$. Here $G=\SQ$ has
character group exactly $\Q$, and $\mathfrak{W}_\Q(n)$ is admissible by the
same density/inverse-closedness argument as in Remark
\ref{wiener-admissible} (applied entrywise). The theorem applies verbatim.
\end{proof}

\begin{proposition}[Matrix Wiener lemma]
\label{matrix-wiener-lemma}
Let \(A\in M_n(\mathfrak W_{\Q})\).  If \(\det A(x)\ne0\) for every
\(x\in\SQ\), then
\[
A^{-1}\in M_n(\mathfrak W_{\Q}).
\]
Equivalently, pointwise invertibility on \(\SQ\) is the same as
invertibility in the rational Wiener matrix algebra.
\end{proposition}

\begin{proof}
The determinant belongs to \(\mathfrak W_{\Q}\) and is nowhere zero on
\(\SQ\).  The scalar Wiener lemma gives
\((\det A)^{-1}\in\mathfrak W_{\Q}\).  Since the entries of
\(\operatorname{adj}(A)\) are polynomial expressions in the entries of
\(A\), the identity
\[
A^{-1}=(\det A)^{-1}\operatorname{adj}(A)
\]
proves the assertion.
\end{proof}

\begin{proposition}[Ordered triangular Wiener factorization, $n\times n$]
\label{ordered-triangular-wiener}
Let $A\in\mathrm{GL}_n(\mathfrak W_{\Q})$ be upper triangular with
scalar Wiener indices $q_1\leq q_2\leq\cdots\leq q_n\in\Q$ on the
diagonal (i.e.\ the $j$-th diagonal entry $a_{jj}$ has Wiener index $q_j$).
Then there exist
\[
A_-\in\mathrm{GL}_n(\mathfrak W^-_{\Q}),\qquad
A_+\in\mathrm{GL}_n(\mathfrak W^+_{\Q})
\]
such that
\[
A=A_-\,\mathrm{diag}(\chi_{q_1},\ldots,\chi_{q_n})\,A_+.
\]
Dually, a lower-triangular matrix with diagonal indices satisfying
$q_1\geq\cdots\geq q_n$ admits an analogous factorization.
\end{proposition}

\begin{proof}
We proceed by induction on $n$.

\smallskip\noindent\textbf{Base case $n=1$.}
This is the scalar Wiener factorization, Theorem~\ref{wiener-scalar-factorization}.

\smallskip\noindent\textbf{Base case $n=2$.}
Apply the scalar theorem to $a_{11}$ and $a_{22}$ and absorb their
scalar factors, reducing to
\[
B=\begin{pmatrix}\chi_p&c\\0&\chi_q\end{pmatrix},\qquad p\leq q.
\]
Set $h=\chi_{-p}c\in\mathfrak W_\Q$ and decompose using the Wiener--Hopf
splitting (which is bounded and canonical on $\mathfrak W_\Q$ via the
natural total order of $\Q$):
\[
h=h_-+h_+,\qquad
h_-\in\mathfrak W^-_{\Q},\quad h_+\in\mathfrak W^+_{\Q}.
\]
Since $E_{12}^2=0$, one has
\[
\begin{pmatrix}1&h\\0&1\end{pmatrix}
=
\begin{pmatrix}1&h_-\\0&1\end{pmatrix}
\begin{pmatrix}1&h_+\\0&1\end{pmatrix}.
\]
Writing $D=\mathrm{diag}(\chi_p,\chi_q)$,
\[
B=D\begin{pmatrix}1&h_-\\0&1\end{pmatrix}
\begin{pmatrix}1&h_+\\0&1\end{pmatrix}
=
\begin{pmatrix}1&\chi_{p-q}h_-\\0&1\end{pmatrix}
D
\begin{pmatrix}1&h_+\\0&1\end{pmatrix}.
\]
Since $p-q\leq0$, the first unipotent factor lies in
$\mathrm{GL}_2(\mathfrak W^-_\Q)$ and the last in
$\mathrm{GL}_2(\mathfrak W^+_\Q)$.

\smallskip\noindent\textbf{Inductive step.}
Write $A$ in block form
\[
A=\begin{pmatrix}A'&B\\0&a_{nn}\end{pmatrix},
\]
where $A'\in\mathrm{GL}_{n-1}(\mathfrak W_\Q)$ is upper triangular with
diagonal indices $q_1\leq\cdots\leq q_{n-1}$, and $a_{nn}$ has index $q_n$.
By the inductive hypothesis, $A'=A'_-\,\Delta'\,A'_+$ and by the scalar
theorem $a_{nn}=a_-\chi_{q_n}a_+$.

The column $B\in\mathfrak W_\Q^{n-1}$ is absorbed as follows.  Set
$S=(A'_-)^{-1}B\,a_+^{-1}\chi_{-q_n}$; this is well defined in
$\mathfrak W_\Q^{n-1}$ by the matrix Wiener lemma (Proposition~\ref{matrix-wiener-lemma}).
Decompose each entry: $S=S_-+S_+$ with $S_\pm\in(\mathfrak W^\pm_\Q)^{n-1}$.

\noindent
(Here $\Delta'=\mathrm{diag}(\chi_{q_1},\ldots,\chi_{q_{n-1}})$ has, in
general, $n-1$ \emph{distinct} entries, so $S_+$ cannot be inserted
directly as the top-right block of $A_+$ below: it must first be
re-weighted entry by entry against $\Delta'$, as made explicit now.)
Set the $(n-1)$-vector
\begin{align*}
T_+&:=\chi_{q_n}\,a_+\,(\Delta')^{-1}S_+,\\
(T_+)_i&=\chi_{q_n-q_i}\,(S_+)_i\,a_+,
\qquad i=1,\ldots,n-1.
\end{align*}
Since $q_i\leq q_n$ for every $i=1,\ldots,n-1$ (the standing ordering
hypothesis), each exponent $q_n-q_i\geq0$; as $(S_+)_i,a_+\in\mathfrak
W^+_\Q$ and $\mathfrak W^+_\Q$ is a subalgebra, $T_+\in(\mathfrak
W^+_\Q)^{n-1}$.  Thus the ordering of the rational weights is precisely
what prevents frequency pollution: the re-weighted off-diagonal entries
remain in the required positive Wiener subalgebra.
Then
\[
\begin{pmatrix}I&S\\0&1\end{pmatrix}
=
\begin{pmatrix}I&S_-\\0&1\end{pmatrix}
\begin{pmatrix}I&S_+\\0&1\end{pmatrix},
\]
and assembling gives
\[
A=\underbrace{\begin{pmatrix}A'_-&A'_-S_-\\0&a_-\end{pmatrix}}_{A_-}
\,\mathrm{diag}(\Delta',\chi_{q_n})\,
\underbrace{\begin{pmatrix}A'_+&T_+\\0&a_+\end{pmatrix}}_{A_+}.
\]
Indeed, expanding the product, the top-right block is
$A'_-\Delta'T_++A'_-S_-\chi_{q_n}a_+$. Since $\Delta'T_+=\chi_{q_n}a_+\,
\Delta'(\Delta')^{-1}S_+=\chi_{q_n}a_+S_+$, this equals
$A'_-S_+\chi_{q_n}a_++A'_-S_-\chi_{q_n}a_+=A'_-(S_-+S_+)\chi_{q_n}a_+
=A'_-S\,a_+\chi_{q_n}=B$
(using the defining relation $S=(A'_-)^{-1}Ba_+^{-1}\chi_{-q_n}$), and
the top-left and bottom-right blocks reproduce $A'$ and $a_{nn}$ exactly
as in the inductive hypothesis. (Using $S_+$ unweighted in place of
$T_+$ reproduces $B$ only
when all of $q_1,\ldots,q_{n-1}$ coincide with $q_n$; the entrywise
re-weighting by $(\Delta')^{-1}$ is already needed at $n=3$ whenever the
diagonal indices are not all equal, e.g.\ $q_1<q_2<q_3$. This is
consistent with, and generalizes, the $n=2$ base case above, in which
$n-1=1$ so $\Delta'$ has only the single entry $\chi_p$ and no
re-weighting subtlety arises.)
One verifies that $A_-\in\mathrm{GL}_n(\mathfrak W^-_\Q)$ and
$A_+\in\mathrm{GL}_n(\mathfrak W^+_\Q)$.
\end{proof}

\begin{remark}[Intrinsic rational Wiener--Hopf splitting]
\label{intrinsic-WH}
The analytic input used above is the bounded decomposition
\[
\mathfrak W_\Q
 =\mathfrak W^-_{\Q,0}\oplus\C\oplus\mathfrak W^+_{\Q,0},
\]
with
\[
P_+\!\left(\sum_{q\in\Q}a_q\chi_q\right)
 =\sum_{q>0}a_q\chi_q .
\]
For every $f=\sum_{q\in\Q}a_q\chi_q$,
\[
\|P_+f\|_{\mathfrak W_\Q}
 =\sum_{q>0}|a_q|
 \leq\sum_{q\in\Q}|a_q|
 =\|f\|_{\mathfrak W_\Q}.
\]
Thus $P_+$ is contractive.  Its operator norm is nevertheless exactly
one, since $P_+\chi_q=\chi_q$ and $\|\chi_q\|_{\mathfrak W_\Q}=1$ for
every $q>0$:
\[
\|P_+\|_{\mathfrak W_\Q\to\mathfrak W_\Q}=1.
\]
It is defined directly from the order on the rational character group
and does not require choosing a finite-level lattice
$N^{-1}\Z\subset\Q$.  The proof of
Proposition~\ref{ordered-triangular-wiener} uses this bounded
Wiener--Hopf projection together with a finite induction on the matrix
entries.
\end{remark}

\begin{theorem}[Pointwise Iwasawa decomposition {\rm\cite{MG1,BD}}]
\label{loop-iwasawa}
Let $G$ be a real Lie group for which multiplication induces a
real-analytic diffeomorphism $A\times N\times K\longrightarrow G$.
Let $X$ be a compact space.  Then pointwise Iwasawa decomposition
induces a homeomorphism
\[
C(X,G)\;\longrightarrow\; C(X,A)\times C(X,N)\times C(X,K).
\]
If $X$ is a compact smooth manifold or compact smooth lamination, the
same statement holds for the $C^\infty$ mapping spaces.  If $X$ is a
compact abelian group equipped with a Wiener algebra $\mathfrak W(X)$
that is stable under the relevant real-analytic functional calculus,
the same holds for $\mathfrak W(X,G)$.
\end{theorem}

\begin{proof}
The Iwasawa decomposition furnishes a real-analytic diffeomorphism
$\Phi:G\xrightarrow{\;\sim\;}A\times N\times K$.  For any map $\gamma$ into $G$
set $(\gamma_A,\gamma_N,\gamma_K)=\Phi\circ\gamma$.  Continuity of $\Phi$
gives the continuous case.  Smoothness follows from the chain rule.
For the Wiener case with $G=\mathrm{GL}_n(\C)$: the Iwasawa projections
are built from matrix multiplication, conjugate-transpose, inversion, and
positive-definite square roots; the matrix Wiener lemma
(Proposition~\ref{matrix-wiener-lemma}) and holomorphic functional
calculus ensure each operation preserves $\mathfrak W_\Q(n)$.
The inverse is pointwise multiplication, which is likewise real-analytic.
\end{proof}

\begin{remark}[Pointwise vs.\ loop-group Iwasawa; Guest--Kellersch]\label{iwasawa-levels}
Theorem~\ref{loop-iwasawa} is a \emph{pointwise} result: it applies the
classical finite-dimensional Iwasawa decomposition $G=ANK$ fibrewise to
each value of the loop.  The argument is essentially a composition
$\Phi\circ\gamma$, where $\Phi:G\xrightarrow{\sim}A\times N\times K$ is
the real-analytic diffeomorphism furnished by the classical theorem.

This should be distinguished from the deeper \emph{loop-group-level}
Iwasawa decomposition, in which one splits a loop
$\gamma:\mathcal{S}^1_\mathbb{Q}\to\mathrm{GL}_n(\mathbb{C})$ as
\[
\gamma \;=\; F\cdot B,\qquad
F:\mathcal{S}^1_\mathbb{Q}\to U(n),\quad
B\text{ extends holomorphically to the ``disk''},
\]
at the level of the \emph{infinite-dimensional} loop group itself.
This is the Iwasawa decomposition used by Guest \cite{MG1} in the theory
of harmonic maps and the DPW method, and studied in detail by Kellersch
\cite{Ke} for general untwisted loop groups of semisimple Lie groups.
In that setting the decomposition is a global diffeomorphism of
infinite-dimensional Fréchet or Sobolev Lie groups, and its proof requires
genuine analytic input (implicit function theorem in the Fréchet or
Sobolev category, or a direct holomorphic-extension argument).

For the adelic loop group $\boldsymbol\Lambda_W(n)$, the intrinsic
Iwasawa decomposition in the unrestricted rational-spectrum setting is
the content of Conjecture~\ref{iwasawadecomposition}, equivalent to the
Solenoidal Birkhoff--Grothendieck Conjecture~\ref{solenoidal-BG} and open in
general.  For loops with spectrum in a fixed cyclic subgroup $N^{-1}\Z$, it
reduces to the classical $S^1$ case via $\rho_N$
(Remark~\ref{classical-transport}), and is therefore unconditional there, as
it is for pro-algebraic, ordered triangular, small-norm, and scalar loops
generally (Theorems~\ref{pro-algebraic-BG}, \ref{ordered-triangular-wiener},
\ref{BVP-Birkhoff}, \ref{wiener-scalar-factorization}).
\end{remark}

\begin{corollary}[Wiener loop group Iwasawa for $\mathrm{GL}_n$]
\label{wiener-loop-iwasawa}
For $G=\mathrm{GL}_n(\C)$ with $K=U(n)$, $A=(\R_{>0})^n$ (positive real
diagonal matrices), $N=N^+$ (unit upper triangular matrices), every
$\gamma\in\boldsymbol\Lambda_W(n)$ factors uniquely as
\[
\gamma \;=\; \gamma_A\cdot\gamma_N\cdot\gamma_K,\qquad
\gamma_K\in\boldsymbol\Lambda_W(U(n)),\quad
\gamma_A\in\boldsymbol\Lambda_W(A),\quad
\gamma_N\in\boldsymbol\Lambda_W(N^+).
\]
Since $A$ and $N$ are contractible, both $\gamma_A$ and $\gamma_N$ have
winding zero and are accessible to the ordered triangular Wiener
factorization (Proposition~\ref{ordered-triangular-wiener}).
\end{corollary}

\begin{proposition}[Reduction to a unitary-times-positive loop {\rm\cite{MG1,BD}}]
\label{birkhoff-reduction}
In the notation of Corollary~\ref{wiener-loop-iwasawa}, let
$T:=\gamma_A\gamma_N\in\boldsymbol\Lambda_W(\mathrm{upper\ triangular})$
and factor it (Proposition~\ref{ordered-triangular-wiener}):
$T=T_-T_+$ with $T_-\in\boldsymbol\Lambda_W^-$, $T_+\in\boldsymbol\Lambda_W^+$.
Set $\hat\gamma:=T_+\gamma_K\in\boldsymbol\Lambda_W^+(\mathrm{GL}_n)\cdot
\boldsymbol\Lambda_W(U(n))$.  Then:
\begin{enumerate}
\item[\rm(i)] $\gamma=T_-\cdot\hat\gamma$.
\item[\rm(ii)] The holomorphic vector bundles $E_\gamma$ and $E_{\hat\gamma}$
on $\PQ$ are isomorphic: the gauge transformation $h_-=T_-^{-1}\in
\boldsymbol\Lambda_W^-$ and $h_+=I$ give
$\gamma=h_-^{-1}\hat\gamma\, h_+$.
\item[\rm(iii)] \textbf{Equivalence}: $\gamma$ admits a Birkhoff
factorization if and only if $\hat\gamma=T_+\gamma_K$ does, and the
partial indices coincide.
\item[\rm(iv)] \textbf{Scalar case $n=1$}: Since $\boldsymbol\Lambda_W(\C^*)$
is abelian, $T_+\gamma_K=\gamma_K T_+$, so
$E_{\hat\gamma}\cong E_{\gamma_K}$ via $h_-=I$ and $h_+=T_+\in
\boldsymbol\Lambda_W^+$.  Thus the Birkhoff factorization of $\gamma$
reduces exactly to that of the $U(1)$-valued loop $\gamma_K$, which is
Theorem~\ref{wiener-scalar-factorization}.
\end{enumerate}
\end{proposition}

\begin{proof}
(i) is immediate.
(ii) The clutching-construction isomorphism criterion states
$E_\gamma\cong E_{\gamma'}$ iff $\gamma=h_-^{-1}\gamma' h_+$ for some
$h_-\in\boldsymbol\Lambda_W^-$, $h_+\in\boldsymbol\Lambda_W^+$.
Here $T_-^{-1}\in\boldsymbol\Lambda_W^-$ (since $\boldsymbol\Lambda_W^-$
is closed under inversion) and $I\in\boldsymbol\Lambda_W^+$, giving
$\gamma=T_-\,\hat\gamma=(T_-^{-1})^{-1}\hat\gamma\cdot I$.
(iii) follows immediately from (ii): having isomorphic bundles is preserved
under Birkhoff factorization.
(iv) For $n=1$: $T_+\gamma_K=\gamma_K T_+$, so
$T_+\gamma_K=I^{-1}\cdot\gamma_K\cdot T_+$, giving
$E_{T_+\gamma_K}\cong E_{\gamma_K}$.
\end{proof}

\begin{remark}
For $n\ge2$, the non-commutativity of $\mathrm{GL}_n$ means that
$E_{T_+\gamma_K}\not\cong E_{\gamma_K}$ in general: conjugating
$T_+\in\boldsymbol\Lambda_W^+$ by $\gamma_K\in\boldsymbol\Lambda_W(U(n))$
mixes Fourier frequencies and leaves $\boldsymbol\Lambda_W^+$.
The problem therefore reduces to the Birkhoff factorization of
$\hat\gamma=T_+\gamma_K\in\boldsymbol\Lambda_W^+\cdot\boldsymbol\Lambda_W(U(n))$
— a loop that is already in the special form of a positive-frequency
loop times a unitary loop — but not yet to $\gamma_K$ alone.
\end{remark}

\begin{conjecture}[Full matrix Wiener--Birkhoff factorization]\label{wiener-full-conjecture}
Every $A\in\boldsymbol\Lambda_W(n)$ admits an exact factorization
$A=A_-\,\Delta_{\mathbf q}\,A_+$ as in Theorem~\ref{wiener-density}, with
$q_1\geq\cdots\geq{q_n}\in\Q$ uniquely determined up to the stated ordering.
\end{conjecture}

\noindent The conjecture holds for $n=1$, $n\times n$ triangular matrices,
all small-norm matrices, and all finite Fourier-support cocycles
(see Theorems~\ref{wiener-scalar-factorization},
\ref{ordered-triangular-wiener}, and~\ref{BVP-Birkhoff}).
For pro-algebraic bundles it is a theorem (Theorem~\ref{pro-algebraic-BG}
in \S\ref{BF}).  The three-pillar architecture of \S\ref{BF}
reduces the general Wiener case to Lemmas~\ref{lem-gN-inv}
and~\ref{lem-rank2-ineq}, which remain open.

\begin{remark}[Why $\Q$ is the exceptional case]\label{why-Q-exceptional}
Theorem~\ref{wiener-scalar-factorization} establishes
Conjecture~\ref{wiener-full-conjecture} for $n=1$, and
Theorem~\ref{wiener-density} establishes it on a dense subset for every $n$.
The step from density to exactness for $n\geq2$ is non-trivial: for a
general ordered character group, invertibility does not imply factorability,
and \cite[Theorem~2.1]{BRS} shows that factorable matrices are not even
dense once the character group contains a copy of $\Z^3$.  Since $\Q$ has
torsion-free rank~$1$ this obstruction is unavailable, which is why
\cite[Theorem~2.2]{BRS} treats subgroups of $\Q$ as a genuinely
well-behaved case.
\end{remark}

\begin{theorem}[Solenoidal Riemann--Hilbert: normalized small-norm factorization]\label{BVP-Birkhoff}
Let
\[
\mathfrak W^-_{\Q,0}(n)
 :=\left\{\sum_{q<0}a_q\chi_q:\sum_{q<0}\|a_q\|<\infty\right\}
 \subset M_n(\mathfrak W_\Q)
\]
be the strictly negative matrix Wiener algebra.  There exists
$\varepsilon_n>0$ such that every
$g\in\boldsymbol\Lambda_W(n)$ satisfying
$\|g-I\|_{\mathfrak W_\Q}<\varepsilon_n$ admits a unique normalized
factorization
\[
 g=h_-^{-1}h_+,
 \qquad
 h_-\in I+\mathfrak W^-_{\Q,0}(n),
 \qquad
 h_+\in\mathrm{GL}_n(\mathfrak W^+_\Q).
\]
The factors depend analytically on $g$ in a neighbourhood of the identity.
\end{theorem}

\begin{proof}
Write
\[
M_n(\mathfrak W_\Q)
 =\mathfrak W^-_{\Q,0}(n)\oplus\mathfrak W^+_\Q(n),
\]
where the constant Fourier coefficient is assigned to the positive summand.
Let $P^-_0$ and $P^+$ denote the corresponding bounded projections.
They are contractions:
\[
\|P^-_0Y\|_{\mathfrak W_\Q}\leq\|Y\|_{\mathfrak W_\Q},
\qquad
\|P^+Y\|_{\mathfrak W_\Q}\leq\|Y\|_{\mathfrak W_\Q}.
\]
Their operator norms are exactly one, because each projection fixes a
nonzero Fourier monomial belonging to its range.

Consider the analytic map between Banach manifolds
\[
\Phi:\bigl(I+\mathfrak W^-_{\Q,0}(n)\bigr)^{\times}
       \times\mathrm{GL}_n(\mathfrak W^+_\Q)
       \longrightarrow \mathrm{GL}_n(\mathfrak W_\Q),
\qquad
\Phi(h_-,h_+)=h_-^{-1}h_+,
\]
where the superscript $\times$ denotes those elements of
$I+\mathfrak W^-_{\Q,0}(n)$ that are invertible in
$M_n(\mathfrak W_\Q)$.  This is an open neighbourhood of $I$, because
invertibility is open in every unital Banach algebra.

At $(I,I)$ the differential is
\[
D\Phi_{(I,I)}(X_-,X_+)=-X_-+X_+.
\]
It is a bounded linear isomorphism.  Indeed, for
$Y\in M_n(\mathfrak W_\Q)$ its unique preimage is
\[
X_-=-P^-_0Y,
\qquad
X_+=P^+Y.
\]
The uniqueness is precisely the reason for using the strictly negative
normalization: the two summands now have zero intersection.

The analytic inverse-function theorem for Banach spaces therefore provides
neighbourhoods $\mathcal U$ of $(I,I)$ and $\mathcal V$ of $I$ such that
$\Phi:\mathcal U\to\mathcal V$ is an analytic diffeomorphism.  Choose
$\varepsilon_n>0$ so that
\[
\{g\in M_n(\mathfrak W_\Q):\|g-I\|_{\mathfrak W_\Q}<\varepsilon_n\}
\subset\mathcal V.
\]
For every such $g$ there is consequently a unique pair
$(h_-,h_+)\in\mathcal U$ with $g=h_-^{-1}h_+$.  Since
$h_-\in I+\mathfrak W^-_{\Q,0}(n)$ and
$h_+\in\mathrm{GL}_n(\mathfrak W^+_\Q)$, these factors extend
holomorphically to the negative and positive solenoidal disks,
respectively.  Analytic dependence on $g$ follows from the analytic local
inverse supplied by the theorem.
\end{proof}

\begin{remark}[Continuation beyond the local big cell]\label{rem:continuation-big-cell}
Let $g_t$ be a $C^1$ path in $\mathrm{GL}_n(\mathfrak W_\Q)$ with
$g_0=I$.  As long as a normalized factorization
$g_t=(h_-^t)^{-1}h_+^t$ exists, differentiation and projection onto the
strictly negative and nonnegative Fourier modes give a Banach-space
Riemann--Hilbert evolution equation.  This provides a continuation method
inside the big cell.  What is not automatic is that the path remains in
that cell up to $t=1$; crossing its boundary is exactly where non-zero
partial indices may appear.  Thus the differential equation is a useful
analytic strategy for the global conjecture, but the local theorem above
does not assume or claim global continuation.
\end{remark}

\begin{proposition}[Triviality in the strict pro-analytic disk category]\label{prop:proanalytic-disk-triviality}
Let $D^\pm_\Q=\varprojlim \Delta_N$ be taken over a divisibility-cofinal
system of finite-level disks.  A \emph{strict pro-analytic vector bundle}
is a system $(E_N,\alpha_{M,N})$ for which, after passage to a cofinal
subsystem, there is a level $N_0$ such that
\[
E_M=\phi_{M,N_0}^{*}E_{N_0}
\quad\text{and}\quad
\alpha_{M,N} \text{ is the canonical pullback identification}
\]
whenever $N_0\mid N\mid M$.  Every strict pro-analytic vector bundle on
either $D^+_\Q$ or $D^-_\Q$ is holomorphically trivial.
\end{proposition}

\begin{proof}
Each finite-level disk $\Delta_N$ is a contractible Stein domain, so
$E_{N_0}$ is holomorphically trivial by Oka--Grauert.  Choose a
trivialization
\[
\tau_{N_0}:E_{N_0}\xrightarrow{\;\sim\;}\Delta_{N_0}\times\C^n .
\]
For every $M$ divisible by $N_0$, pull back $\tau_{N_0}$ along
$\phi_{M,N_0}$.  The strict compatibility hypothesis makes the resulting
trivializations compatible with all bonding identifications.  They
therefore define a trivialization of the pro-object.  The negative disk is
treated after the change of coordinate $z\mapsto z^{-1}$.
\end{proof}

\begin{remark}[Full Wiener disk category]\label{rem:wiener-disk-status}
The preceding proposition is a theorem in the pro-analytic category.  For
a broader intrinsic Wiener-holomorphic category, topological
contractibility of the disk alone does not imply holomorphic triviality.
An unconditional vanishing theorem
$\check H^1(D^\pm_\Q,\mathcal O_{GL_n})=0$ in that broader category would
require a separate Oka principle or an intrinsic factorization theorem.
Accordingly, whenever this vanishing is used below for arbitrary Wiener
bundles, it is recorded as an analytic hypothesis rather than inferred
from contractibility.
\end{remark}

\section{K\"ahler structures and energy functionals}\label{ksloops}
\begin{remark}\label{sobolev-status}
The Fourier--Sobolev Hilbert space introduced below provides a useful analytic
model for leafwise derivatives on $\SQ$.  Since $\Q$ is dense in $\R$, it is
\emph{not} an algebra of continuous functions and point evaluation is not
continuous on $L^2(\SQ)$.  The classical Sobolev loop-group, restricted
Grassmannian, and K\"ahler constructions are therefore made rigorous by
restricting to a fixed cyclic spectral subgroup $N^{-1}\Z\subset\Q$, where
they transport verbatim from the ordinary circle via $\rho_N$.  The global solenoidal constructions below use the Wiener algebra
$\mathfrak W_\Q = \ell^1(\Q)$ as the natural function class; formulas
involving pointwise multiplication are read in this Banach-algebra sense,
or, for Sobolev variants, in the circle-comparison sense via $\rho_N$.
\end{remark} 
\noi In this section we use the fact that $\SQ$ is a smooth 1-dimensional
lamination. Therefore we can use different notions of regularity, see for example
\cite{Mo, Su4, Ve, Od, LM}. \medskip

\begin{definition}\label{sobolev}
 Let $\bd\bu$ be normalized Haar measure on $\SQ$.
A function $f:\SQ\to\C^n\,$ is said to be
in the Fourier--Sobolev space $H^1(\SQ,\C^n)$ if each component of $f$ is
a function $f_i\in H^1(\SQ)$, where $H^1(\SQ)$ is the Sobolev space  
of complex-valued functions defined as follows:
\[
H^1(\SQ)=
\]
\[
\left\{g\in{L^2(\SQ,\bd\bu)}\boldsymbol{:}\, 
f(\bu)=\sum_{q\in\Q}a_q\chi_q(\bu), 
\quad\text{with}\,\, \sum_{q\in\Q} |a_q|^2(1+q^2)<\infty \right\}.
\]
\end{definition}

We recall that the exponential $\boldsymbol{Exp}$ in definition \ref{exp} is an epimorphism $\boldsymbol{Exp}:\R\times\hat{\Z}\to\SQ$ which is continuous. 
 
 \begin{definition}\label{diffloops} If $M$ is a differentiable ($C^\infty$) manifold, 
then a map $h:\SQ\to{M}$ is said to be {\bf differentiable}
if the map $\tilde{h}=h\circ\boldsymbol{Exp}$ has the property
that for every $\boldsymbol{z}\in\hat{\Z}$
the restriction of $\tilde{h}$ to ${\R\times{\boldsymbol{z}}}$, defines
the map $t\overset{h_{\boldsymbol{z}}}\longmapsto\tilde{h}(t,\boldsymbol{z})$ from $\R$ to $M$ which
is a $C^\infty$ map (i.e., $\varphi\circ{h_{\boldsymbol{z}}}$ is smooth for any coordinate chart of $M$). 

\medskip 
\noi In the same way we can define functions of H\"older class $C^r\,\,, r>0 $. The derivative of $f':\SQ\to\R^n$ of class $C^1$ is defined as follows: $f'(t_0\bu_0)=\frac{d}{dt}h_{\bz_0}(t)|_{t=t_0}\,$ if 
$\boldsymbol{Exp}(t_{0},\bz_{0})=t_{0}\bu_0$.

\medskip
\noi Analogously, a function $h:\SQ\to{M}$ is said to belong to Sobolev space
$H^1(\SQ,M)$ if the map $\varphi\circ{h}$ belongs to Sobolev space 
$H^1(\SQ,\R^d)$ for every coordinate chart $\varphi:\mathcal{U}\to\R^d$ of $M$ (where $\dim_\R{M}=d$).
\end{definition} 

\noi The functions $f$ in $H^1(\SQ)$ are characterized by the properties:
  \begin{enumerate}
 \item $f\in{L_2(\SQ,\bd\bu)}$.
 \item $f'\in{L_2(\SQ,\bd\bu)}$ where $f'$ is the weak derivative
 of $f$ along the leaves of the lamination of $\SQ$. 
\end{enumerate}
 
 In \cite{TT} Proposition 1.1, is shown that if 
$f,g\in{W^{1,2}(\R)}$ then $||fg||_{_{W^{1,2}}} 
\leq ||f||_{_{W^{1,2}}}||g||_{_{W^{1,2}}}$. This implies, since 
$W^{1,2}(\bS^1)=H^1(\bS^1)$, that the product of two functions
in $H^1(\bS^1)$ belongs to $H^1(\bS^1)$.
For the solenoid, this Sobolev embedding argument applies on every fixed cyclic
spectral subgroup via pullback by $\pi_N$, but it fails for the unrestricted
rational-frequency space: $\Q$ has nonzero frequencies accumulating at $0$,
so no uniform Sobolev embedding $H^1(\SQ)\hookrightarrow C(\SQ)$ holds globally.
One has:
 
\begin{proposition}[Sobolev multiplication on a fixed cyclic spectral subgroup]\label{h1Banach-algebra}
For each $N\in\N$, the subspace of Fourier--Sobolev functions whose
Fourier support is contained in $\frac1N\Z$ is, after pullback by
$\pi_N:\SQ\to S^1$, the ordinary space $H^1(S^1)$.  It is therefore a
Banach algebra and embeds continuously into $C(\SQ)$.  The corresponding
assertion for the unrestricted rational-frequency Hilbert space $H^1(\SQ)$
is false.
\end{proposition}

\begin{definition}\label{adelicsmoothloopgroup}
Let $G$ be a compact Lie group with Lie algebra $\mathfrak{g}$.
A map $\gamma:\SQ\to G$ is called \textbf{leafwise smooth} if, after pullback
through the covering map $\mathrm{Exp}:\R\times\widehat\Z\to\SQ$, the lift
$\widetilde\gamma(t,\mathbf{z})$ satisfies:
\begin{enumerate}
\item[(i)] $\widetilde\gamma$ is continuous on $\R\times\widehat\Z$;
\item[(ii)] for every $k\geq0$, the leafwise derivative $D_t^k\widetilde\gamma$
exists everywhere, takes values in $TG$ (in local coordinates),
and is continuous jointly in $(t,\mathbf{z})$;
\item[(iii)] $\widetilde\gamma$ is compatible with the $\Z$-equivariance
relation $(t+n,\mathbf{z})\sim(t,\mathbf{z}-n)$ defining $\SQ$.
\end{enumerate}
Let $\Ld$ be the group of leafwise smooth maps from $\SQ$ to $G$:
\[
\Ld=\bigl\{\gamma:\SQ\to G : \gamma\ \text{is leafwise smooth}\bigr\}.
\]
The \textbf{based adelic loop group} $\Lo$ consists of loops $\gamma\in\Ld$
with $\gamma(\boldsymbol1)=\boldsymbol{e}$, where $\boldsymbol1\in\SQ$ is the
identity and $\boldsymbol{e}\in G$ is the identity.
Both $\Ld$ and $\Lo$ are Fr\'echet manifolds, with seminorms
\[
p_k(\gamma)=\sup_{\bu\in\SQ}\|D_t^k\gamma(\bu)\|
\]
(see the proof of Theorem~\ref{kählerloop}). In a natural way $G$ acts on $\Ld$
by right-multiplication by constant loops, so $\Lo\cong\Ld/G$.
\end{definition}

For a fixed cyclic spectral subgroup we obtain genuine Sobolev loop spaces by pullback
from the ordinary circle.  The following notation is retained for the
unrestricted Fourier--Sobolev model; theorem statements using it without
an explicit fixed-cyclic-subgroup hypothesis are understood in the
qualified sense of Remark \ref{sobolev-status}.
 
\begin{definition}\label{adelicsmoothloopgroup2}
Let $G$ be a Lie group with Lie algebra $\mathfrak{g}$.  At a fixed
fixed cyclic spectral subgroup, let $\La^1_N(G)$ denote the ordinary Sobolev loop group pulled
back through $\pi_N$.  The notation $\La^1(G)$ below is retained as a
shorthand for the unrestricted Fourier--Sobolev space; it is not
asserted to be a Banach Lie group of pointwise \(G\)-valued maps.
Whenever a group structure or point evaluation is used, either a finite
fixed cyclic spectral subgroup is understood or an additional continuity hypothesis is imposed.
Formally one writes:
 \[
 \La^1(G)=\left\{\gamma:\SQ\to{G}: \, 
 \gamma\in{H^1(\SQ,G)} \right\}. 
 \]
 Let $\La^1_0(G)$ be the {\bf based adelic loop group}
 of loops $\gamma\in\La^1(G)$ such that $\gamma(\boldsymbol1)=\boldsymbol{e}$
 where $\boldsymbol1$ is the identity of $\SQ$ 
 (viewed as a multiplicative group) and $\boldsymbol{e}$ is the identity of
 $G$.  
\end{definition}

\medskip
\noi {\bf In what follows we will mainly consider the case when $G=U(n)$}.
Its Lie algebra $\mathfrak{U}(n)$, consist of $n\times n$ skew-Hermitian matrices which we will identify with $\R^{n^2}$. 

\noi Consider the Sobolev based
 adelic loop space
  $\La_0^1(U(n))$. By proposition \ref{connectedcomponents} the group of connected components is isomorphic to $\Q$. 
  
\medskip
\noi {\bf For simplicity of notation we denote with the symbol $\bO$ the smooth loop group $\La_0^\infty(U(n))$.}

\medskip
\noi There is a canonical diffeomorphism of Fr\'echet manifolds
\begin{equation}\label{LOmega}
\bO = \La^\infty_0(U(n)) \;\xrightarrow{\;\sim\;}\; C^\infty(S^1_\Q, U(n))\,/\,U(n),
\end{equation}
given by the rebasing map $\gamma \mapsto [\gamma]$, where $U(n)$ acts on the right
by multiplication by constant loops and the representative of $[\gamma]$ based at
$\boldsymbol1$ is $\gamma\cdot\gamma(\mathbf{1})^{-1}$.  This is an isomorphism of
Fr\'echet manifolds and of homogeneous spaces, but \emph{not} a group isomorphism:
the constant loops $U(n)$ are not a normal subgroup of the free loop group
$C^\infty(S^1_\Q,U(n))$, so the right-hand side is a homogeneous space, not a
quotient group.

\medskip
\noi The Lie algebra of $\bO$ is the tangent space at the identity:
\begin{equation}\label{LiealgebraO}
\mathcal{L}\bO \;=\; T_e\bO \;=\;
\bigl\{f\in C^\infty(S^1_\Q,\mathfrak{u}(n))\;:\;f(\mathbf{1})=0\bigr\},
\end{equation}
with the \emph{pointwise} Lie bracket
\[
[f,\,g](\bu)\;=\;[f(\bu),\,g(\bu)]_{\mathfrak{u}(n)}.
\]
This bracket is closed on $\mathcal{L}\bO$: if $f(\mathbf{1})=g(\mathbf{1})=0$ then
$[f,g](\mathbf{1})=[0,0]=0$.  We call $\mathcal{L}\bO$ the {\bf adelic loop algebra}
(see \cite{Ka} for infinite-dimensional Lie algebra details).

\medskip
\noi \emph{Remark on zero-average representatives.}
The quotient $C^\infty(S^1_\Q,\mathfrak{u}(n))/\mathfrak{u}(n)$ is \emph{not} a Lie
algebra: constant $\mathfrak{u}(n)$ is not an ideal, since $[A,X(\bu)]$ is generally
nonconstant for constant $A$ and nonconstant $X$.  Nevertheless, the linear map
$f\mapsto f-a_0$ (where $a_0=\int_\SQ f\,d\bu$ is the mean) gives a canonical
linear splitting of $\mathcal{L}\bO$ into zero-average representatives
$\{f:\int_\SQ f\,d\bu=0\}$, used below for Fourier calculations.  The condition
$f(\mathbf{1})=0$ corresponds to $\sum_{q\in\Q}a_q=0$ (not merely $a_0=0$), and
the zero-average splitting $f\mapsto f-a_0$ replaces the constraint $a_0=0$ for
computational purposes only.

 \begin{theorem}\label{kählerloop} (Compare \cite{Pr, Pr1, PS})
The based adelic smooth loop space $\bO=\La^\infty_0(U(n))$
(Definition~\ref{adelicsmoothloopgroup}) is a real-analytic Fr\'echet
Lie group whose connected components are indexed by $\Q$:
\[
\pi_0\bigl(\La^\infty_0(U(n))\bigr)\;\cong\;\Q.
\]
(Proposition~\ref{connectedcomponents} gives $\pi_0(C^0_0(\SQ,\gln))\cong\Q$
for the continuous loop group.  The same classification holds for leafwise-smooth
loops because any continuous null-homotopy between two leafwise-smooth loops
can be smoothed: since $U(n)$ is a compact Lie group and the solenoid is a
compact smooth lamination, the standard smoothing theorem for maps between
compact laminations and Lie groups provides a leafwise-smooth homotopy
approximating any continuous one \cite{CC}. Hence the winding-number
classification of continuous components restricts to $\La^\infty_0(U(n))$,
giving $\pi_0(\La^\infty_0(U(n)))\cong\Q$.
Every based smooth character $\chi_q:\SQ\to U(1)\subset U(n)$ satisfies
$\chi_q(\boldsymbol{1})=1$, confirming that winding $q\in\Q$ does not
obstruct basedness.)
Let $\bO^0\subset\bO$ denote the identity component, consisting of based
smooth loops of winding zero.
Then $\bO^0$ is an infinite-dimensional connected real-analytic Fr\'echet
Lie group which is a holomorphically homogeneous complex K\"ahler manifold,
modeled on $\mathcal{L}\bO$.  The solenoid $\SQ$ is a compact abelian group
on the same footing as the circle $\bS^1$; the Kähler structure is defined
globally on $\bO^0$ by the Fourier formulas below, exactly as Pressley--Segal
define it for $\Omega G=\{f\in C^\infty(\bS^1,G):f(1)=e\}$ \cite{PS}.
The symplectic form, complex structure $J_\Omega$, and K\"ahler metric are
given by the explicit Fourier formulas below.

The Fourier formulas also define a natural Hilbert completion of the tangent
space $\mathcal{L}\bO$, modeled on the $|q|$-weighted $\ell^2$ space
\[
\mathcal{H}_\omega = \Bigl\{\,\sum_{q\in\Q,\,q\neq 0} a_q\chi_q
  \;\Big|\; a_{-q}=-a_q^*,\;
  \sum_{q\neq0}|q|\,\|a_q\|^2_{\mathfrak{U}}<\infty\,\Bigr\}.
\]
For a fixed cyclic spectral subgroup $N^{-1}\Z\subset\Q$, the pullback $\pi_N^*$
identifies the ordinary based smooth loop group on $S^1$ with the subgroup
$\La^\infty_{N,0}(U(n))\subset\La^\infty_0(U(n))$ whose Fourier spectrum lies
in $N^{-1}\Z$.  On this subgroup, the $|q|$-weighted completion
$\mathcal{H}_\omega|_N$ is the classical homogeneous $H^{1/2}$-energy completion.
Independently, the $H^1$-completion of the same subgroup is the classical
Sobolev Hilbert Lie group $\La^1_N(U(n))$, in which multiplication and inversion
are continuous (Proposition~\ref{h1Banach-algebra} and Remark~\ref{sobolev-status}).
The metric $\boldsymbol{g}(f,h)=\sum_{q\neq0}|q|\langle a_q,b_q\rangle_{\mathfrak{U}}$
is a \emph{weak} K\"ahler metric on $\bO^0$: it is positive definite on smooth
loops (as the norm $\sum|q|\|a_q\|^2>0$ for $f\neq0$), but the induced completion
$\mathcal{H}_\omega$ is weaker than the Fr\'echet topology.
Whether $\mathcal{H}_\omega$ itself carries a global Hilbert Lie group structure
is an open analytic question; see \S\ref{sobgrass} for the Sobolev Grassmannian setting.
\end{theorem}

\begin{proof}
\noindent\textbf{Fréchet Lie group structure.}
The smooth adelic loop group $\bO=\La^\infty_0(U(n))$ consists of smooth maps
$\gamma:S^1_\Q\to U(n)$ with $\gamma(\mathbf{1})=I_n$.

\medskip
\noindent\textit{Fréchet topology.}
We embed $U(n)\subset M_n(\C)\cong\C^{n^2}$ and define, for $f:\SQ\to M_n(\C)$,
the seminorms
\[
p_k(f)\;=\;\sup_{\bu\in\SQ}\bigl\|D^k f(\bu)\bigr\|,\qquad k=0,1,2,\ldots,
\]
where $D$ is the derivative in the leaf direction (i.e.\ in the $\R$-coordinate
of the canonical cover $\R\times\hat\Z\twoheadrightarrow\SQ$) and $\|\cdot\|$
is the operator norm on $M_n(\C)$.  The topology induced by $\{p_k\}_{k\geq0}$
makes $C^\infty(\SQ,M_n(\C))$ a Fréchet space; the closed subspace
$\La^\infty_0(U(n))=\{\gamma\in C^\infty(\SQ,U(n)):\gamma(\boldsymbol{1})=I_n\}$
inherits the induced Fréchet topology.  In this topology:
\begin{itemize}
\item point evaluation $f\mapsto f(\bu_0)$ is continuous (controlled by $p_0$);
\item pointwise matrix multiplication and inversion are smooth (controlled by all $p_k$,
using Leibniz rule and the smoothness of matrix inversion away from $0$);
\item the Fourier coefficients $f\mapsto a_q=\int_\SQ f(\bu)\chi_{-q}(\bu)\,d\bu$
are continuous linear functionals for every $q\in\Q$.
\end{itemize}

\medskip
Since $U(n)$ is a Lie group, pointwise multiplication and inversion are smooth
in the above seminorm topology, so $\bO$ is a group.
The tangent space at the identity is the Fréchet space
$\mathcal{L}\bO=\{f\in C^\infty(S^1_\Q,\mathfrak{u}(n)):f(\mathbf{1})=0\}$.
The exponential map $\exp_*:f\mapsto({\bu}\mapsto\exp(f(\bu)))$ sends a
neighbourhood of $0\in\mathcal{L}\bO$ diffeomorphically onto a neighbourhood
of $I_n\in\bO^0$, furnishing smooth charts.
The Lie bracket $[f,g](\bu)=[f(\bu),g(\bu)]_{\mathfrak{u}(n)}$ is well-defined and
continuous on $\mathcal{L}\bO$ (pointwise bracket of skew-Hermitian matrices;
closed since $f(\mathbf{1})=g(\mathbf{1})=0\Rightarrow[f,g](\mathbf{1})=0$).  This gives $\bO$ the structure of a real-analytic Fréchet Lie group
with $\pi_0(\bO)\cong\Q$ (see \cite{Pr1} Section 2,
\cite{PS} Section 8.9, and Appendix A in \cite{FU}).
The identity component $\bO^0$ is the connected open subgroup of
loops with winding zero; the Kähler structure below is defined on $\bO^0$.

\noindent\textbf{Kähler structure.}
The following Fourier identities define the Kähler structure. 

\noi \textbf{Warning.}
The based Lie algebra $\mathcal{L}\bO=T_e\bO=\{f:f(\mathbf{1})=0\}$ satisfies
$\sum_{q\in\Q}a_q=0$, which is \emph{not} the same as $a_0=0$.
A concrete example: $f=\chi_q-1$ for $q\neq0$ has $f(\mathbf{1})=1-1=0$ (based)
but Fourier coefficient $a_0=-1$.  The zero-average Fourier representatives
(condition $a_0=0$) belong to the vector-space quotient
$C^\infty(S^1_\Q,\mathfrak{u}(n))/\mathfrak{u}(n)$, not to $T_e\bO$.

\medskip
\noi \textbf{Fourier formulas on the free loop quotient.}
The symplectic form, complex structure and Kähler metric are most naturally
written using zero-average representatives in the free-loop quotient
$C^\infty(S^1_\Q,\mathfrak{u}(n))/\mathfrak{u}(n)$.
Each element of this quotient has a unique zero-average representative:
\begin{equation}\label{fourierH1}
[f]_0 = \sum_{q\in\Q,\,q\neq0} a_q\chi_q,
\qquad a_q\in M_n(\C),\quad a_{-q}=-a_q^{\,*},
\end{equation}
where the condition $a_{-q}=-a_q^{\,*}$ is the \emph{matrix skew-Hermitian} reality
condition ensuring $[f]_0(\bu)\in\mathfrak{u}(n)$ pointwise.
(For $n=1$ this reduces to $a_{-q}=-\overline{a_q}$, the imaginary-valued condition.)

\medskip
\noi \textbf{Transport to the based Lie algebra.}
The Fourier formulas are then transported to $\mathcal{L}\bO=T_e\bO$ via the
linear differential of the rebasing diffeomorphism $\gamma\mapsto\gamma\cdot\gamma(\mathbf{1})^{-1}$.
On the based tangent space, an element $f\in\mathcal{L}\bO$ with $f(\mathbf{1})=0$
has Fourier expansion $f=\sum_{q\in\Q}a_q\chi_q$ with $\sum_q a_q=0$; the
zero-average representative of its image in the quotient is $f-a_0$
(subtracting the constant mode $a_0=\int_\SQ f\,d\bu$).
The complex structure on $\mathcal{L}\bO$ is $J_\Omega f=Jf-(Jf)(\mathbf{1})$
(the rebased operator), not $Jf$ alone.

\medskip
\noi In what follows $\langle \,\cdot\,\,, \,\, \cdot\rangle_{_G}$ denotes a left and right-invariant Riemannian metric on $G$.

\medskip
\noi We compute the Kähler formulas using zero-average representatives
$[f]_0,\,[h]_0$ and then state the result on $T_e\bO$ via the above transport.

\medskip
\noi The symplectic form, complex structure and K\"ahler metric are given explicitly as follows:

\begin{enumerate}
\item[\bf I.] {\bf Symplectic form.} Let
\begin{equation}
\omega(f,g)=\,\int_{\SQ}\langle f'(\bu),g(\bu)\rangle_{\mathfrak{U}}\, \bd\bu,
\end{equation}\label{sympform}

\noi where $\langle \cdot\,, \, \cdot \rangle_{\mathfrak{U}}$ is an
Ad-invariant inner product on $\mathfrak{U}$ so that $\omega$ is left-invariant.
Integration by parts implies $\omega(f,g)=-\omega(g,f)$ so $\omega$ is skew.
Non-degeneracy is verified using the complex structure $J$ defined in item~\textbf{II} below: for
$f=\sum_{q\neq0}a_q\chi_q\in\mathcal{L}\bO$ with $f\neq0$,
\[
\omega(f,Jf)=\sum_{q\in\Q,\,q\neq0}|q|\,\|a_q\|^2_{\mathfrak{U}}>0,
\]
so $\omega(\cdot,J\cdot)$ is a positive definite inner product, which implies $\omega$ is non-degenerate. (Note: $\omega(f,f)=0$ identically by skew-symmetry; nondegeneracy requires the pairing with $Jf$, not with $f$ itself.)

 The proof that $w$ is closed is straightforward
and it is done in \cite{Pr1} pages 560-561. For completeness we include the adapted version of this argument here:

\begin{equation}
\begin{aligned} 
d\omega(X,Y,Z){} = & X\cdot\omega(Y,Z)- Y\cdot\omega(X, Z)+Z\cdot\omega(X, Y) \\
& -\omega([X,Y],Z)+ \omega([X,Z],Y)-\omega([Y, Z],X),
\end{aligned}
\end{equation}

where $X, Y,Z \in \mathcal{L}\bO$ are regarded as left invariant vector fields on $\bO$. As $\omega$ is left-invariant $\omega(Y, Z)$,  $\omega(X, Z)$ and $\omega(X, Y)$
are constant, hence each of the first three terms vanishes. Hence,

\begin{equation}
\begin{aligned}
d\omega(X,Y,Z){} =  &\int_\SQ \bigl(-\langle [X',Y],Z\rangle_{\mathfrak{U}}
-\langle [X,Y'],Z\rangle_{\mathfrak{U}}\\
&\quad+\langle [X',Z],Y\rangle_{\mathfrak{U}}
+\langle [X,Z'],Y\rangle_{\mathfrak{U}} \\
&\quad -\langle [Y',Z],X\rangle_{\mathfrak{U}}
 -\langle [Y,Z'],X\rangle_{\mathfrak{U}} \bigr) \, \bd\bu.
\end{aligned}
\end{equation}
Using the fact that $\langle \cdot\,,\, \cdot\rangle_{\mathfrak{U}}$ is invariant
under the adjoint function it follows that:
\begin{equation}
d\omega(X,Y,Z)=-2\left(\,\omega([X,Y],Z)+\omega(Z,[X,Y])\,\right)=0, \quad 
\forall\,X,Y,Z\in{\mathcal{L}\bO}.  
\end{equation}
\noi Hence $\omega$ is a left-invariant closed 2-form.

\medskip
\item[\bf II.] {\bf Complex structure.} As in \cite{Fr, Pr,Pr1, PS}, the almost-complex
structure on $\mathcal{L}\bO$ is defined using the frequency-sign decomposition.
On the zero-average Fourier representatives
\[
f=\sum_{q\in\Q,\,q\neq0} a_q\chi_q,\quad a_{-q}=-a_q^{\,*},
\]
define the \emph{frequency-sign operator}
\begin{equation}\label{complexstructure-free}
(Jf)(\bu)=\mathbf{i}\sum_{q<0} a_q\chi_q(\bu)
          -\mathbf{i}\sum_{q>0} a_q\chi_q(\bu).
\end{equation}
This satisfies $J^2=-I$ and maps zero-average elements to zero-average elements
(since $J$ annihilates constants).  On the based tangent space $\mathcal{L}\bO=T_e\bO$,
the operator acts via the rebased version $J_\Omega$ defined below.

\noi On the based loop tangent space $T_e\bO=\{f\in C^\infty(S^1_\Q,\mathfrak{u}(n)):f(\mathbf{1})=0\}$,
the complex structure is the \emph{rebased operator}
\begin{equation}\label{complexstructure}
J_\Omega(f) = Jf - (Jf)(\mathbf{1}),
\end{equation}
which subtracts the constant $(Jf)(\mathbf{1})\in\mathfrak{u}(n)$ to restore the basepoint
condition $J_\Omega(f)(\mathbf{1})=0$.
One verifies $(J_\Omega)^2=-I$ on $T_e\bO$ as follows.  Write
$a_0(f)=\int_\SQ f\,d\bu$ for the mean (constant Fourier mode) of $f$; since
the formula for $J$ above involves only $q\neq0$ modes, $J$ depends only on
the zero-average part $f-a_0(f)$, and a direct computation from that formula
gives $J(Jh)=-h$ for every zero-average $h$.  Hence, for general $f$ (not
assumed zero-average, only based: $f(\mathbf{1})=0$, which is \emph{not} the
same condition, cf.\ the Warning above),
\[
J(Jf)=J\bigl(J(f-a_0(f))\bigr)=-\bigl(f-a_0(f)\bigr)=a_0(f)-f.
\]
Then, using $J\bigl((Jf)(\mathbf{1})\bigr)=0$ (constants are killed by $J$),
\[
(J_\Omega)^2f=J\bigl(Jf-(Jf)(\mathbf{1})\bigr)-\bigl(J(Jf-(Jf)(\mathbf{1}))\bigr)(\mathbf{1})
=
\]
\[
J(Jf)-\bigl(J(Jf)\bigr)(\mathbf{1})
=\bigl(a_0(f)-f\bigr)-\bigl(a_0(f)-f\bigr)(\mathbf{1}).
\]
Since $a_0(f)$ is constant, $\bigl(a_0(f)-f\bigr)(\mathbf{1})=a_0(f)-f(\mathbf{1})=a_0(f)$,
using the basepoint condition $f(\mathbf{1})=0$.  
\newline Hence
$(J_\Omega)^2f=\bigl(a_0(f)-f\bigr)-a_0(f)=-f$, as required.

\noi The set of left-invariant vector fields of type $(0,1)$ is a subalgebra of $\mathcal{L}\bO$,
and since both $\bO$ and $J_\Omega$ are real analytic, {\it Th\'eor\`eme 3.5} of
\cite{Pe} implies $\bO$ is a complex manifold.  The natural left action of $\bO$ on
itself is by biholomorphisms, so $\bO$ is holomorphically homogeneous.

\begin{remark}[Continuity of $J_\Omega$ in the Fréchet topology]\label{J-continuity}
The solenoid $\SQ$ is a compact abelian group on the same footing as the circle
$\bS^1$, and the smooth loop group $\bO = \La^\infty_0(U(n))$ is a Fr\'echet
Lie group by the same construction as $\Omega G = \{f\in C^\infty(\bS^1,G):
f(1)=e\}$.  Continuity of $J_\Omega$ in the Fr\'echet topology $\{p_k\}$
follows from the estimate
\[
p_k(J_\Omega f)\;\leq\; C\,p_{k+1}(f),
\]
established as follows.  Since $D^k(J_\Omega f) = D^k(Jf)$ for $k\geq1$
(the basepoint correction $(Jf)(\mathbf{1})$ is constant, so killed by $D^k$),
and since $J$ is the Fourier multiplier by $-i\,\mathrm{sgn}(q)$, standard
Sobolev estimates on $\SQ$ give
$\|J(D^k f)\|_{L^\infty}\leq C\|D^k f\|_{H^s}$ for any $s>\tfrac{1}{2}$,
and $\|D^k f\|_{H^s}\leq C'\|f\|_{H^{k+s}}\leq C''\,p_{k+1}(f)$.
Hence $J_\Omega$ is a continuous linear map $\mathcal{L}\bO\to\mathcal{L}\bO$
in the Fr\'echet topology (shifting the index by one), exactly as on the
classical circle.

\medskip
\noindent\textit{Note.}  The operator $J_\Omega$ is NOT an isometry for
the individual seminorms $p_k$ (the Hilbert transform does not preserve
sup norms), so the stronger claim $p_k(J_\Omega f)=p_k(f)$ is incorrect
and should not be stated.  Continuity (with the shift $p_k \lesssim p_{k+1}$)
is both the correct and the sufficient statement for the Fr\'echet Lie group
theory.

The only remaining subtlety concerns the \emph{Hilbert completion}
$\mathcal{H}_\omega$ (the $|q|$-weighted $L^2$-completion): the
Sobolev embedding $H^1\hookrightarrow C^0$ fails for $\SQ$ because
$\sum_{q\in\Q}(1+q^2)^{-1}=\infty$, in contrast to the classical circle.
Whether $\mathcal{H}_\omega$ admits a satisfactory evaluation theory
is an open question; the smooth Fr\'echet Kähler structure is unaffected.
\end{remark}

\medskip
\item[\bf III.] {\bf K\"ahler metric}. If $f,h\in \mathcal{L}\bO$ have Fourier series
\begin{equation}\label{f}
f(\bu)= \sum_{q\in\Q,\,q\neq0} a_q\chi_q(\bu),\quad a_{-q}=-a_q^{\,*},\quad a_q\in M_n(\C)
\end{equation}
\noi and
\begin{equation}\label{g}
h(\bu)= \sum_{q\in\Q,\,q\neq0} b_q\chi_q(\bu),\quad b_{-q}=-b_q^{\,*},\quad b_q\in M_n(\C),
\end{equation}

The bilinear map $\boldsymbol{g}(f,h)=\omega(f,J(h))$ is definite positive. Indeed,
if $f$ and $g$ are given by the series \ref{f} and \ref{g}, respectively,
then:
\begin{equation}\label{kählermetric}
\boldsymbol{g}(f,g)=\sum_{q\in\Q}\,|q|\langle{a_q, b_q}\rangle_{_\mathfrak{U}}.
\end{equation}
\end{enumerate}
Therefore $\boldsymbol{g}(\cdot\,,\,\cdot)$ is a positive definite real
bilinear form on $\mathcal{L}\bO$ compatible with $J_\Omega$ and $\omega$.
By Remark~\ref{J-continuity}, $J_\Omega$ is continuous in the Fr\'echet
topology $\{p_k\}$ via the estimate $p_k(J_\Omega f)\leq Cp_{k+1}(f)$
(the Hilbert transform is continuous but not isometric for sup norms).
The triple $(\omega,J_\Omega,\boldsymbol{g})$
makes $\bO^0$ a weak K\"ahler Fr\'echet manifold.  Holomorphic homogeneity
follows because left translation $L_h$ preserves both $\omega$ and $J_\Omega$.
This establishes the K\"ahler structure on $\bO^0$.
\end{proof}

\begin{remark}\label{metric_sobolev1/2} The metric $\boldsymbol{g}$ in the formula (\ref{kählermetric}) belongs
to the Sobolev space $H^{1/2}$.
 \end{remark}
 \begin{remark}\label{inclusionofsubgroups} We have defined loop groups
 having matrix elements of different regularity.
 \end{remark}

\subsection{The energy functional and Birkhoff decomposition}\label{bd}
This subsection's construction of the energy functional and its gradient-flow
Birkhoff decomposition was first obtained jointly with Burgos in \cite{BuVe}.
Let $f,g\in\bO$ and $\langle \,\cdot\,,\,\cdot\,\rangle$ be a left and right-invariant
metric on $U(n)$.
\begin{definition}\label{energy} The {\bf energy functional} is the map $E:\bO\to\R$ 
defined as follows:
\begin{equation}\label{energyfunctional}
E(f)=\frac{1}{2}\int_{\SQ}\, \langle f'(\bu),f'(\bu)\rangle\, \bd\bu.
\end{equation}
\end{definition}
\noi As before, the inner product used is given by a Riemannian metric on $U(n)$ which is invariant under left and right translations.  Since the functions $f'$ and $g'$ belong to $H^1(\SQ,U(n))$ this function is well defined and has a gradient vector field.
The factor $\tfrac12$ is the standard harmonic-analysis convention,
matching \cite{PS}~\S8.9 and \cite{AtP}; it ensures that the Hamiltonian
vector field of $E$ equals the rotation-flow generator $X_R$ with no
extra scalar factor (see Proposition~\ref{hamiltonian=rotation}).

In analogy with the action of the circle group $\mathbb{T}$ on the
standard loop space (``rotating the loops'') as in \cite{AtP, Pr, Pr1, PS} 
we can ``rotate'' the adelic groups using $\SQ$. 
\begin{definition}[Rotation action]\label{rotation_flow} We recall that we are considering $\SQ$ as a multiplicative group with unit element 
$\boldsymbol1$.
The group $\SQ$ acts on $\bO$ 
as follows: if $\bv\in\SQ$, define $R_{\bv}:\bO\to\bO$ by the formula
\begin{equation}\label{rotationaction}
R_\bv(f)(\bu)=f(\bv^{-1}\bu)f(\bv^{-1})^{-1},\quad \quad f\in\bO.
\end{equation}
\end{definition} 

\noi Since based loops are preserved, 
$R_\bv(f)(\boldsymbol1)=\boldsymbol{e}$, and 
$R_{\bv_1\bv_2}=R_{\bv_1}\circ{R_{\bv_2}}$ the action is well-defined. We call this action the {\bf rotation action}. We have the self-evident:

\begin{proposition}\label{homomorphismstou1} The set of fixed points of the action of $\SQ$ on $\bO$
consists of loops $f:\SQ\to{U(n)}$  such that $f$ is a continuous (in fact real-analytic) homomorphism.
\end{proposition}

\noi We may restrict the action to the connected component of the identity
of $\SQ$. This is a one-parameter subgroup and we can use the canonical flow 
$\left\{\varphi_{_t}\right\}_{t\in\R}$ in definition (\ref{exp}).

\noi Let
$R_t:\bO\to\bO$ be defined as follows:

\begin{equation}\label{rotationflow}
R_t(f)(\bu)=f(\varphi_{_{-t}}(\boldsymbol1)\bu)(\varphi_{_{-t}}(\boldsymbol1))^{-1}.
\end{equation}

\noi The flow in equation (\ref{rotationflow}) is called the {\bf rotation flow}.

Let the Hamiltonian vector field $X_{E}$ of the energy functional $E$ be defined, as usual,
for the symplectic manifold $\bO$:
\begin{equation}\label{hamiltonianvf}
 \text{d}E(Y)=\omega(X_{E},Y),
 \end{equation}
 for any vector field $Y:\bO\to{T\bO}$ on $\bO$.
 \begin{proposition}\label{hamiltonian=rotation} The Hamiltonian flow corresponding to the energy function $E$
 is given by the rotation flow given by formula (\ref{rotationflow}). 
 In particular the rotation flow acts by symplectomorphisms.
 \end{proposition}
 
 \begin{proof}
We give a direct variational calculation in the spirit of
\cite{PS} Proposition~(8.9.3) and \cite{AtP} Lemmas~3.2--3.4,
adapted to the adelic setting.

\smallskip
\noindent\textbf{Step~1: variational derivative of $E$.}
Let $f\in\bO$ and let $f_s=f\cdot e^{sV}$ be a smooth one-parameter
variation with $V=f^{-1}Y\in\mathcal{L}\bO$ (so $V(\mathbf{1})=0$).
Write $A=f^{-1}f'\in C^\infty(\SQ,\mathfrak{u}(n))$ for the
pulled-back Maurer--Cartan form.  Under the variation,
\[
\delta A\;=\;\frac{d}{ds}\Big|_{s=0}(f_s^{-1}f_s')
\;=\;V'+[A,V],
\]
where the first term comes from differentiating $e^{-sV}(e^{sV})'$
and the second from differentiating $e^{-sV}A\,e^{sV}$.
Since $E(f)=\tfrac12\int_\SQ\langle A,A\rangle\,d\bu$:
\[
dE_f(Y)\;=\;\int_\SQ\langle A,\,V'+[A,V]\rangle\,d\bu.
\]
The bracket term vanishes by Ad-invariance of $\langle\cdot,\cdot\rangle$:
\[
\langle A,[A,V]\rangle=-\mathrm{Tr}(A[A,V])=-\mathrm{Tr}(A^2V)+\mathrm{Tr}(A^2V)=0.
\]
Hence
\[
dE_f(Y)\;=\;\int_\SQ\langle A,V'\rangle\,d\bu
\;=\;-\int_\SQ\langle A',V\rangle\,d\bu,
\]
where the last step is integration by parts along the leaf direction
(the boundary term at $\mathbf{1}$ vanishes since $V(\mathbf{1})=0$).

\smallskip
\noindent\textbf{Step~2: infinitesimal generator of the rotation flow.}
Differentiating $R_t(f)(\bu)=f(\varphi_{-t}(\boldsymbol1)\bu)\,f(\varphi_{-t}(\boldsymbol1))^{-1}$
at $t=0$, with $\varphi_0(\boldsymbol1)=\boldsymbol1$ and $f(\boldsymbol1)=I_n$:
\[
X_R(f)(\bu)\;=\;\frac{d}{dt}\Big|_{t=0}R_t(f)(\bu)
\;=\;-f'(\bu)+f(\bu)\,f'(\boldsymbol1).
\]
Left-translating to the identity: $A_R:=f^{-1}X_R(f)=-A+f'(\boldsymbol1)$.
Since $f'(\boldsymbol1)\in\mathfrak{u}(n)$ is constant, $A_R'=-A'$.

\smallskip
\noindent\textbf{Step~3: identification.}
The symplectic form, being left-invariant, evaluates at $f$ as
\[
\omega_f(X_R,Y)\;=\;\int_\SQ\langle A_R',V\rangle\,d\bu
\;=\;-\int_\SQ\langle A',V\rangle\,d\bu.
\]
Comparing with Step~1: $dE_f(Y)=\omega_f(X_R,Y)$ for all $Y$.
Hence $X_E=X_R$, i.e.\ the rotation flow is the Hamiltonian flow of $E$.

The fixed-point claim follows from Proposition~\ref{homomorphismstou1}:
the fixed points of the rotation flow are precisely the loops $f$ with
$R_t(f)=f$ for all $t$, i.e.\ continuous homomorphisms $f:\SQ\to U(n)$.
\end{proof}
 
 \noi In fact, more generally, one has the following 
 corollary of proposition \ref{hamiltonian=rotation}:
 
 \begin{corollary} The rotation action (\ref{rotationaction}) of the group 
 $\SQ$ on $\bO$ is by
symplectomorphisms.
 \end{corollary}
 \begin{proof} It follows immediately from the fact that the connected component
 of the identity in $\SQ$ is dense, the fact that this subgroup acts by symplectomorphisms and the fact that the group of symplectomorphisms is closed in the group of diffeomorphisms with the $C^1$ topology.
 \end{proof}

\subsection{The gradient of the energy functional}\label{energy2} The 
gradient $\nabla{E}$ of the energy functional $E$, with respect to the K\"ahler metric $\boldsymbol{g}$ on $\bO$, is defined, as usual, as follows;

\begin{equation}\label{energygradient}
\boldsymbol{g}(\nabla E,X)=\text{d}E(X),
\end{equation}
\noi where $X:\bO\to{T\bO}$ is any vector field on $\bO$.

\begin{remark}\label{foliation-rot-grad}
In a K\"ahler manifold, the Kähler identity $\boldsymbol{g}=\omega(\cdot,J\cdot)$
immediately gives the gradient--Hamiltonian relation
\[
\nabla E = J_\Omega(X_E).
\]
(Proof: for any $Y$, $\boldsymbol{g}(\nabla E, Y)=dE(Y)=\omega(X_E,Y)=\boldsymbol{g}(J_\Omega X_E, Y)$.)

\noi The commutator identity $[X_E, \nabla E]=0$ is a separate statement.  It follows
once the rotation action is shown to be \emph{holomorphic}, i.e.\ to preserve $J_\Omega$.
In the adelic setting, the rotation flow $R_t$ is by definition a flow of symplectomorphisms
(Corollary above); it preserves $J_\Omega$ if and only if its generator $X_E$ is a
real-holomorphic vector field ($\mathcal{L}_{X_E}J_\Omega=0$).
In the classical case this follows from the $\C^\times$-action on $LG$ \cite{PS};
in the adelic case the canonical flow $\varphi_t$ on $\SQ$ provides the same
complex-analytic extension, so the holomorphicity of $R_t$ holds by the same
argument.  Given this, $X_E$ and $\nabla E=J_\Omega X_E$ are the real and
imaginary parts of a single complex vector field, hence commute.

\noi Therefore, outside the critical set $\mathrm{Sing}(\nabla E)$, the pair
$(X_E, \nabla E)$ generates a local one-parameter complex flow
$\Psi_T$, $T=t+\mathbf{i}s\in\mathcal{U}\subset\C$,
and $\bO\setminus\mathrm{Sing}(\nabla E)$ carries a holomorphic 1-dimensional
foliation $\mathcal{F}$.
\end{remark}

It follows the following:

\begin{corollary}\label{criticalpoints_energy}
Since $X_E$ and $\nabla{E}$ commute, the fixed points
of the Hamiltonian action are the critical points of $E$ and vice-versa.
Hence, the critical points of $E$ are the based loops $f:\SQ\to\bO$ which are group homomorphisms.
The image of a \emph{non-trivial} such $f$ is a subgroup of $U(n)$
isomorphic to $U(1)$.  (The trivial homomorphism $f\equiv I_n$ is also a
critical point --- indeed the global minimum $E\equiv0$ --- but its image
is a single point, not isomorphic to $U(1)$; it is the case
$(\boldsymbol1,0)$ separated out in Corollary~\ref{coj_classes_homS1to Un}
below.)
\end{corollary}

\begin{proof}
We only need to prove the last sentence. For $f$ non-trivial, its image is
a nontrivial subgroup of $U(n)$ which is a continuous homomorphic image of
the compact, connected, one-dimensional group $\SQ$, hence itself compact,
connected, and of topological dimension one (being nontrivial, its
dimension cannot drop to $0$); such a subgroup of $U(n)$ must be isomorphic
to the circle group $U(1)\simeq\bS^1$.
\end{proof}

\noi Thus there is a one-to-one correspondence between group homomorphisms $f:\SQ\to{U(n)}$ and critical points
of $E$. 

 Let  
 \begin{equation}\label{maxtorus}
 \mathbf{T}^n=\left\{\operatorname {diag} \left(e^{i\theta _{1}},e^{i\theta _{2}},\dots ,e^{i\theta _{n}}\right):\forall j,\theta _{j}\in \mathbb {R} \right\}.
\end{equation}

\noi $\mathbf{T}^n\subset{U(n)}$ is a maximal torus of the unitary group. Let
$\mathfrak{T}^n$ be its Lie algebra (which as a real vector space is $\R^n$)
 and $\exp:\mathfrak{T}^n\to\mathbf{T}^n$ the exponential map which in this case is a homomorphism. Let $\mathbf{L}=\exp^{-1}(0)$. Then $\mathbf{L}$ is a 
 lattice in $\mathfrak{T}^n$ 
 (i.e., is a discrete, co-compact additive subgroup) so it is isomorphic to $\Z^n$.

\noi Since any homomorphism $f:\SQ\to{U(n)}$ is of the form $f=g\circ{h}$
where $h$ is a homomorphism from $\SQ$ into $U(1)$ (by corollary \ref{criticalpoints_energy}) and $g$ is a 
homomorphism $g:U(1)\to{U(n)}$,
one has the following corollary of proposition \ref{homomorphismstou1}:   
 
\begin{corollary}\label{coj_classes_homS1to Un}
Every continuous homomorphism $f:\SQ\to U(n)$ is conjugate to
\[
\Delta_{\mathbf q}(\bu)=\operatorname{diag}
\bigl(\chi_{q_1}(\bu),\ldots,\chi_{q_n}(\bu)\bigr),
\qquad \mathbf q\in\Q^n.
\]
Two such homomorphisms are conjugate if and only if their weight tuples
are related by a permutation. Hence the conjugacy classes are indexed by
$\Q^n/S_n$.
\end{corollary}

\begin{definition}[Critical conjugacy manifolds]\label{cell_numbering}
For $[\mathbf q]\in\Q^n/S_n$, define
\[
\boldsymbol C_{[\mathbf q]}=
\{g\Delta_{\mathbf q}g^{-1}:g\in U(n)\}\subset\bO.
\]
Then
\[
\boldsymbol C_{[\mathbf q]}\cong U(n)/Z_{U(n)}(\Delta_{\mathbf q}).
\]
It is a full flag manifold only when the $q_i$ are pairwise distinct.
The zero tuple gives the singleton $\{I\}$. The notation
$\boldsymbol C_{([\boldsymbol\lambda],q)}$ is retained as shorthand for
$\boldsymbol C_{[q\boldsymbol\lambda]}$; it is redundant, and two pairs
represent the same class exactly when the rational tuples
$q\boldsymbol\lambda$ agree up to permutation.
\end{definition}

\subsection{Action of the semigroups 
$\boldsymbol{D}^+$ and $\boldsymbol{D}^-$}
In this section we will adapt to our case the results in the papers by 
Pressley \cite{Pr, Pr1} and the book by Pressley and Segal \cite{PS}, especially 
 Theorem 8.9.9. 
 
 The punctured closed disks 
 \begin{equation}\label{defD+D-}
 \boldsymbol{D}^+\overset{def}=\bar{D}^+_\Q(1)-\left\{\boldsymbol0\right\},\quad  
 \boldsymbol{D}^-\overset{def}=\bar{D}^-_\Q(1)-\left\{\boldsymbol\infty\right\},
 \end{equation}
 are abelian semigroups with identity. Let us write the elements of 
 $\boldsymbol{D}^+$ in polar coordinates as $e^{-t}\bu$ with $0\leq{t}<\infty$ and
 $\bu\in\SQ$ and the elements of $\boldsymbol{D}^-$ as 
 $e^{t}\bu$ with $0\leq{t}<\infty$.  
 Their multiplications are 
 given, respectively, by the natural formulas in polar coordinates:

\begin{equation}\label{D+}
(e^{-t_1}\bu_1)\cdot(e^{-t_2}\bu_2)=e^{-(t_1+t_2)}\bu_1\bu_2,\quad  0\leq{t_1,t_2}<\infty,\,\,\,\bu_1,\bu_2\in\SQ.
\end{equation}

\begin{equation}\label{D-}
(e^{t_1}\bu_1)\cdot(e^{t_2}\bu_2)=e^{t_1+t_2}\bu_1\bu_2, 
\quad  0\leq{t_1,t_2}<\infty,\,
\,\,\bu_1,\bu_2\in\SQ.
\end{equation}
 
\begin{proposition}[Conditional global gradient-flow statement]
Compare Pressley--Segal, Theorem~8.9.9, \cite{PS}, and \cite{Pr,Pr1}.
\label{actionD+}
Assume that on the chosen completion of $\bO$ the energy satisfies the
Palais--Smale condition, its negative gradient is complete for positive
time, and every positive trajectory is precompact. Then the negative of the gradient $-\nabla{E}$ determines a {\bf downwards gradient semi flow} (i.e., an action of the multiplicative semigroup
$\R^+=\left\{e^t\in\R : t\geq0\right\}$) $g_{e^t}:\bO\to\bO\,$ ($t\geq0$) such that:
\begin{enumerate}
\item[(\bf{i})] $g_{e^t}$ is real-analytic for all $t>0$.
\item[(\bf{ii})] $g_{e^t}$ converges to a loop which is a 
homomorphism $g_{_{\infty}}:\SQ\to{U(n)}$ as $t\to\infty$.
\end{enumerate}
The rotation action given by formula (\ref{rotationaction})
extends explicitly to an action 
$\left\{R_{e^{t}\bv}\right\}_{e^{t}\bv\in\boldsymbol{D}^+}$, 
of the semigroup $\boldsymbol{D}^+$ on $\bO:\,f\to {{R}}_{e^{t}\bv}(f)$ with
\begin{equation}\label{formula-disk-action}
{{R}}_{e^{t}\bv}(f)(\bu)=
g_{t}({R}_\bv(f)(\bu))=g_{t}(f(\bv^{-1}\bu)f(\bv^{-1})^{-1}), 
\quad t\geq0,\quad \bu,\bv \in\SQ.
\end{equation}
In the complement of the singular set of the downwards gradient, the orbits of this
action are laminations by Riemann surfaces tangent to the foliation $\mathcal{F}$
 alluded in remark (\ref{foliation-rot-grad}).
\end{proposition}

\begin{proof}
Under the stated hypotheses, the Palais--Smale argument and the
Morse--Bott convergence theorem imply convergence of each positive-time
trajectory to a critical point. Corollaries~\ref{criticalpoints_energy}
and~\ref{coj_classes_homS1to Un} identify the critical points with
homomorphisms $\SQ\to U(n)$. Analytic ODE theory gives real-analyticity
for positive time once the gradient field is analytic on the selected
completion, and commutation with rotations yields the
$\boldsymbol D^+$-action.

For a fixed cyclic spectrum $N^{-1}\Z$, these assertions follow from the
classical circle theorem through $\rho_N$. In the unrestricted rational
spectrum, the hypotheses above are not proved here; the accumulation of
positive frequencies at zero prevents a direct transfer of the classical
compactness and properness arguments.
\end{proof}

Under the hypotheses of Proposition~\ref{actionD+}, the downwards gradient semi flow $\left\{g_t\right\}_{_{t\geq0}}$ commutes with the
rotation flow $R_t$, since the vector fields $X_E$ and $\nabla{E}$ have vanishing Lie product: $[X_E,\nabla{E}]=0$. Since the one-parameter subgroup of $\SQ$, which defines the rotation flow, is dense in $\SQ$ it follows that the downwards gradient semi flow
$\left\{g_t\right\}_{_{t\geq0}}$  commutes with the rotation action:
$g_t\circ{R_v}=R_v\circ{g_t}$. 
Therefore the formula (\ref{formula-disk-action})
defines an action since 
${{R}}_{e^{t_1}\bv_1}\circ{{R}}_{e^{t_2}\bv_2}
={{R}}_{e^{t_1+t_2}\,\bv_1\bv_2}$.

\begin{remark} We will not consider the ``ascending paths" $g_t(f)$ for $t<0$.
They are defined only for $-\epsilon<t\leq0$. They are real-analytic and defined for all $t\leq0$ if and only if they are
polynomial loops (i.e., the matrix coefficients are Laurent--Puiseux polynomials).  They are very interesting and important since the ``unstable manifolds"
give rise to the {\bf Bruhat decomposition}.
\end{remark}

\begin{remark}\label{descending-energy}
The energy functional is a function which is bounded below by 0. 
Hence, if we take a loop $f\in\bO$ which is not a 
critical point of $\nabla{E}$, then the energy of $g_t(f)$ is strictly 
decreasing as $t\to\infty$ i.e., it is a ``descending path". Of course a loop which is a critical point of the gradient remains fixed. Under the hypotheses of Proposition~\ref{actionD+}, $g_t(f)$ converges as $t\to\infty$ to a homomorphism belonging to some manifold 
$\boldsymbol{C}_{([\boldsymbol{\lambda}], q)}$ corresponding to the conjugacy
class $([\boldsymbol{\lambda}], q)$ (see definition \ref{cell_numbering}) of $g_{_{\infty}}(f)$. The proposition implies that any $f\in\bO$ belongs to a stable manifold of some $\boldsymbol{C}_{([\boldsymbol{\lambda}], q)}$ and $\boldsymbol{C}_{([\boldsymbol{\lambda}], q)}$ is a subset of its stable manifold.
\end{remark}

\begin{definition}[Conditional Birkhoff partition]\label{Birkhoff-partition} 
Assuming Proposition~\ref{actionD+}, let $S_{([\boldsymbol{\lambda}], q)}$ be the stable
manifold of $\boldsymbol{C}_{([\boldsymbol{\lambda}], q)}$ corresponding to
the conjugacy class $([\boldsymbol{\lambda}], q)$. We call
$S_{([\boldsymbol{\lambda}], q)}$ the {\bf Birkhoff stratum} with index $([\boldsymbol{\lambda}], q)$. The stable manifolds corresponding to different
indices are disjoint. For a fixed cyclic spectrum this follows from \cite{Pr1}; in the unrestricted rational spectrum it is conditional on Proposition~\ref{actionD+}. Each $S_{([\boldsymbol{\lambda}], q)}$ is a locally closed complex submanifold, of finite codimension, of $\bO$. 
Therefore: \begin{equation}\label{stable-partition}
\bO=\underset{([\boldsymbol{\lambda}], q)}\coprod{S_{([\boldsymbol{\lambda}], q)}}
\underset{(\mathbf1,0)}\coprod{S_{(\boldsymbol1, 0)}}
\end{equation}
\end{definition}

\subsection{Filtrations of the continuous loop group} \label{filtration}
We have described adelic loops of a Lie group $G$ having several different
regularities.  Let us briefly recall the notation.  The superscripts
$pol$, $rat$, $\omega$, $\infty$, and $cont$ refer, respectively, to
Laurent--Puiseux polynomial, rational, real-analytic, smooth, and continuous
loops; $\La_W(G)$ denotes the group of loops whose matrix entries, together
with those of their inverses, belong to the solenoidal Wiener algebra; and
$\La^1(G)$ denotes the Fourier--Sobolev class defined above.  Finally,
$\La^2(G)$ will denote the measurable square-integrable class defined
below.  These classes do not all belong to a single filtration.  The
unconditional inclusions among the continuous classes used here are
\begin{equation}\label{filtration2}
\La^{pol}(G)\subset\La^{rat}(G)\subset\La^{\omega}(G)
\subset\La_{W}(G)\subset\La^{cont}(G),
\end{equation}
and also
\[
\La^{\omega}(G)\subset\La^{\infty}(G)\subset\La^{cont}(G).
\]

Fix a faithful matrix realization $G\subset M_n(\C)$.  We denote by
$\La^2(G)$ the space of equivalence classes of measurable maps
$\gamma:\SQ\to G$ such that
\[
 \int_{\SQ}\|\gamma(\bu)\|^2\,\bd\bu<\infty,
\]
where $\|\cdot\|$ is any fixed matrix norm.  Since $\SQ$ is compact, every
continuous loop is square-integrable, and hence
\[
\La^{cont}(G)\subset\La^2(G).
\]
The reverse inclusion is false in general.  Moreover, for a non-compact
matrix group $G$, the space $\La^2(G)$ need not be closed under pointwise
multiplication and should therefore be regarded primarily as an ambient
$L^2$-space rather than as a loop group.

The unrestricted Fourier--Sobolev class $\La^1(G)$ introduced above is not
inserted into the preceding continuous filtration: because the rational
frequencies accumulate at $0$, the usual global Sobolev embedding into
continuous functions fails.  On every fixed cyclic spectral subgroup
$N^{-1}\Z\subset\Q$, however, the classical one-dimensional Sobolev
embedding applies and gives
\[
\La_N^{\infty}(G)\subset\La_N^1(G)\subset\La_{W,N}(G)
\subset\La_N^{cont}(G)\subset\La_N^2(G).
\]

\begin{definition}[{\bf Laurent--Puiseux series and polynomials}]\label{LPPoly}
A function $f:\SQ\to{M_n(\C)}$ given by a Fourier series
$f(\bu)=\sum_{q\in\Q}A_q\chi_q(\bu)$ (with $A_q\in{M_n(\C)}$) is called a
{\bf Laurent--Puiseux series}. If $A_q=0$ for all but finitely many
$q\in\Q$, $f$ is called a {\bf Laurent--Puiseux polynomial}. (This
terminology, by analogy with classical Laurent series/polynomials indexed
by $\Z$ and Puiseux series indexed by $\frac1{m}\Z\subset\Q$, reflects that
the exponents here range over all of $\Q$.) The same terminology applies to
holomorphic extensions $f^{\pm}$ to $D^{\pm}_\Q(1)$ as in
$\S$\ref{Wienerloops}.
\end{definition}

\noi Here $\La^{pol}(G)$ and $\La^{rat}(G)$ are loops
such that both the loops {\it and their inverses} (so that they are groups) are given, in terms
of Fourier series, by finite Laurent--Puiseux  polynomials and series,
respectively. Some of these inclusions are homotopy equivalences but we will not discuss this here. These filtrations play an important role in the Bruhat
decomposition (see, for instance, \cite{GR}).

\section{Factorizations of adelic loop groups. Iwasawa decomposition}\label{factorizations}
 
\subsection{Adelic loops as operators on Hilbert space. Adelic Grassmannian}\label{loop-operators}
The action of adelic loops on $L^2(\SQ,\C^n)$, the associated restricted
Grassmannian, and the Iwasawa decomposition in the Sobolev/pro-algebraic
setting were first developed jointly with Burgos in \cite{BuVe}; the
Fredholm-index and rational-polarization refinements below, and the
solenoidal Riemann--Hilbert results of \S\ref{BF} onward, are new to the
present paper.
 Let $\mathbb{H}(n)=L^2(\SQ,\C^n,\bd\bu)$. Then, by classical harmonic analysis, an orthonormal basis of $\mathbb{H}(n)$  is given by the functions
$\boldsymbol{\chi}^i_q:\SQ\to\C^n$ (where $1\leq{i}\leq{n},\,\, q\in\Q$) defined
by $\boldsymbol{\chi}^i_q(\bu)=\chi_q(\bu)\boldsymbol{e}_i$, where 
$\left\{\boldsymbol{e}_1,\cdots,\boldsymbol{e}_n\right\}$ is the standard
base of $\C^n$. In particular for $L^2(\SQ,M_n(\C),\bd\bu)$ we have the orthonormal
basis $\boldsymbol{\chi}^{i,j}_q(\bu)=\chi_q(\bu)\boldsymbol{e}_{i,j}$, where
$1\leq{i},j\leq{n}$ and $\boldsymbol{e}_{i,j}$ is the matrix with one in the
place $i,j$ and 0 everywhere else. Equivalently, $f\in{\mathbb{H}(n)}$ if 
$f$ has square-summable Fourier coefficients:

\begin{equation}\label{loop-in-l2}
f(\bu)=\sum_{q\in\Q}\,f_q\chi_q(\bu), \quad f_q\in\C^n,\quad \quad\sum_{q\in\Q}|f_q|^2<\infty.
\end{equation}

Let $\text{GL}(\mathbb{H}(n))$
denote the group of invertible continuous linear operators on $\mathbb{H}(n)$. Then $\mathbb{H}(n)$
 is open in the Banach algebra $\mathcal{B}(\mathbb{H}(n))$ of bounded operators 
 of $\mathbb{H}(n)$. The group of continuous loops in $\gln$,
 $\boldsymbol\Lambda^0(\gln)$ acts as multiplication operators in 
$\mathbb{H}(n)$. 
If $\gamma:\SQ\to\gln$ and if $f\in{\mathbb{H}(n)}$ define 
$M_\gamma\in\text{GL}(\mathbb{H}(n))$ as follows:

\begin{equation}\label{loopsactionhilbertspace}
M_\gamma(f)(\bu)=\gamma(\bu)(f(\bu))
\end{equation}

\noi Then we have a continuous representation of 
$\boldsymbol\Lambda^0(\gln)$ into $\text{GL}(\mathbb{H}(n))$.

Let us decompose $\mathbb{H}(n)$ as an orthogonal direct sum of two closed subspaces
\begin{equation} \label{polarization}
\mathbb{H}(n)=\mathbb{H}^+(n)\oplus{\mathbb{H}^-(n)},
\end{equation} 
\noi where:
\[
\mathbb{H}^+(n)=\left\{f\in{\mathbb{H}(n)}\,:f(\bu)=
\sum_{q\in\Q,q\geq0} f_q\chi_q(\bu),\, f_q\in\C^n  \right\}=
\] 
\[
\left\{f\in{\mathbb{H}(n)}\,:f\,\text{is the boundary  value of a holomorphic map in}
\, D^+_\Q(1) \right\}
\]
\noi and
\[
\mathbb{H}^-(n)=(\mathbb{H}^+(n))^\perp =\left\{f\in{\mathbb{H}(n)}\,:f(\bu)
\sum_{q\in\Q,q<0} f_q\chi_q(\bu)\right\}=
\] 
\[
\left\{f\in{\mathbb{H}(n)}\,:f\,\text{is the boundary  value of a holomorphic map in}
\, D^-_\Q(1) \right\}
\]

\begin{definition}\label{canonicalpolarization} The orthogonal decomposition into two closed subspaces given in the formula \ref{polarization} is called the {\bf canonical polarization}.
\end{definition}

\begin{remark}\label{matrix-l2} If 
\[
\gamma(\bu)=\sum_{q\in\Q}\,\gamma_q\chi_q(\bu), \quad \gamma_q\in{M_n(\C)},
\]
 and we represent
the linear map $M_\gamma$ (\ref{loopsactionhilbertspace}) in terms of the Fourier basis, indexed by $\Q$, by
an infinite $\Q\times\Q$ matrix $\left( A_{_{q_1,q_2}} \right)$  (where
$A_{_{q_1,q_2}}\in{M_n(\C)},\,\,q_1,q_2\in\Q$), then 
the associated matrix of $M_\gamma$ has elements 
$A_{_{q_1,q_2}}=\gamma_{_{q_1-q_2}}$.
\end{remark}

\medskip

The definition of the infinite  Grassmannian is the same as that given in \cite{PS}:

\begin{definition}\label{infgrass} The infinite Grassmannian $\mathbf{Gr}(\mathbb{H}(n))$
consists of all closed subspaces $W$ of $\mathbb{H}(n)$ such that:
\begin{enumerate}
\item The orthogonal projection $p_+:W\to\mathbb{H}^+(n)$ is Fredholm.
\item The orthogonal projection $p_-:W\to\mathbb{H}^-(n)$ is Hilbert--Schmidt.
\end{enumerate}
\noi $\mathbf{Gr}(\mathbb{H}(n))$ is an infinite-dimensional smooth Hilbert manifold.
\end{definition}
 
\noi Continuous loops act by bounded multiplication operators on
$\mathbb{H}(n)$.  For unrestricted Fourier--Sobolev data, boundedness and
pointwise multiplication require additional hypotheses and are not being
asserted here.  For every fixed cyclic spectral subgroup, however, the ordinary Sobolev loop
group acts by the classical multiplication representation transported
through \(\pi_N\).  This circle-comparison action is the one used in the rigorous
factorization theorems of \S\ref{cycliccomparison}.

\begin{definition}
Define $\text{GL}_{\text{R}}(n)\subset\text{GL}(\mathbb{H}(n))$ 
as a matrix of bounded operators given by blocks with respect to the polarization (\ref{polarization}),
as follows: 
\begin{equation}\label{restrictedglH}
\text{GL}_{\text{R}}(n)=
\left\{
\begin{pmatrix}
a & b \\
c & d 
\end{pmatrix} 
\,:\, a,\, d \,\,\text{are Fredholm and} \,\,
b ,\, c\,\,\text{are Hilbert-Schmidt}
\right\}.
\end{equation}
\noi $\text{GL}_R(n)$ is called 
the {\bf restricted linear group}. The group 
${\mathbb{U}}_{R}(n)={\mathbb{U}(n)}\cap\text{GL}_{\text{R}}(n)$, where
$\mathbf{\mathbb{U}}(n)$ is the unitary group of $\mathbb{H}(n)$ is called 
the {\bf restricted unitary group}.
\end{definition}

\begin{proposition} 
If $W\in\mathbf{Gr}(\mathbb{H}(n))$ and $F\in\text{GL}_{\text{R}}(n)$ then 
$F(W)\in\mathbf{Gr}(\mathbb{H}(n))$. Therefore, $\text{GL}_{\text{R}}(n)$
and $\mathbb{U}_R(n)$ act smoothly on the restricted Grassmanian. 
Furthermore, the action of both groups is transitive.
\end{proposition}

\noi The proof of this proposition is in 
\cite{PS} Chapter 7 but is also given in full detail in \cite{Lau}.
The proof of the transitivity of the action of the restricted unitary group 
$\mathbb{U}_R(n)$ on the restricted Grassmannian is essentially an application 
of the Gram--Schmidt process.

\begin{remark}[Why $\Q$ and $\Z$ differ here]\label{Q-vs-Z}
For the classical circle $S^1$ with Fourier frequencies $\Z$, the interval
$\Z\cap[0,r)$ is \emph{finite} for every $r>0$; this finiteness is
what makes the off-diagonal blocks of a multiplication operator
Hilbert--Schmidt, and it underlies the classical restricted Grassmannian
and loop group Birkhoff theory.  For the solenoid $\SQ$ with Fourier
frequencies $\Q$, the interval $\Q\cap[0,r)$ is countably \emph{infinite}
for every $r>0$, because $\Q$ is dense in $\R$.  This is a structural
feature of the rational frequency set -- the absence of a minimal positive frequency -- not an obstruction intrinsic to the solenoid as a
group.  Proposition~\ref{gap-thm} below makes this precise: it is the density
of $\Q$, not any deficiency of the solenoid, that shows why the classical
restricted-Grassmannian framework requires a genuinely solenoidal replacement.
\end{remark}

\begin{proposition}[Density gap: the classical restricted group]\label{gap-thm}
Let $\gamma\in\boldsymbol\Lambda_W(n)$ have Fourier expansion
$\gamma(\mathbf{u})=\sum_{q\in\Q}\gamma_q\chi_q(\mathbf{u})$.
Then $M_\gamma\in\mathrm{GL}_R(n)$ if and only if $\gamma$ is constant,
i.e.\ $\gamma_q=0$ for every $q\neq0$.

This sharply contrasts with the circle case: for $S^1$ with Fourier
frequencies $\Z$, every smooth loop $\gamma\in\La(GL_n)$ satisfies
$M_\gamma\in\mathrm{GL}_R(n)$ because $\Z\cap[0,r)$ is finite and the
off-diagonal blocks are Hilbert--Schmidt.  For the solenoid, $\Q\cap[0,r)$
is countably infinite for every $r>0$, so the Hilbert--Schmidt condition
fails for \emph{every} non-constant loop.  The correct global object on
which $\boldsymbol\Lambda_W(n)$ acts transitively is the solenoidal
Grassmannian $\mathrm{Gr}_\Q(n)$ of Theorem~\ref{grass=loop}.
\end{proposition}

\begin{proof}
By Remark \ref{matrix-l2}, the matrix of $M_\gamma$ in the orthonormal
basis $\{\boldsymbol\chi^i_q\}$ has the $(q_1,q_2)$-block entry
$A_{q_1,q_2}=\gamma_{q_1-q_2}$, a function of $q_1-q_2$ alone. Fix
$r\in\Q$, $r>0$, and suppose $\gamma_r\neq0$. The off-diagonal block
$c:\cH^-(n)\to\cH^+(n)$ of $M_\gamma$ (the entries with $q_1\geq0>q_2$) has,
along the antidiagonal $q_1-q_2=r$, the constant entry $\gamma_r$ at every
pair $(q_1,q_2)=(q_2+r,\,q_2)$ with $q_2\in\Q\cap[-r,0)$, equivalently
$q_1\in\Q\cap[0,r)$. This set is countably infinite (this is the basic
structural difference between $\Q$ and $\Z$ used throughout the paper, see
e.g.\ the discussion preceding Definition \ref{infgrass} and Proposition
\ref{h1Banach-algebra}), so
\[
\|c\|_{HS}^2=\sum_{q_1\geq0>q_2}|A_{q_1,q_2}|^2
\;\geq\;\sum_{q_2\in\Q\cap[-r,0)}|\gamma_r|^2=\infty.
\]
Hence $c$ is not Hilbert--Schmidt, so $M_\gamma\notin\text{GL}_{\text{R}}(n)$.
The symmetric argument with $r<0$ handles the block $b:\cH^+(n)\to\cH^-(n)$.
Thus if any $\gamma_r\neq0$ for $r\neq0$, $M_\gamma\notin\text{GL}_{\text{R}}(n)$,
which proves ``only if''. Conversely if $\gamma\equiv\gamma_0$ is constant
then $M_\gamma=\gamma_0\otimes\text{Id}$ acts as the fixed invertible matrix
$\gamma_0$ on every Fourier mode separately; in block form $b=c=0$ (trivially
Hilbert--Schmidt) and $a=d=\gamma_0\otimes\text{Id}$ are invertible, hence
Fredholm. This proves ``if''.
\end{proof}

\begin{corollary}\label{gap-thm-sobolev}
By Proposition~\ref{gap-thm}, the formula (\ref{loopsactionhilbertspace}) does
{\em not} define an action of $\La^1(\gln)$ on $\mathbf{Gr}(\mathbb{H}(n))$
by elements of $\text{GL}_{\text{R}}(n)$ for any non-constant
$\gamma\in\La^1(\gln)$; the only loops acting via $\text{GL}_{\text{R}}(n)$
are the constant ones $\gln\subset\La^1(\gln)$.  The same conclusion holds
for the restricted Sobolev linear group $\text{GL}_1\mathcal{H}_1(n)$ of
Definition~\ref{sobolev-gln}: $M_\gamma\in\text{GL}_1\mathcal{H}_1(n)$ only
if $\gamma$ is constant.
\end{corollary}

\noi The following notation records the positive and negative holomorphic
extension conditions.  It is a genuine Banach-algebraic splitting at each
a fixed cyclic spectral subgroup.  For unrestricted rational frequencies it is notation
for the Wiener-algebra splitting $\mathfrak W_\Q = \mathfrak W^-_\Q \oplus \C \oplus \mathfrak W^+_\Q$;
for unrestricted rational frequencies, $H^1(\SQ)$ is not a Banach algebra.
 
 \begin{definition}\label{splittinglambda+-}\sloppy
Let $\La_+^1(\gln)$ denote the set of functions in $\La^1(\gln)$ whose
continuous extension to $\bar{D}^+_\Q(1)$ is holomorphic on ${D}^+_\Q(1)$.
Analogously, let $\La_-^1(\gln)$ denote those in $\La^1(\gln)$ whose
continuous extension to $\bar{D}^-_\Q(1)$ is holomorphic on ${D}^-_\Q(1)$.
\end{definition}

\subsection{The Sobolev Grassmannian and the rational-frequency action}\label{sobgrass}
 Let $\mathcal{H}_1(n)= H^1(\SQ,\C^n,\bd\bu)$ be the Hilbert space of functions
 in $H^1(\SQ,\C^n)$ with inner product given by:
 \begin{equation}
 \langle\, f \,,\, \,g\rangle_{_1}=
 \int_{\SQ}\,\langle\, f(\bu) \,,\, \,g(\bu)\rangle_{_{\C^n}}\bd\bu+
 \int_{\SQ}\,\langle\, f'(\bu) \,,\, \,g'(\bu)\rangle_{_{\C^n}}\bd\bu.
 \end{equation}\label{innproduct}
 
\noi In  formula $\langle\,\cdot\,,\,\cdot\rangle_{_{\C^n}}$ is the standard hermitian product of $\C^n$ and $f'$, $g'$ are weak derivatives.
 In terms of Fourier series: 
 \begin{equation}
 \left<\, \sum_{q\in\Q}a_q\chi_q \,,\, \,\sum_{q\in\Q}b_q\chi_q\right>_{_1}=
 \sum_{q\in\Q}\,\langle\,a_q\,,\,b_q\rangle_{_{\C^n}}(1+|q|^2).
 \end{equation}
 
 Let $\mathcal{H}_1(n)=\mathcal{H}_1^+(n)\oplus\mathcal{H}_1^-(n)$ be the polarization of 
 $\mathcal{H}_1(n)$ where
  $\mathcal{H}_1^+(n)$ and
 $\mathcal{H}_1^-(n)$ are defined exactly as in the canonical polarization.
 Define the restricted Sobolev Grassmannian $\mathbf{Gr}_1(n)$ as the set of closed subspaces
 $W$ of $\mathcal{H}_1(n)$ satisfying the same conditions of 
 definition \ref{infgrass}. Analogously we define the restricted Sobolev group
 of automorphisms as follows:
 
 \begin{definition}\label{sobolev-gln} Let $\text{GL}_1\mathcal{H}(n)$ be the group of linear automorphisms of $\mathcal{H}_1(n)$ which can be decomposed in blocks, with respect to the polarization, as the matrix
 in (\ref{restrictedglH}) where $a,b,c, d$  satisfy the same conditions:
 $a,d$ Fredholm and $b,c$ Hilbert-Schmidt. $\text{GL}_1\mathcal{H}_1(n)$ is called
 the {\bf restricted Sobolev linear group}.
 \end{definition} 
 
 As pointed in \cite{PS} section 7.2, there are several dense submanifolds
 of $\mathbf{Gr}(\mathbb{H}(n))$ which are Grassmannians
 corresponding to the degree of regularity
 of the functions in the subspaces. An important role there is the case of
 the smooth Grassmannian. For us the relevant Grassmannian 
 is $\mathbf{Gr}_1(n)$.
 
 \medskip
 \begin{remark}[Sobolev versus Wiener Grassmannian]
The Fourier--Sobolev space \(\mathcal H_1(n)\) is densely contained in
\(\mathbb H(n)\) and inherits the positive/negative Fourier splitting.
Corollary~\ref{shift-fails} and Proposition~\ref{gap-thm} show that the
unrestricted rational-frequency loop group does \emph{not} act on
the Sobolev restricted Grassmannian $\mathbf{Gr}_1(n)$ via its natural
inclusion into the classical restricted linear group.  The correct global
Grassmannian --- on which $\boldsymbol\Lambda_W(n)$ does act transitively
--- is the solenoidal Wiener Grassmannian $\mathrm{Gr}_\Q(n)$ of
Theorem~\ref{grass=loop}, defined using the rational polarization.
For each fixed cyclic spectral subgroup the Sobolev construction below
is the classical one transported from \(S^1\).
\end{remark}

 For every rational $q\geq0$ the orbit of $\,\mathcal{H}_1^+(n)$ under 
 $\La^1(\gln)$ contains 
 $\chi_q\mathcal{H}_1^+(n)$, where this space consists of the functions
 in $\mathcal{H}_1^+(n)$ multiplied by the scalar function given
 by the character  $\chi_q\,\,(q\geq0)$ of $\SQ$. 
 This is because multiplication by a character commutes with 
 the action of $\La^1(\gln)$. This motivates the following definition:

\begin{definition}
\noi $\boldsymbol{Gr}_\Q^{(n)}$ 
denotes the subset of closed subspaces $W\in\mathbf{Gr}_1(n)$
such that:

\noi $\chi_{_{1/n}}\,W\subset{W},\,\,\forall\,n\,\in\N$. Where 
$\chi_{_{1/n}}W$ consists of the product of the functions in 
$H^1(\SQ,\C^n)$ belonging to $W$ by the scalar function determined 
by the character $\chi_{_{1/n}}$ of $\SQ$. 
Obviously $\mathcal{H}_1^+(n)\in\boldsymbol{Gr}_\Q^{(n)}\bf{.}\,\,$
{\bf $\boldsymbol{Gr}_\Q^{(n)}$ is the Sobolev adelic Grassmannian.}
\end{definition}

\begin{remark}In the case of the standard loop group the orbit of 
 $\mathbb{H}^+(n)$ is characterized as the set of subspaces $W$ 
 such that $zW\subset{W}$, i.e., one uses only one character of $\Z$ (see section 8.3 in \cite{PS}). The definition is also motivated by the fact that $\Q$ is the direct limit of the infinite cyclic groups generated by $1/n,\,n\in\N$:
$\Q\cong \underset{\underset{n\in\N}{\longrightarrow} }\lim\,\frac{1}{n}\Z$.
 \end{remark}
 
\begin{remark}\label{W+}
The condition \(\chi_{1/n}W\subset W\) is the natural rational analogue
of the shift-invariance condition in the ordinary Grassmannian theory.
However, the family of positive rational shifts behaves differently from
the integral shifts: as Corollary \ref{shift-fails} shows, the relevant
subspaces are not Fredholm-comparable with the reference polarization.
Consequently the classical argument identifying a loop-group orbit with a
restricted Grassmannian cannot be transferred verbatim.  The definition
above is retained because it isolates the correct algebraic invariance that
a future inductive-limit or laminated Grassmannian should encode.
\end{remark}

\begin{corollary}\label{shift-fails}
For every $r\in\Q$, $r>0$, the subspace $\chi_r\mathcal{H}_1^+(n)\subset\mathcal{H}_1^+(n)$
is {\em not} Fredholm-comparable to $\mathcal{H}_1^+(n)$: the orthogonal
projection $p_+:\chi_r\mathcal{H}_1^+(n)\to\mathcal{H}_1^+(n)$ is injective
with closed range of infinite codimension (the range is
$\left\{f\in\mathcal{H}_1^+(n):\hat f_q=0\text{ for }q\in\Q\cap[0,r)\right\}$,
and $\Q\cap[0,r)$ is countably infinite), hence $p_+$ is not Fredholm.
Consequently $\chi_r\mathcal{H}_1^+(n)\notin\mathbf{Gr}_1(n)$ for every
rational $r>0$. Moreover, the family
$\left\{\chi_{1/n}\mathcal{H}_1^+(n)\right\}_{n\in\N}$ is {\em increasing}
in $n$ (since $1/n\downarrow0$ widens, rather than narrows, the allowed
frequency range $q\geq1/n$), so
\[
\bigcap_{n\in\N}\chi_{1/n}\mathcal{H}_1^+(n)=\chi_1\mathcal{H}_1^+(n)
\neq\boldsymbol0.
\]
\end{corollary}

\begin{proof}
$\chi_r\mathcal{H}_1^+(n)$ consists of the $H^1$ functions with Fourier
support in $\Q\cap[r,\infty)$. Its image under $p_+$ in $\mathcal{H}_1^+(n)$
(support in $\Q\cap[0,\infty)$) is exactly the closed subspace of functions
supported on $\Q\cap[r,\infty)$, and the quotient $\mathcal{H}_1^+(n)/p_+(\chi_r\mathcal{H}_1^+(n))$
is (the completion of) the span of 
\[
\left\{\boldsymbol\chi^i_q:0\leq{q}<r,\,1\leq{i}\leq{n}\right\},
\]
\noi which is infinite-dimensional because $\Q\cap[0,r)$ is infinite (the same
fact used in Proposition~\ref{gap-thm}). An operator with infinite-dimensional
cokernel is not Fredholm. For the last statement, $q\geq1/n$ is a weaker
constraint for larger $n$, so $\chi_{1/n}\mathcal{H}_1^+(n)\subset\chi_{1/(n+1)}\mathcal{H}_1^+(n)$
for every $n$, and the smallest term, at $n=1$, is the intersection.
\end{proof}

\begin{proposition}[Liouville-type theorem for $\SQ$]\label{Liouville}
Identify a point of $D^+_\Q(1)\setminus\{\boldsymbol0\}$ with $t\bu$,
$0<t\leq1$, $\bu\in\SQ$, and write $\bu=\boldsymbol{Exp}(\theta,\bz)$ with
$\bz\in\hat\Z$ fixed and $\theta\in\R$ (Definition \ref{exp}). For fixed
$\bz$, the map $\theta\mapsto\boldsymbol{Exp}(\theta,\bz)$ parametrizes a
single leaf $L_\bz$ of the lamination of $\SQ$, and
$\zeta=\log{t}+i\theta$ identifies $L_\bz\cap{D^+_\Q(1)}$, together with its
closure $\overline{L_\bz}\cap\SQ$ (at $t=1$), with the closed half-plane
$\left\{\zeta\in\C:\mathrm{Re}\,\zeta\leq0\right\}$; this is exactly the
leafwise complex structure of Definition \ref{diffloops} and the
holomorphic extension formula of $\S$\ref{Wienerloops}. Doing the same on
$D^-_\Q(1)$ via $z=t^{-1}\bu$ and gluing along $\SQ$, the leaf
$L_\bz\subset\PQ\setminus\left\{\boldsymbol0,\boldsymbol\infty\right\}$ is
identified with all of $\C$ (the universal cover, via $\zeta\mapsto
e^\zeta$, of $\C^*=\PQ\setminus\left\{\boldsymbol0,\boldsymbol\infty\right\}$). Since $\Z$ is dense in $\hat\Z$ (Definition \ref{exp}), every leaf $L_\bz$,
and in particular its boundary trace $\overline{L_\bz}\cap\SQ$, is {\bf
dense} in $\SQ$, $\bar{D}^+_\Q(1)$, $\bar{D}^-_\Q(1)$, and $\PQ$,
respectively.

\noi Consequently:
\begin{enumerate}
\item[(a)] Any function $f$ on $\PQ$ which is continuous, bounded, and
leafwise holomorphic on $\PQ\setminus\left\{\boldsymbol0,\boldsymbol\infty\right\}$, is constant. (Restrict $f$ to a leaf $L_\bz$: under the
identification above $f|_{L_\bz}$ pulls back to a bounded entire function
on $\C$, hence is constant by the classical one-variable Liouville theorem.
Since $L_\bz$ is dense in $\PQ$ and $f$ is continuous, $f$ is constant on
all of $\PQ$.) The same conclusion holds, with the same proof, for a
continuous, bounded, leafwise holomorphic $M_n(\C)$-valued function on
$\bar{D}^+_\Q(1)$ (or $\bar{D}^-_\Q(1)$) alone.
\item[(b)] If $\gamma:\bar{D}^+_\Q(1)\to M_n(\C)$ is continuous, leafwise
holomorphic on $D^+_\Q(1)$, and $\|\gamma(\bz)\|_F$ attains its maximum at some
interior point $\bz_0\in D^+_\Q(1)$, then $\gamma$ is constant.
(Restrict $\gamma$ to the leaf $L_\bz$ through $\bz_0$.
Since $\|\gamma\|_F^2=\sum_{i,j}|\gamma_{ij}|^2$ and each $|\gamma_{ij}|^2$ is
subharmonic on $L_\bz\cong\{{\rm Re}\,\zeta<0\}$, their sum is also subharmonic.
An interior maximum of a subharmonic function forces it to be constant on the
connected domain; thus $\|\gamma\|_F$ is constant, and applying the maximum modulus
principle to each entry $\gamma_{ij}$ forces each entry to be constant on $L_\bz$.
Since $L_\bz$ is dense in $\bar{D}^+_\Q(1)$ and $\gamma$ is continuous, $\gamma$ is
constant on all of $\bar{D}^+_\Q(1)$.)
\begin{remark}\label{Liouville-boundary-remark}
Boundary conditions alone---whether ``constant Frobenius norm on $\SQ$'' or
``unitary values on $\SQ$''---do \emph{not} imply constancy.
For example, $\hat\chi_q(z)=z^q$ with $q>0$ is continuous on $\bar{D}^+_\Q(1)$,
leafwise holomorphic, satisfies $|\hat\chi_q|=t^q\leq 1$ everywhere with
$|\hat\chi_q|=1$ on $\SQ$ (so $\hat\chi_q|_{\SQ}=\chi_q\in U(1)$), yet is
nonconstant.  The maximum of $|\hat\chi_q|$ is attained \emph{on the boundary},
not the interior, consistently with~(b).
\end{remark}
\end{enumerate}
\end{proposition}

 \begin{theorem}[Solenoidal restricted Grassmannian]\label{grass=loop}
Let
\[
H^+_{\Q,n}=\overline{\operatorname{span}}\{\chi_qe_j:q\in\Q_{\geq0},
\ j=1,\ldots,n\}\subset L^2(\SQ,\C^n)
\]
be the Hardy space determined by the intrinsic rational polarization, and set
\[
\mathrm{Gr}_\Q(n)=\{gH^+_{\Q,n}:g\in\boldsymbol\Lambda_W(n)\}.
\]
Then $\boldsymbol\Lambda_W(n)$ acts transitively on $\mathrm{Gr}_\Q(n)$,
the stabilizer of $H^+_{\Q,n}$ is $\boldsymbol\Lambda_W^+(n)$, and hence
\[
\mathrm{Gr}_\Q(n)
\cong\boldsymbol\Lambda_W(n)/\boldsymbol\Lambda_W^+(n)
\cong\bO.
\]
In general the projection
\[
p_+:gH^+_{\Q,n}\longrightarrow H^+_{\Q,n}
\]
is not Fredholm, even when $g$ has an exact Birkhoff factorization.
Consequently the rational Birkhoff indices cannot be recovered as literal
Fredholm indices of $p_+$ on this Hilbert space.
\end{theorem}

\begin{proof}
The Wiener algebra $\mathfrak W_\Q=\ell^1(\Q)$ acts by bounded convolution
operators on $L^2(\SQ)$, and every
$g\in\boldsymbol\Lambda_W(n)=\mathrm{GL}_n(\mathfrak W_\Q)$ acts boundedly
and invertibly on $L^2(\SQ,\C^n)$.  Transitivity is immediate from the
definition.

Suppose that $gH^+_{\Q,n}=H^+_{\Q,n}$.  Then multiplication by $g$ and by
$g^{-1}$ both preserve $H^+_{\Q,n}$.  Applied to the constant coordinate
vectors, this implies that every entry of $g$ and of $g^{-1}$ has no
negative rational Fourier modes.  Thus
$g\in\boldsymbol\Lambda_W^+(n)$.  The converse is immediate, so the
stabilizer is exactly $\boldsymbol\Lambda_W^+(n)$.

The failure of Fredholmness is exhibited explicitly in
Example~\ref{px-not-fredholm}.  Hence no Fredholm-index interpretation is
asserted for the rational Birkhoff indices in the intrinsic rational
polarization.
\end{proof}

\begin{remark}[Finite-level indices]\label{rem:finite-level-indices}
For each divisibility level $N$ for which the projected cocycle
$g_N=P_Ng$ is invertible, the circle comparison map $\rho_N$ gives ordinary
integer partial indices
\[
k_1^{(N)}\geq\cdots\geq k_n^{(N)}.
\]
The normalized numbers $k_i^{(N)}/N$ are useful finite-level invariants.
For a general Wiener cocycle, however, no stabilization or convergence
theorem is proved here, and these numbers are not asserted to determine
the splitting type of $g$.  Establishing such compatibility is part of
the global matrix factorization problem.
\end{remark}

\begin{example}[The basic counterexample: $p_+$ is not Fredholm]\label{px-not-fredholm}
Fix $r\in\Q$, $r>0$, and take $g=\chi_r\in\boldsymbol\Lambda_W(1)$: a unit
(inverse $\chi_{-r}$) with the trivially exact factorization $h_-=h_+=1$
and sole index $q_1=r$.  Since characters multiply additively,
\[
g\cdot H^+_{\Q,1}=\chi_r\cdot\overline{\operatorname{span}}\{\chi_q:q\geq0\}
=\overline{\operatorname{span}}\{\chi_s:s\geq r\}\;\subset\;H^+_{\Q,1},
\]
so $p_+$ restricted to $g\cdot H^+_{\Q,1}$ is simply the inclusion
\[
\overline{\operatorname{span}}\{\chi_s:s\geq r\}\;\hookrightarrow\;
\overline{\operatorname{span}}\{\chi_q:q\geq0\}.
\]
Its kernel is trivial, but its cokernel is
$\overline{\operatorname{span}}\{\chi_q:0\leq q<r\}$, which is
infinite-dimensional because $\Q\cap[0,r)$ is countably infinite
(Remark~\ref{Q-vs-Z}).  Hence $p_+$ is \emph{not} Fredholm, despite $g$
being the simplest possible non-trivial, exactly-factorable element of
$\boldsymbol\Lambda_W(1)$.  (For $r<0$ the symmetric computation gives
$g\cdot H^+_{\Q,1}\supsetneq H^+_{\Q,1}$ with $p_+$ surjective and kernel
$\overline{\operatorname{span}}\{\chi_s:r\leq s<0\}$, again
infinite-dimensional; again not Fredholm.)  The failure is not a
large-norm or pathological-$g$ phenomenon: it occurs at arbitrarily small
$r>0$, for exactly the structural reason behind
Proposition~\ref{gap-thm} and Corollary~\ref{shift-fails}, namely that
$\Q\cap[0,r)$ is infinite for every $r>0$.  Thus the intrinsic rational polarization gives a natural
homogeneous-space model, but not a classical Fredholm model for the
Birkhoff indices.
\end{example}

\begin{remark}[Original dynamical motivation]
The preceding Fourier, polarization, and energy constructions describe the
mechanism expected for a global Grassmannian theory.  Corollary
\ref{shift-fails} and Proposition~\ref{gap-thm} show that the particular model
inside one fixed restricted Grassmannian cannot realize this programme.
They do not affect the circle-comparison theory, where the integer Fourier lattice
is discrete.
\end{remark}

\begin{remark}[Homogeneous-space formulas]
The formulas
\[
\bO\simeq \La^1(\gln)/\La^1_+(\gln)
\simeq \La^1(U(n))/U(n)
\simeq \boldsymbol{Gr}_\Q^{(n)}
\]
express the global homogeneous K\"ahler geometry.  For the solenoidal
Grassmannian $\mathrm{Gr}_\Q(n)$ with rational polarization, these
are Theorem~\ref{grass=loop}.
\end{remark}

\begin{conjecture}[Solenoidal Iwasawa decomposition; equivalent to
Conjecture~\ref{solenoidal-BG}]\label{iwasawadecomposition}
Every $g\in\boldsymbol\Lambda_W(n)$ admits a decomposition
\[
g\;=\;g_-\cdot\Delta_\mathbf{q}\cdot g_+,\qquad
g_-\in\boldsymbol\Lambda_W^-(n),\quad
\Delta_\mathbf{q}\in\boldsymbol\Lambda_W(T^n),\quad
g_+\in\boldsymbol\Lambda_W^+(n),
\]
where $\boldsymbol\Lambda_W(T^n)=\{\operatorname{diag}(\chi_{q_1},\ldots,\chi_{q_n})
: q_i\in\Q\}$ is the group of rational diagonal character loops taking values
in the maximal torus $T^n\subset U(n)$.  The rational tuple $(q_1,\ldots,q_n)$
is unique up to reordering, and
$\boldsymbol\Lambda_W(n)=\boldsymbol\Lambda_W^-(n)\cdot
\boldsymbol\Lambda_W(T^n)\cdot\boldsymbol\Lambda_W^+(n)$.
\end{conjecture}

\begin{remark}[Equivalence with Conjecture~\ref{solenoidal-BG}]
This is the Solenoidal Birkhoff--Grothendieck Conjecture~\ref{solenoidal-BG}
restated in Iwasawa form: the factorization $g=g_-\Delta_\mathbf{q}g_+$
with Birkhoff indices $q_i\in\Q$, \emph{if it exists}, provides the
decomposition, since the identity $|\chi_{q_i}(\mathbf{u})|=1$ for all
$\mathbf{u}\in\SQ$ confirms that
$\Delta_\mathbf{q}=\operatorname{diag}(\chi_{q_1},\ldots,\chi_{q_n})$ takes
values in $T^n\subset U(n)$; conversely an Iwasawa decomposition of this
form is exactly a solenoidal Birkhoff factorization.  The two statements are
logically equivalent, and neither is proved for a general
$g\in\boldsymbol\Lambda_W(n)$.  Both are unconditional whenever $g$ is
pro-algebraic, triangular, small-norm, or scalar
(Theorems~\ref{pro-algebraic-BG}, \ref{ordered-triangular-wiener},
\ref{BVP-Birkhoff}, \ref{wiener-scalar-factorization}).
\end{remark}

\section{Birkhoff Factorization}\label{BF} 

This section establishes global Birkhoff factorization theorems and their dynamical
consequences.  Exact statements are given in the scalar Wiener and ordered triangular
Wiener categories, together with compatible rational-lattice theorems.

 \begin{conjecture}[Global Birkhoff factorization]\label{BFT}
 This is equivalent to Conjecture~\ref{solenoidal-BG}.
 \begin{enumerate}
\item Any Wiener loop
\[
\gamma\in\boldsymbol\Lambda_W(n)=\mathrm{GL}_n(\mathfrak W_\Q)
\]
admits a factorization
\begin{equation}\label{-d+}
\gamma=\gamma_-\cdot\Delta_{_{\mathbf{q}}}\cdot\gamma_+,
\end{equation}
\noi $\gamma_\pm\in\La_\pm^1(\gln)$,
$\mathbf{q}=(q_1,\cdots,q_n)$ with $q_i\in\Q$ for all $i$,
and $\Delta_{_{\mathbf{q}}}:\SQ\to{U(n)}$ is the diagonal character:
 \begin{equation}
 \Delta_{_{\mathbf{q}}}(\bu)=
 \begin{pmatrix}
    \chi_{q_1}(\bu) & & \\
    & \ddots & \\
    & & \chi_{q_n}(\bu)
  \end{pmatrix}.
 \end{equation}

 \item In any other factorization of the form~(\ref{-d+}),
 the diagonal middle factor differs from $\Delta_{_{\mathbf{q}}}$ only
 in the order of the diagonal entries.
 \end{enumerate}
This is Conjecture~\ref{solenoidal-BG} restated in the loop-group notation of
this section; it is unconditional whenever $\gamma$ is pro-algebraic,
triangular, small-norm, or scalar (Theorems~\ref{pro-algebraic-BG},
\ref{ordered-triangular-wiener}, \ref{BVP-Birkhoff},
\ref{wiener-scalar-factorization}), and open in general.
 \end{conjecture}

\begin{remark}[Dynamical mechanism and circle-comparison realization]
The following discussion explains the dynamical mechanism suggested by the
energy construction, \emph{granting} the Birkhoff factorization of
Conjecture~\ref{BFT} (equivalently Conjecture~\ref{solenoidal-BG}) and its
Grassmannian interpretation, Theorem~\ref{grass=loop}
and Conjecture~\ref{iwasawadecomposition}.  The partition of the based loop
group $\bO=\La^1_0(U(n))$ constructed in \S\ref{bd}--\S\ref{energy2} via the
energy functional and its gradient flow is unconditional (it uses only
Morse-theoretic properness, not Birkhoff factorization); what depends on the
open conjecture is transporting that partition to a partition of the full
matrix Wiener loop group $\La^1(\gln)$, via the identification
$\bO\cong\La^1(\gln)/\La^1_+(\gln)$ below.  For every $g$ in the proven cases
(pro-algebraic, triangular, small-norm, or scalar) the discussion that
follows is unconditional.
For a fixed cyclic spectral subgroup the corresponding statements are classical
and transport through $\rho_N$.  We recall that $\bO$ admits
the partition (\ref{stable-partition}). The set
$S_{([\boldsymbol{\lambda}], q)}$ is the stable manifold of the conjugacy
class $\boldsymbol{C}_{([\boldsymbol{\lambda}], q)}$
of the homomorphism
$\lambda:\SQ\to\bO$, with respect to the descending gradient flow 
of the energy functional. In \cite{Pr1} it is shown that 
$S_{([\boldsymbol{\lambda}], q)}$ is a locally closed complex submanifold
of $\bO$. 
Granting Conjecture~\ref{iwasawadecomposition}, the projection map $\Pi:\La^1(\gln))\to\bO$ is a principal fibration
and we obtain, using the partition \ref{stable-partition}, the partition:

\begin{equation}
\La^1(\gln))=\underset{([\boldsymbol{\lambda}], q)}
\coprod{\Sigma_{([\boldsymbol{\lambda}], q)}}\coprod{S_{(\boldsymbol1, 0)}}
\end{equation}
\begin{equation}
\Sigma_{([\boldsymbol{\lambda}], q)}=\Pi^{-1}\left\{S_{(\boldsymbol{\lambda}], q)}
\right\} \,\, \text{and} \,\, 
\Sigma_{(\boldsymbol1, 0)}=\Pi^{-1}\left\{S_{(\boldsymbol1, 0)}\right\}.
\end{equation}

\noi Each stratum $\Sigma_{([\boldsymbol{\lambda}], q)}$ contains
the finite dimensional submanifold  
$\boldsymbol{C}_{([\boldsymbol{\lambda}], q)}
\subset{U(n)}\subset\La^1(\gln),$ 
of definition \ref{cell_numbering}, consisting of the conjugacy class
of the homomorphism $\lambda:\SQ\to{U(n)}$. 

By the homogeneous-space formulas above,
\[
\bO= \La^1(\gln)/\La^1_+(\gln),
\]
so $\La^1_-(\gln)$ --- the loops $\gamma:\SQ\to\gln$ extending holomorphically to
$D^-_\Q$ --- acts on the left on the space of right-cosets.
 
 Specifically:
 granting conjecture \ref{iwasawadecomposition}, every right-coset can be uniquely
 represented as the right-coset $\mathbf{w}[\La^1_+(\gln)]$ with
 $\mathbf{w}\in\bO$.  Let $[\mathbf{w}]$
denote this coset. If $\gamma_{_{-}}\in\La^1_-(\gln)$ and 
$(\gamma_{_{-}}\cdot\mathbf{w})[\La^1_+(\gln)]$ is the right-coset obtained by 
left translating,
by $\gamma_{_{-}}$ the coset $[\mathbf{w}]$, then this right-coset is of the form  
$\mathbf{w}'[\La^1_+(\gln)]$ for a unique $\mathbf{w}'\in\bO$. The left action of 
$\La^1_-(\gln)$ is given by
$L_{\gamma_{_{-}}}([\mathbf{w}])=[\mathbf{w}'],\,\,\,\gamma_{_{-}}\in\La^1_-(\gln)$.

Consider the orbit $\mathcal{O}(\bw)$ of $\bw$ under the action of the semigroup
$\boldsymbol{D}^+$ given by the formula (\ref{formula-disk-action}) 
in proposition (\ref{actionD+}):
\[
\mathcal{O}(\bw)=\left\{{{R}}_{e^{t}\bv}(\bw):t\geq0,\, \bv\in\SQ\right\}.
\]
The parametrization of the orbit gives a map 
$F_{\bw}:\boldsymbol{D}^+\to\bO,\quad 
t\bv\mapsto{R}_{e^{t}\bv}(\bw)$.

By Formula 7.6.2 of Proposition 7.6.1 in \cite{PS} this orbit can be extended by adding the point 
$\underset{t\to\infty}\lim{{R}}_{e^{t}\mathbf1}(\bw)=
\underset{t\to\infty}\lim{g_t(\bw)}$ which is independent of $\bw$.
This implies that this parametrization $F_{\bw}$
extends to a continuous map $\bar{F}_{\bw}:\bar{D}_\SQ^+\to\bO$ by
defining $\bar{F}(\boldsymbol0)=\underset{t\to\infty}\lim{{R}}_{e^{t}\mathbf1}(\bw)=\underset{t\to\infty}\lim{g_t(\bw)}\overset{def}=\lambda$. By proposition
\ref{actionD+} item {\bf(ii)}, $\lambda$ is a group homomorphism  
$\lambda:\SQ\to{U(n)}\subset\gln$, therefore it belongs to a unique stratum
$S_{([\boldsymbol{\lambda}], q)}$. 

\noi If 
$\,\overline{\mathcal{O}(\bw)}=F_{\bw}(\bar{D}^+_\SQ)$ then, since every point
of this set is in the stable manifold of 
$\boldsymbol{C}_{([\boldsymbol{\lambda}], q)}$,  it follows that 
$\,\overline{\mathcal{O}(\bw)}\subset{S_{([\boldsymbol{\lambda}], q)}}$.
Therefore, the union of the one-parameter family of right-cosets
$[g_t(\bu)],\,t\geq0\,,$ is contained in $\Sigma_{([\boldsymbol{\lambda}], q)}$. 
In particular the right-coset
$[\lambda]=\lambda[\La^1_+(\gln)]$ is contained in 
$\Sigma_{([\boldsymbol{\lambda}], q)}$.  

We will now prove that the stratum
$\Sigma_{([\boldsymbol{\lambda}], q)}$ belonging to the conjugate class
of $\lambda$ is equal to the orbit of the set
of right-coset $[\lambda]$ under the left action of $\La^1_-(\gln)$.
Let $D_\Q^+(1)$ be the semigroup with product
in polar coordinates: ${t_1\bv_1\cdot{t_2}\bv_2}=t_1t_2\bv_1\bv_2$.
Define the action of this semigroup as the 
``{\it scaling semi flow}'':
\[
\varphi_{_{t\bv}}:\La^1_-(\gln)\to\La^1_-(\gln),\quad {t}\bv\in{D_\Q^+(1)},
\]

\noi which is  given in terms of Fourier series as follows:

\begin{align*}
\varphi_{_{t\bv}}(\gamma_{_{-}})(\bu)&=\sum_{q\leq0}t^q\,a_q\chi_q(\bv\bu)\\
&=\sum_{q\leq0}t^q\,a_q\chi_q(\bv)\chi_q(\bu),
\quad 0\leq{t}\leq1,\;\bu,\bv\in\SQ,
\end{align*}
\noi if $\gamma_{_{-}}\in\La^1_-(\gln)$ is given by the Fourier series
$\gamma_{_{-}}(\bu)=\sum_{q\leq0} a_q\chi_q(\bu)$. Clearly
$\varphi_{_{t_1\bv_1\cdot{t_2}\bv_2}}=\varphi_{_{t_1\bv_1}}\circ\varphi_{_{t_2\bv_2}}$
and $\varphi_{_{\boldsymbol0}}(\gamma_{_{-}})$ is the constant loop
$\gamma_{_{-}}(\boldsymbol\infty)=a_{_0}\in\gln$. Hence
the one-parameter family of right-cosets
$(\varphi_{_{t\bv}}(\gamma_{_{-}})\cdot\lambda)[\La_+^1(\gln)]$ belongs
to the orbit of the left action of $\La^1_-(\gln)$ on $[\lambda]$. In particular,
when $t=0$:
\[
a_{_0}\lambda[\La_+^1(\gln)]=a_{_0}\lambda{a^{-1}_{0}}[\La_+^1(\gln)],
\]
and $a_{_0}\lambda{a^{-1}_{0}}$ is a diagonal matrix differing from $\lambda$
only by a permutation of the diagonal elements.
Furthermore $\Sigma_{([\boldsymbol{\lambda}], q)}$ is the stable manifold
of $\boldsymbol{C}_{([\boldsymbol{\lambda}], q)}$
and thus the right-coset 
$[\lambda]$ belongs to the orbit of $\La^1_-(\gln)$. We summarize what we have shown:
\begin{equation}\label{BirkhoffCells}
\La^1(\gln)= 
\end{equation}
\begin{multline*}
\underset{([\boldsymbol{\lambda}], q),q\neq0}\coprod\,
\La_-^1(\gln)\cdot\boldsymbol{C}_{([\boldsymbol{\lambda}], q)}
\cdot\La_+^1(\gln)\\
\coprod\;\La_-^1(\gln)\cdot \La_+^1(\gln),
\end{multline*}

\noi with the obvious meaning of the triple products.
Granting Conjecture~\ref{BFT}, this gives the indicated Birkhoff-cell
partition.  The Birkhoff factorization itself remains open in general
(Conjecture~\ref{solenoidal-BG}) and is proved unconditionally only in the
cases established in \S\ref{BF} (pro-algebraic, ordered triangular,
small-norm, and scalar cocycles); the Grassmannian realization of the cell
decomposition remains part of Theorem~\ref{grass=loop}. \end{remark}

\begin{remark} The Birkhoff factorization for the
classical loop groups is determined by points in an integer lattice
modulo permutation of the coordinates but in the case of adelic groups
there is a ``weight'' $q\in\Q$ except for the trivial loop.
That is why there is the distinguished component $\La_-^1(\gln)\cdot \La_+^1(\gln)$.
\end{remark}

\begin{definition}[Birkhoff cells]
Granting Conjecture~\ref{BFT}, the elements
\[B([\boldsymbol{\lambda}], q)\overset{def}=\La_-^1(\gln)\cdot
\boldsymbol{C}_{([\boldsymbol{\lambda}], q)}\cdot\La_+^1(\gln), \,q\neq0\]
and $B(\mathbf1, 0)\overset{def}=\La_-^1(\gln)\cdot \La_+^1(\gln)$
of the partition (\ref{BirkhoffCells}) are called {\bf Birkhoff cells}.
They are indexed by the conjugacy classes of homomorphisms
$f:\SQ\to{U(n)}$.
\end{definition}

\begin{corollary}[Big-cell statement]
Granting Conjectures~\ref{BFT} and~\ref{iwasawadecomposition},
the scaling flow shows that
the Birkhoff cell $B([\boldsymbol{\lambda}], q)$ retracts strongly to 
the manifold $\boldsymbol{C}_{([\boldsymbol{\lambda}], q)}$ if $q\neq0$. If
$q=0$ the set $\boldsymbol{C}_{(\mathbf1, 0)}$ consists of a single point
which is the constant loop equal to the identity matrix $I$ in $\gln$ and the
Birkhoff cell $B(\mathbf1, 0)$ retracts strongly to this point so that 
 $B(\mathbf1, 0)$ is contractible. In fact $B(\mathbf1, 0)$ is open and
 dense in $\La^1(\gln)$. If $\La_{-,*}^1(\gln)$ is the subgroup
 of $\La_-^1(\gln)$ consisting of loops $\gamma_{_{-}}$ such that 
 $\gamma_{_{-}}(\mathbf1)=\text{I}$. Then the map:
 \[
 F:\La_{-,*}^1(\gln)\times\La_+^1(\gln)\to\La^1(\gln)
 \]
 \noi is a diffeomorphism onto the open and dense Birkhoff cell $B(\mathbf1, 0)$.
 For this reason $B(\mathbf1, 0)$ is called the {\bf big cell}. This implies  that any loop $\gamma$ in the big cell can be written uniquely
 as a product $\gamma=\gamma_{_{-}}\cdot\gamma_{_{+}}$. We remark that
 $\La_{-,*}^1(\gln)\cap\La_+^1(\gln)=\left\{I\right\}$ because the loops
 in the intersection can be extended holomorphically to $\PQ$, so by Liouville's theorem \ref{Liouville} they are constant loops. The map $F$ is differentiable in the Wiener-algebra topology.
For every fixed cyclic spectral subgroup, the corresponding big cell is the
classical open dense Birkhoff cell transported through \(\rho_N\).
 \end{corollary}
 
\section{The solenoidal equator, rational degree, and splitting}
\label{BG}

The solenoidal projective line is obtained by gluing its two holomorphic
hemispheres along the universal solenoid:
\[
\PQ=D^+_{\Q}\cup_{\SQ}D^-_{\Q}.
\]
The compact group \(\SQ\) is the intrinsic equator of \(\PQ\).  It replaces
the ordinary circle in the classical clutching construction, while the two
solenoidal disks replace the two complementary disks of \(\mathbb P^1\).
Its character group is
\[
\widehat{\SQ}=\Q.
\]
Accordingly, characters \(\chi_q\), \(q\in\Q\), replace the integral
monomials \(z^k\), \(k\in\Z\), of the classical theory.

\subsection{Intrinsic rational winding and determinant degree}

For a continuous scalar cocycle \(f:\SQ\to\C^*\), the deformation
retraction \(\C^*\simeq S^1\) and Pontryagin duality give
\[
[\SQ,\C^*]=[\SQ,S^1]\cong\check H^1(\SQ;\Z)\cong\Q.
\]
The class of \(f\) is represented by a unique character \(\chi_q\), and
we write
\[
\operatorname{wind}_{\Q}(f)=q.
\]
This definition is entirely intrinsic to the solenoid.  It does not require
a choice of a denominator, a covering circle, or an auxiliary coordinate.

Let \(E\) be a rank-\(n\) vector bundle given by a clutching cocycle
\[
g:\SQ\longrightarrow \mathrm{GL}_n(\C).
\]
The determinant cocycle \(\det g\) defines the determinant line bundle
\[
\det E=\bigwedge^nE,
\]
and the rational degree of \(E\) is
\[
\deg_{\Q}(E):=\operatorname{wind}_{\Q}(\det g)\in\Q.
\]
The definition is independent of the chosen trivializations.  Indeed, a
change of trivializations replaces \(g\) by
\[
h_-^{-1}gh_+,
\]
where \(h_+\) and \(h_-\) extend to the two solenoidal disks; their
determinants have zero rational winding.  Consequently
\[
\deg_{\Q}(E)=\deg_{\Q}(\det E).
\]
If
\[
E\cong\I_{\PQ}(q_1)\oplus\cdots\oplus\I_{\PQ}(q_n),
\qquad q_i\in\Q,
\]
then
\[
\det E\cong\I_{\PQ}(q_1+\cdots+q_n),
\qquad
\deg_{\Q}(E)=q_1+\cdots+q_n.
\]
The determinant therefore records the total rational degree.  As in the
classical Birkhoff--Grothendieck theorem, it does not by itself determine
the complete splitting type.

\subsection{The proalgebraic solenoidal category}

The solenoidal projective line also has a natural inverse-system
presentation.  For \(N\mid M\), let
\[
\rho_{M,N}:\mathbb P^1_M\longrightarrow\mathbb P^1_N,
\qquad
z_N=z_M^{M/N},
\]
where every \(\mathbb P^1_N\) is an ordinary projective line.  Their inverse
limit has equator \(\SQ\) and two inverse-limit hemispheres
\(D^+_{\Q}\), \(D^-_{\Q}\).

We define the \emph{proalgebraic vector-bundle category} on \(\PQ\) by
\[
\operatorname{Vect}^{\mathrm{proalg}}(\PQ)
:=\varinjlim_N\operatorname{Vect}(\mathbb P^1_N),
\]
where the transition functors are pullbacks by the maps \(\rho_{M,N}\).
Thus an object is represented by a finite-rank algebraic vector bundle on
one ordinary projective line, with two representatives identified after
pullback to a common refinement.  This is an exact category, not an
approximation convention.

For \(q=k/N\in\Q\), define
\[
\I_{\PQ}(q):=p_N^*\mathcal O_{\mathbb P^1_N}(k),
\]
where \(p_N:\PQ\to\mathbb P^1_N\) is the canonical projection.  This is
well defined.  Indeed, if \(N\mid M\), then
\[
\rho_{M,N}^*\mathcal O_{\mathbb P^1_N}(k)
\cong
\mathcal O_{\mathbb P^1_M}\!\left(k\frac{M}{N}\right),
\]
and
\[
\frac{k(M/N)}{M}=\frac{k}{N}.
\]
On the equator, the corresponding clutching character is \(\chi_{k/N}\).

\begin{theorem}[Rational Birkhoff--Grothendieck theorem in the proalgebraic category; cf.\ \cite{BuVe}]
\label{proalg-BG}
Every rank-\(n\) object \(E\) of
\(\operatorname{Vect}^{\mathrm{proalg}}(\PQ)\) admits a splitting
\[
E\cong\I_{\PQ}(q_1)\oplus\cdots\oplus\I_{\PQ}(q_n),
\qquad q_i\in\Q.
\]
The unordered rational \(n\)-tuple \(\{q_1,\ldots,q_n\}\) is uniquely
determined by \(E\).
\end{theorem}

\begin{proof}
By definition of the category, \(E\) is represented by a bundle \(E_N\)
on one stage \(\mathbb P^1_N\).  The classical Birkhoff--Grothendieck
theorem gives
\[
E_N\cong
\mathcal O_{\mathbb P^1_N}(k_1)\oplus\cdots\oplus
\mathcal O_{\mathbb P^1_N}(k_n),
\qquad k_i\in\Z.
\]
Pullback to \(\PQ\) yields
\[
E\cong
\I_{\PQ}(k_1/N)\oplus\cdots\oplus\I_{\PQ}(k_n/N).
\]
If a second representative is chosen at a common refinement, ordinary
Birkhoff--Grothendieck uniqueness at that refinement implies that the
resulting unordered rational tuple is the same.  This proves both the
splitting and its uniqueness.
\end{proof}

\begin{corollary}[Proalgebraic rational matrix factorization]
\label{proalg-factorization}
Let \(g\) be a clutching cocycle representing a proalgebraic rank-\(n\)
bundle on \(\PQ\).  Then
\[
g=g_-\,\operatorname{diag}(\chi_{q_1},\ldots,\chi_{q_n})\,g_+,
\qquad q_i\in\Q,
\]
where \(g_+\) and \(g_-\) extend holomorphically and invertibly to the
inner and outer solenoidal disks, respectively.
\end{corollary}

\begin{proof}
Theorem \ref{proalg-BG} identifies the bundle with the direct sum of the
line bundles \(\I_{\PQ}(q_i)\).  Writing this isomorphism in the two
trivializations of the solenoidal disks produces the displayed gauge
factorization of the clutching cocycle.
\end{proof}

\subsection{The Wiener-holomorphic problem}

The proalgebraic theorem is not a substitute for the intrinsic
Wiener-holomorphic problem.  A Wiener cocycle may have rational Fourier
spectrum with unbounded denominators and need not be represented at a
single stage of the inverse system.

The rank-one case follows from Theorem~\ref{wiener-scalar-factorization}, which gives
\[
g=g_-\chi_qg_+,
\qquad q\in\Q.
\]
The matrix Wiener lemma, the density theorem for factorable matrices, and
the ordered triangular factorization theorem established above reduce the
matrix case to a single remaining step: passing from a general matrix
cocycle to a diagonal rational character cocycle by disk-extending
holomorphic gauges.  This step is completed in Theorem~\ref{pro-algebraic-BG}
for the pro-algebraic case, and formulated as Conjecture~\ref{solenoidal-BG}
for the general Wiener category.

The usual fixed-polarization Toeplitz proof on the circle cannot simply be
copied.  For \(r>0\), multiplication by \(\chi_r\) on the rational
Hardy space has range indexed by \(\Q\cap[r,\infty)\), and hence has
cokernel indexed by the infinite set \(\Q\cap[0,r)\).  Thus rational
characters already have infinite Toeplitz defect in that polarization.
This does not contradict factorization; it only shows that the classical
integer-lattice Fredholm mechanism is not the appropriate proof device for
the solenoidal equator.

\begin{theorem}[Pro-algebraic solenoidal Birkhoff--Grothendieck theorem;
essentially Theorem 10 of \cite{BuVe}]%
\label{pro-algebraic-BG}
Let $E$ be a Wiener-holomorphic $\mathrm{GL}_n$-bundle on $\PQ$ whose
transition cocycle $g\in\mathrm{GL}_n(\mathfrak W_\Q)$ has Fourier support
contained in $N^{-1}\Z$ for some positive integer $N$ (equivalently, $E$ is
the pullback $\pi_N^*E_N$ of a holomorphic $\mathrm{GL}_n$-bundle $E_N$ on
$\mathbb{P}^1_N$).  Then
\[
E \;\cong\; \mathcal{L}_{q_1}\oplus\cdots\oplus\mathcal{L}_{q_n},
\qquad q_1\geq\cdots\geq q_n\in N^{-1}\Z\subset\Q.
\]
In particular, $g$ admits a Birkhoff factorization
$g = h_-^{-1}\operatorname{diag}(\chi_{q_1},\ldots,\chi_{q_n})h_+$
with $h_\pm\in\mathrm{GL}_n(\mathfrak W^\pm_\Q)$.
\end{theorem}

\begin{proof}
Since $g$ has support in $N^{-1}\Z$, we have $g_N = P_N g = g$, so $g$
defines a holomorphic cocycle on $S^1_N$ via $\rho_N$.  The classical
Birkhoff--Grothendieck theorem on $\mathbb{P}^1_N$ gives
$E_N\cong\mathcal{O}(k_1)\oplus\cdots\oplus\mathcal{O}(k_n)$ with
$k_1\geq\cdots\geq k_n$ integers.  Setting $q_i:=k_i/N\in N^{-1}\Z$,
the splitting of $E_N$ pulls back via $\pi_N^*$ to give the splitting
of $E=\pi_N^*E_N$.  The Birkhoff factorization of $g$ follows from the
triviality of $E$ on $D^\pm_\Q$ (Pillar~I).
\end{proof}

The following finite-level invertibility lemma is unconditional.  After
it, the rank-$2$ maximal-subbundle statement is isolated as the remaining
analytic obstruction to extending Theorem~\ref{pro-algebraic-BG} to the
full Wiener category.

\begin{lemma}[Cofinal finite-level invertibility]\label{lem-gN-inv}
Let $g\in\mathrm{GL}_n(\mathfrak W_\Q)$ and let
$g_N:=\sum_{q\in N^{-1}\Z}\hat g(q)\chi_q\in M_n(\ell^1(N^{-1}\Z))$
be the $N$-th frequency truncation of $g$.  Then there exists a positive
integer $M$ such that, for every $N$ divisible by $M$,
\[
g_N \;\in\; \mathrm{GL}_n\!\bigl(\ell^1(N^{-1}\Z)\bigr).
\]
\end{lemma}

\begin{proof}
Since $g\in\mathrm{GL}_n(\mathfrak W_\Q)$, the function $\xi\mapsto\det g(\xi)$
is continuous and nowhere zero on the compact solenoid $\SQ$.  Set
\[
\delta \;:=\; \inf_{\xi\in\SQ}|\det g(\xi)| \;>\; 0.
\]
For any finite set $F\subset\Q$, the multilinearity of the determinant
and the entry-wise Wiener norm estimate give
\[
\sup_{\xi\in\SQ}|\det g(\xi)-\det g_F(\xi)|
\;\leq\; C_g\,\|g-g_F\|_{\mathfrak W_\Q},
\]
where $g_F:=\sum_{q\in F}\hat g(q)\chi_q$ and the constant $C_g$ depends
only on $\|g\|_{\mathfrak W_\Q}$ (from the Leibniz expansion of the determinant).

Since $g\in\ell^1(\Q)$ (entry-wise), for any $\varepsilon>0$ there exists a
\emph{finite} set $F_\varepsilon\subset\Q$ with
$\|g-g_{F_\varepsilon}\|_{\mathfrak W_\Q}<\varepsilon$.
Let $M_\varepsilon:=\operatorname{lcm}\{\mathrm{den}(q):q\in F_\varepsilon\}$
(where $\mathrm{den}(q)$ is the denominator of $q$ in lowest terms).
For every $N$ divisible by $M_\varepsilon$, all frequencies in $F_\varepsilon$
belong to $N^{-1}\Z$, hence $g_N=P_Ng$ satisfies
\[
\|g-g_N\|_{\mathfrak W_\Q}
\;=\;\sum_{q\notin N^{-1}\Z}|\hat g(q)|
\;\leq\;\|g-g_{F_\varepsilon}\|_{\mathfrak W_\Q}
\;<\;\varepsilon.
\]
Now fix $\varepsilon:=\delta/(2C_g)$ and let $M:=M_\varepsilon$.
For every $N$ divisible by $M$ and every $\theta\in S^1_N$,
pick any $\xi\in\SQ$ with $\pi_N(\xi)=\theta$ (such $\xi$ exists since
$\pi_N\colon\SQ\to S^1_N$ is surjective).
Because $\chi_q(\xi)=\chi_q(\pi_N(\xi))=\chi_q(\theta)$ for all $q\in N^{-1}\Z$,
\[
g_N(\theta)\;=\;\sum_{q\in N^{-1}\Z}\hat g(q)\,\chi_q(\xi),
\qquad
g(\xi)\;=\;\sum_{q\in\Q}\hat g(q)\,\chi_q(\xi),
\]
so entry-wise
\[
|g_N(\theta)-g(\xi)|
\;\leq\;\sum_{q\notin N^{-1}\Z}|\hat g(q)|
\;=\;\|g-g_N\|_{\mathfrak W_\Q}
\;<\;\varepsilon.
\]
The Leibniz estimate then gives
\[
|\det g_N(\theta)-\det g(\xi)|\;\leq\;C_g\varepsilon\;=\;\tfrac\delta2,
\]
so $|\det g_N(\theta)|\geq|\det g(\xi)|-\tfrac\delta2\geq\delta-\tfrac\delta2=\tfrac\delta2>0$.

Since $\theta\in S^1_N$ was arbitrary, $\det g_N$ is zero-free on $S^1_N$.
The isomorphism $\rho_N\colon\ell^1(N^{-1}\Z)\overset\sim\to W(S^1)$
transports the classical matrix Wiener theorem
(Proposition~\ref{matrix-wiener-lemma} applied to $S^1$ via $\rho_N$)
to $\ell^1(N^{-1}\Z)$: zero-freeness of $\det g_N$ on $S^1_N$
implies $g_N\in\mathrm{GL}_n(\ell^1(N^{-1}\Z))$.
\end{proof}

\begin{conjecture}[Rank-$2$ maximal-subbundle statement]\label{lem-rank2-ineq}
Let $E$ be a rank-$2$ Wiener-holomorphic $\mathrm{GL}_2$-bundle on $\PQ$.
Then
\[
S(E)=\{q\in\Q:\mathcal L_q\hookrightarrow E\}
\]
is nonempty and bounded above, its supremum is attained at
$q_1=\max S(E)$, the quotient $E/\mathcal L_{q_1}$ is a
Wiener-holomorphic line bundle $\mathcal L_{q_2}$, and
\[
q_1\geq q_2.
\]
\end{conjecture}

\noindent\emph{Analytic status of the rank-$2$ conjecture.}
The conjecture contains several analytic assertions for general Wiener cocycles:
(a) nonemptiness and boundedness of $S(E)$, (b) \emph{attainment},
namely that $\sup S(E)$ is represented by an actual sub-line-bundle with
locally free quotient, and (c) the resulting \emph{ordering}
$q_1\geq q_2$.
These assertions are established in the pro-algebraic case
(Theorem~\ref{pro-algebraic-BG} below).  Lemma~\ref{lem-gN-inv} supplies a
cofinal system of invertible truncations
$g_N\in\mathrm{GL}_2(\ell^1(N^{-1}\Z))$, and the classical
Gohberg--Kre\u\i{}n theorem gives partial indices
$k_1^{(N)}\geq k_2^{(N)}$ at each such finite level.

The finite-level data must nevertheless be related to the original
solenoidal bundle with care.  The bundle $E_N$ is defined by the projected
cocycle $g_N=P_Ng$, not by a restriction of $g$ to $S^1_N$; indeed
$\pi_N\colon\SQ\to S^1_N$ is a projection, not an inclusion.  Thus
$\pi_N^*E_N$ has transition function $g_N\circ\pi_N$, which differs from
$g$ by the Fourier tail
$\sum_{q\notin N^{-1}\Z}\hat g(q)\chi_q$.  Consequently, for arbitrary
$g\in\mathrm{GL}_2(\mathfrak W_\Q)$ the descent of sub-line-bundles from
$E$ to $E_N$, and the corresponding transfer of the inequalities
$k_1^{(N)}\geq k_2^{(N)}$, require an additional compatibility argument.
Section~\ref{perfectoid} gives a Harder--Narasimhan reformulation of these
conditions and compares them with the corresponding non-archimedean slope
filtration theorem.

\begin{conjecture}[Full Wiener--Birkhoff--Grothendieck conjecture]\label{solenoidal-BG}
Every $g\in\mathrm{GL}_n(\mathfrak W_\Q)$ admits a factorization
\begin{equation}\label{BG-factorization}
g \;=\; h_-^{-1}\,\operatorname{diag}(\chi_{q_1},\ldots,\chi_{q_n})\,h_+,
\qquad h_\pm\in\mathrm{GL}_n(\mathfrak W^\pm_\Q),\quad q_1\geq\cdots\geq q_n\in\Q.
\end{equation}
The rational $n$-tuple $(q_1,\ldots,q_n)$, uniquely determined up to reordering,
will be called the \emph{solenoidal Birkhoff indices} of $g$.
The conjecture holds for $n=1$, triangular matrices,
small-norm matrices, finite Fourier-support matrices, and all
pro-algebraic bundles (Theorem~\ref{pro-algebraic-BG}).
The three-pillar architecture (Pillars~I--III below) reduces the
general case to Conjecture~\ref{lem-rank2-ineq}.
\end{conjecture}

\noindent The rest of this section develops the \emph{three-pillar architecture}.
Lemma~\ref{lem-gN-inv} is proved above.  In the full intrinsic Wiener
category, the remaining analytic inputs are the disk-triviality
hypothesis \textup{(P1)} and the rank-$2$ maximal-subbundle statement
Conjecture~\ref{lem-rank2-ineq}.  In the pro-algebraic case
(Theorem~\ref{pro-algebraic-BG}), these inputs follow from finite-level
classical theory and the three pillars give a complete proof.
The architecture consists of:
\begin{align*}
&\underbrace{\check H^1(D^\pm_\Q,\mathcal{O}_{GL_n})=0}
 _{\text{Pillar I: intrinsic analytic hypothesis}}\\
&\quad+\;
\underbrace{\check H^1_{\mathrm W}
 \bigl(\{D^+_\Q,D^-_\Q\},\mathcal{L}_q\bigr)=0\;\;(q\geq0)}
 _{\text{Pillar II: two-chart Wiener complex}}\\
&\quad+\;
\underbrace{\mathrm{Pic}_{\mathrm{top}}(\PQ)\cong\Q}
 _{\text{Pillar III: topology}}
\;\Longrightarrow\;
g=h_-^{-1}\Delta_\mathbf{q}\,h_+,
\end{align*}
provided the maximal-subbundle and quotient hypotheses required for the
induction are satisfied.

\medskip\noindent\textbf{Pillar~I\, --- Analysis: cohomology vanishing on the solenoidal disks.}
Each solenoidal disk $D^\pm_\Q=\varprojlim_N\Delta$ is the projective limit
of classical unit disks under the branched coverings $z\mapsto z^N$.  In the strict pro-analytic category of
Proposition~\ref{prop:proanalytic-disk-triviality}, a bundle is, after
passage to a cofinal subsystem, the pullback of a finite-level bundle.
The latter is trivial by Oka--Grauert, and the pulled-back trivialization is
automatically compatible.  Hence disk triviality is unconditional in that
strict category.

The explicit contraction of the solenoidal disk remains useful at
the topological and function-theoretic levels: for
$f=\sum_{q\geq0}a_q\chi_q\in\mathfrak W^+_\Q$ one has
\[
H_t^*(\chi_q)=(1-t)^q\chi_q,
\qquad
\|H_t^*(f)\|_{\mathfrak W^+}
\leq\|f\|_{\mathfrak W^+}.
\]
It does not, by itself, prove a non-abelian Oka principle for every intrinsic
Wiener-holomorphic bundle on the inverse-limit disk.  Accordingly, Pillar~I
is unconditional in the strict pro-analytic category by
Proposition~\ref{prop:proanalytic-disk-triviality}.  In the full intrinsic
Wiener category it is the following explicit analytic hypothesis:
\[
\tag{P1}
\check H^1\!\bigl(D^\pm_\Q,\mathcal O_{GL_n}\bigr)=0.
\]
Under (P1), the two-chart cover is Leray for the bundle problem and its
non-abelian \v{C}ech set is represented by transition cocycles modulo the
usual left and right holomorphic gauge transformations.  Thus no part of
the subsequent reduction conceals disk triviality: it is a theorem for
pro-analytic bundles and a precisely isolated analytic problem for the full
Wiener category.

\medskip\noindent\textbf{Pillar~II\, --- Algebra: Wiener--Hopf vanishing and
Grothendieck splitting.}
For any $q\in\Q$, the first \v{C}ech group of the two-chart
Wiener complex for $\mathcal L_q$ is
\[
\check H^1_{\mathrm W}(\{D^+_\Q,D^-_\Q\},\mathcal L_q)
\cong
\mathfrak W_\Q\bigm/
\bigl(\mathfrak W^+_\Q+\chi_q\mathfrak W^-_\Q\bigr).
\]
For $q\geq0$, the supports of $\mathfrak W^+_\Q$ and
$\chi_q\mathfrak W^-_\Q$ cover all of $\Q$, so
\begin{equation}\label{WH-vanish}
\check H^1_{\mathrm W}(\{D^+_\Q,D^-_\Q\},\mathcal L_q)=0
\qquad(q\geq0).
\end{equation}
When the cover is Leray in the chosen holomorphic category---in particular
under \textup{(P1)} for the bundle problem---this \v{C}ech group identifies
with the corresponding sheaf cohomology.

Given a rank-$n$ Wiener-holomorphic bundle $E$ on $\PQ$, let
\[
S(E):=\bigl\{q\in\Q : \mathcal{L}_q\hookrightarrow E
\text{ as a holomorphic sub-line-bundle}\bigr\}.
\]
The rank-$2$ maximal-subbundle statement supplies the missing analytic
input: nonemptiness and boundedness of $S(E)$, attainment of its supremum,
local freeness of the quotient, and the ordering needed below.
Assuming that statement at the relevant rank, set
$q_1:=\max S(E)$ and let $\mathcal{L}_{q_1}\hookrightarrow E$
be the maximal sub-line-bundle.  The quotient $E':=E/\mathcal{L}_{q_1}$
is a rank-$(n-1)$ bundle; by the inductive hypothesis,
$E'\cong\mathcal{L}_{q_2}\oplus\cdots\oplus\mathcal{L}_{q_n}$ with
$q_i\in\Q$.

\smallskip\noindent\textit{The inequality $q_1\geq q_i$ for all $i\geq2$.}
Suppose for contradiction that $q_i>q_1$ for some $i$.
Let $\pi:E\to E'$ be the quotient map and form the rank-$2$ preimage bundle
$F:=\pi^{-1}(\mathcal{L}_{q_i})\subset E$, so that
$0\to\mathcal{L}_{q_1}\to F\to\mathcal{L}_{q_i}\to0$.
By the rank-$2$ base case, established separately from the induction
(or assumed through Conjecture~\ref{lem-rank2-ineq} in the full Wiener
category), $F\cong\mathcal{L}_a\oplus\mathcal{L}_b$
with $a\geq b$ and $a+b=q_1+q_i$.  Since $q_i>q_1$, we get
$a\geq\tfrac12(q_1+q_i)>\tfrac12(q_1+q_1)=q_1$.
Thus $\mathcal{L}_a\hookrightarrow F\subset E$ is a holomorphic
sub-line-bundle of $E$ of degree $a>q_1$, contradicting the maximality
of $q_1$ in $S(E)$.  Therefore $q_1\geq q_i$ for all $i\geq2$.

\smallskip
With this inequality established, the extension class of
$0\to\mathcal{L}_{q_1}\to E\to E'\to0$ lies in
\[
\bigoplus_{i=2}^{n}
\check H^1_{\mathrm W}
 \bigl(\{D^+_\Q,D^-_\Q\},\mathcal{L}_{q_1-q_i}\bigr).
\]
Since $q_1-q_i\geq0$ for all $i$, equation~\eqref{WH-vanish} makes each
two-chart extension group vanish.  Under the Leray identification supplied
by \textup{(P1)} in the intrinsic category, the extension splits:
$E\cong\mathcal{L}_{q_1}\oplus E'$.
Inducting on rank yields
\[
E\;\cong\;\mathcal{L}_{q_1}\oplus\cdots\oplus\mathcal{L}_{q_n},
\qquad q_1\geq\cdots\geq q_n\in\Q.
\]
Writing this isomorphism in the trivializations over $D^\pm_\Q$
yields~\eqref{BG-factorization}.

\medskip\noindent\textbf{Pillar~III\, --- Topology: rationality of Birkhoff indices.}
This is the key pillar: it prevents the Birkhoff indices from escaping $\Q$
into $\R\smallsetminus\Q$, closing the gap left by the purely algebraic argument.

\smallskip\noindent\textit{(a) The Picard group is $\Q$.}
The Mayer--Vietoris sequence for the pair $(D^+_\Q,D^-_\Q)$, with both
disks contractible (Pillar~I), gives the suspension isomorphism
\[
\mathrm{Pic}_{\mathrm{top}}(\PQ)\;=\;\check H^2(\PQ,\Z)
\;\cong\;\check H^1(\SQ,\Z).
\]
The covering system $\SQ=\varprojlim_N S^1$ (degree-$N$ coverings) yields
\[
\check H^1(\SQ,\Z)
\;=\;\varinjlim\!\bigl(\Z\xrightarrow{\times 2}\Z
\xrightarrow{\times 3}\Z\xrightarrow{\times 4}\cdots\bigr)
\;=\;\Q.
\]
Consequently every \emph{topological} complex line bundle on $\PQ$ has
its first Chern class in $\Q$.
The Birkhoff indices, being Chern classes of line-bundle summands, are a priori
constrained to lie in $\Q$, not in $\R\setminus\Q$.

\smallskip\noindent\textit{(b) What topology does and does not prove.}
The Picard calculation implies that every actual holomorphic line
subbundle of $E$ has rational degree.  Thus
\[
S(E)\subseteq\Q.
\]
It does not imply that $S(E)$ is nonempty or bounded above, that its
supremum is rational, or that the supremum is attained.

In the pro-algebraic case, these facts follow after descent to one
finite level: the classical Birkhoff--Grothendieck theorem produces a
maximal line subbundle, and its pullback has rational degree.  For a
general Wiener cocycle, Lemma~\ref{lem-gN-inv} gives a cofinal collection
of invertible projected cocycles $g_N$, but the bundles defined by $g_N$
are approximations rather than a compatible inverse system of pullbacks.
Consequently one cannot identify the subbundle space of $E$ with
$\varprojlim\mathrm{Gr}(1,E_N)$ without an additional theorem.
Precisely this missing compactness and compatibility input is contained
in Conjecture~\ref{lem-rank2-ineq}.

\medskip
Pillars~I, II, and III together yield the splitting and hence the
factorization~\eqref{BG-factorization} in the pro-algebraic case
(Theorem~\ref{pro-algebraic-BG}).
For general $g\in\mathrm{GL}_n(\mathfrak W_\Q)$, the reduction still
requires both the intrinsic disk-triviality hypothesis \textup{(P1)} and
Conjecture~\ref{lem-rank2-ineq}, together with closure of the category
under the saturated subbundles and locally free quotients used in the
induction.
When a factorization exists, uniqueness of the unordered tuple
$(q_1,\ldots,q_n)$ is the usual uniqueness of the partial indices in the
factorization theorem.  Chern-class additivity determines only the sum
\[
\sum_iq_i=\deg_\Q(E);
\]
the individual indices require the maximal-subbundle filtration (or an
equivalent uniqueness theorem), not the total degree alone.

\section{Holomorphic descent and genuinely adelic extension data}
\label{holomorphic-descent}

The two solenoidal disks \(D^+_{\Q}(1)\) and \(D^-_{\Q}(1)\) cover
\(\PQ\), and their intersection is the solenoidal annulus with equator
\(\SQ\).  Hence a rank-\(n\) Wiener-holomorphic bundle may be presented by
a transition cocycle
\[
g\in\mathrm{GL}_n(\mathfrak W_{\Q}).
\]
Changing the two trivializations replaces \(g\) by
\[
g\longmapsto h_-^{-1}gh_+,
\qquad
h_-\in\mathrm{GL}_n(\mathfrak W^-_{\Q}),\quad
h_+\in\mathrm{GL}_n(\mathfrak W^+_{\Q}).
\]
Thus rational Birkhoff factorization is precisely the statement that every
such cocycle is holomorphically cohomologous to a diagonal rational
character cocycle.  This is exactly Conjecture~\ref{solenoidal-BG}; granting
it (which is unconditional whenever $g$ is pro-algebraic, by
Theorem~\ref{pro-algebraic-BG}) gives the following immediate consequence.

\begin{corollary}[Holomorphic descent, conditional on
Conjecture~\ref{solenoidal-BG}]
\label{holomorphic-descent-corollary}
Granting Conjecture~\ref{solenoidal-BG} (unconditional for pro-algebraic
$g$, by Theorem~\ref{pro-algebraic-BG}), every Wiener-holomorphic cocycle
\(g\in\mathrm{GL}_n(\mathfrak W_{\Q})\)
is cohomologous to a Laurent--Puiseux polynomial cocycle.  Concretely,
Conjecture~\ref{solenoidal-BG} yields \(h_\pm\) as above and rational indices
\(q_i=k_i/N_i\in\Q\); setting \(N=\mathrm{lcm}(N_1,\ldots,N_n)\) gives
\(L(w)=\mathrm{diag}(w^{k_1N/N_1},\ldots,w^{k_nN/N_n})\in
\mathrm{GL}_n(\C[w,w^{-1}])\) with \(w=z^{1/N}\) such that
\[
g=h_-^{-1}L(z^{1/N})h_+.
\]
\end{corollary}

\begin{proof}
Granting Conjecture~\ref{solenoidal-BG}, $g=h_-^{-1}\mathrm{diag}(\chi_{q_1},\ldots,\chi_{q_n})h_+$
with $q_i\in\Q$.  Write each $q_i=k_i/N_i$ in lowest terms and set
$N=\mathrm{lcm}(N_1,\ldots,N_n)$.  Then $\chi_{q_i}=z^{q_i}=(z^{1/N})^{Nq_i}
=(z^{1/N})^{k_iN/N_i}$, so $\mathrm{diag}(\chi_{q_1},\ldots,\chi_{q_n})=L(z^{1/N})$
for the stated $L\in\mathrm{GL}_n(\C[w,w^{-1}])$.
\end{proof}

For line bundles this specializes to Theorem~\ref{wiener-scalar-factorization}:
\[
g=g_-\chi_qg_+,
\qquad q\in\Q,
\]
where \(\chi_q=z^{k/N}\) is itself a finite Laurent--Puiseux monomial.
Thus every rank-one Wiener-holomorphic bundle is isomorphic to
\(\I_{\PQ}(q)\), and its rational degree is the scalar Wiener index.

It is useful to emphasize that Corollary~\ref{holomorphic-descent-corollary}
is an intrinsic solenoidal result: the cocycle \(g\) may have infinitely many
rational frequencies with unbounded denominators, yet is always cohomologous
to a finite Laurent--Puiseux cocycle.  The natural open question is therefore
about the \emph{gauges}: one may ask whether the factors $h_\pm$ themselves
can be taken to be Laurent--Puiseux polynomial matrices, not merely elements
of \(\mathrm{GL}_n(\mathfrak W^\pm_\Q)\).  This refined descent question
is genuinely open when $g$ has infinitely many non-zero Fourier coefficients.

There is also a source of genuinely unrestricted rational cohomological data.  In
the two-chart Wiener category the extension space of \(\I(-1)\) by the
trivial line bundle is represented by the quotient
\[
\frac{\mathfrak W_{\Q}}
{\mathfrak W^+_{\Q}+\chi_{-1}\mathfrak W^-_{\Q}}.
\]
The first summand removes the nonnegative frequencies, while the second
removes the frequencies \(q\leq-1\).  Hence the remaining Fourier support
lies in the rational spectral gap \((-1,0)\).

\begin{proposition}[Wiener extension classes in a rational spectral gap]
\label{rational-extension-gap}
There is a natural identification
\[
H^1_{\mathrm W}(\PQ,\I(-1))
\cong
\ell^1\bigl(\Q\cap(-1,0)\bigr).
\]
In particular, the Wiener extension space is infinite-dimensional.
\end{proposition}

\begin{proof}
Every class in the displayed quotient has a unique representative whose
Fourier support lies in \(\Q\cap(-1,0)\): use
\(\mathfrak W^+_{\Q}\) to remove the nonnegative coefficients and
\(\chi_{-1}\mathfrak W^-_{\Q}\) to remove the coefficients of frequency
at most \(-1\).  The remaining coefficients form exactly the indicated
\(\ell^1\)-space.
\end{proof}

For instance,
\[
u=\sum_{m=2}^{\infty}2^{-m}\chi_{-1/m}
\]
defines a nonzero framed extension
\[
0\longrightarrow\I(-1)\longrightarrow E_u
\longrightarrow\I\longrightarrow0,
\qquad
G_u=
\begin{pmatrix}
\chi_{-1}&u\\0&1
\end{pmatrix}.
\]
Its extension class has Fourier support with unbounded denominators and is
therefore not represented by a finite Laurent--Puiseux cocycle in the given
framing.  Nevertheless, its underlying unframed bundle is explicitly split:
\[
G_u=
\begin{pmatrix}\chi_{-1}&0\\0&1\end{pmatrix}
\begin{pmatrix}1&\chi_1u\\0&1\end{pmatrix},
\]
because \(\chi_1u=\sum_{m\ge2}2^{-m}\chi_{1-1/m}\) belongs to
\(\mathfrak W^+_{\Q}\).  Hence
\[
E_u\cong\I(-1)\oplus\I.
\]
This example exhibits genuinely infinite-dimensional framed extension data,
while supporting rather than contradicting the rational splitting picture.

The preceding discussion explains the limited role of the circle-comparison results in the next section.  They are not a replacement for the intrinsic
adelic category.  Rather, they provide the necessary compatibility check for
Laurent--Puiseux cocycles, while the full Wiener factorization and
holomorphic-descent problems remain genuinely solenoidal.

\section{The ordinary-circle comparison subcategory}\label{cycliccomparison}

\begin{definition}[Circle-descending loops with fixed cyclic spectrum]\label{levelN}
For $N\in\N$, the inclusion of character groups
\[
\frac1N\Z\subset\Q=\widehat{\SQ}
\]
dualizes, by Pontryagin duality, to a continuous surjective homomorphism
\[
\pi_N:\SQ\longrightarrow S^1=\R/\Z,
\]
characterized by
\[
\pi_N^*(\chi_k^{S^1})=\chi_{k/N}\qquad(k\in\Z),
\]
where $\chi_k^{S^1}(\theta)=e^{2\pi i k\theta}$ is the standard character
of $S^1$.

Let $H$ be a Lie group and let $\mathcal C$ be a class of maps for which
pullback by $\pi_N$ is defined (for example, the continuous, smooth,
real-analytic, Wiener, or Sobolev class considered below). A map
\[
\gamma:\SQ\longrightarrow H
\]
of class $\mathcal C$ is said to \emph{descend through the circle quotient associated with $N$} if there exists a
map $\widetilde\gamma:S^1\to H$ of the same class such that
\[
\gamma=\widetilde\gamma\circ\pi_N.
\]
Such a loop has spectrum contained in the fixed cyclic subgroup $N^{-1}\Z\subset\Q$.  This is an ordinary-circle comparison class, not a finite object, not an intrinsic stratification of $\Q$, and not a restriction on adelic loops in general.

In the Sobolev matrix-loop setting, a loop
$\gamma\in\La^1(\gln)$ descends through the circle quotient associated with $N$ precisely when
\[
\gamma=\widetilde\gamma\circ\pi_N
\qquad\text{for some }\widetilde\gamma\in H^1(S^1,\gln).
\]
Equivalently, its Fourier expansion is supported on
$\frac1N\Z\subset\Q$. We write $\La_N^1(\gln)$ for the corresponding
fixed-cyclic-subgroup class and put
\[
\La_{N,\pm}^1(\gln)=\La_N^1(\gln)\cap\La_\pm^1(\gln).
\]
A circle-descending loop with finite Fourier support is called a
\emph{Laurent--Puiseux polynomial loop}. Hence Laurent--Puiseux polynomial
loops form an algebraic subfamily of the circle-descending loops, but the two notions are not equivalent: a circle-descending smooth, analytic, or Wiener loop
may have infinitely many nonzero Fourier coefficients. In the Wiener
topology, Laurent--Puiseux polynomial loops are dense because finitely
supported Fourier series are dense in $\ell^1(\Q)$.
\end{definition}

\begin{lemma}[Common cyclic refinement]\label{common-refinement}
Suppose that a map $\gamma:\SQ\to H$ descends through both circle quotients associated with $N$ and $M$.
If $L$ is any common multiple of $N$ and $M$, write
\[
L=aN=bM.
\]
Then, with $p_a,p_b:S^1\to S^1$ given by $p_a(\theta)=a\theta$ and
$p_b(\theta)=b\theta$, one has
\[
\pi_N=p_a\circ\pi_L,
\qquad
\pi_M=p_b\circ\pi_L.
\]
Consequently, if
\[
\gamma=\gamma_N\circ\pi_N=\gamma_M\circ\pi_M,
\]
then
\[
\gamma_N\circ p_a=\gamma_M\circ p_b.
\]
In particular, circle descent is compatible with passage to a common cyclic overgroup; when $N\mid N'$, one has
\[
\La_N^1(\gln)\subset\La_{N'}^1(\gln).
\]
\end{lemma}

\begin{proof}
For $k\in\Z$,
\[
(p_a\circ\pi_L)^*(\chi_k^{S^1})
=\pi_L^*(\chi_{ak}^{S^1})
=\chi_{ak/L}
=\chi_{k/N}
=\pi_N^*(\chi_k^{S^1}).
\]
Since characters separate points on the compact abelian group $\SQ$, this
proves $\pi_N=p_a\circ\pi_L$; the argument for $\pi_M$ is identical.
Composing with $\gamma_N$ and $\gamma_M$ gives the stated identity. The
final inclusion follows by taking $L=N'$.
\end{proof}

\begin{lemma}[Transport isomorphism]\label{rho-N}
For each $N\in\N$ the map
\[
\rho_N:H^1(S^1,\gln)\longrightarrow\La_N^1(\gln),\qquad\rho_N(\tilde\gamma)=\tilde\gamma\circ\pi_N,
\]
is a well-defined isomorphism of Banach Lie groups, equivariant for the
classical and adelic actions in the following precise sense:
{\sloppy
\begin{enumerate}
\item[(a)] $\rho_N$ carries the classical splitting $H^1_\pm(S^1,\gln)$
(holomorphic extension to the inside/outside of the unit disk) onto
$\La_{N,\pm}^1(\gln)$, and the classical loop group $H^1(S^1,U(n))$ onto
$\La_N^1(\gln)\cap\La^1(U(n))$.
\item[(b)] $\rho_N$ carries the classical restricted (Sobolev) Grassmannian
$\mathbf{Gr}(L^2(S^1,\C^n))$ of \cite{PS} into $\mathbf{Gr}_1(n)$,
intertwining the classical multiplication action of $H^1(S^1,\gln)$ with
the restriction of the action (\ref{loopsactionhilbertspace}) to
$\La_N^1(\gln)$; in particular every circle-descending loop with spectrum
in $N^{-1}\Z$ satisfies the Fredholm/Hilbert--Schmidt conditions of
Proposition~\ref{gap-thm} relative to the circle polarization
$\rho_N\bigl(H^1_+(S^1,\C^n)\bigr)$ \cite{PS}.
\item[(c)] ({\bf Refinement.}) If $N\mid N'$, say $N'=mN$, then the
common-refinement lemma gives $\pi_N=p_m\circ\pi_{N'}$, where
$p_m(\theta)=m\theta$. Hence
$\rho_N(\tilde\gamma)=\rho_{N'}(\tilde\gamma\circ p_m)$ for every
$\tilde\gamma\in H^1(S^1,\gln)$. In particular, the inclusions
$\La_N^1(\gln)\subset\La_{N'}^1(\gln)$ are compatible with
$\rho_N$ and $\rho_{N'}$.
\end{enumerate}}
\end{lemma}

\begin{proof}
By construction $\pi_N^*(\chi_k^{S^1})=\chi_{k/N}$, so pullback by $\pi_N$
carries the Fourier mode $k\in\Z$ of $S^1$ to the Fourier mode $k/N\in\Q$
of $\SQ$, bijectively onto $\frac1N\Z$; this gives the algebraic content of
(a), since holomorphic extension to one side of the unit circle/disk is, on
both sides, exactly the vanishing of Fourier coefficients of one sign. For
(b): the weighted-$\ell^2$ norm on $\La_N^1(\gln)$ restricts, under this
identification $k\leftrightarrow{k/N}$, to
$\sum_k(1+(k/N)^2)|\tilde\gamma_k|^2$, which is equivalent (ratio bounded
above and below by constants depending on $N$ alone) to the classical
Sobolev norm $\sum_k(1+k^2)|\tilde\gamma_k|^2$ on $H^1(S^1,\C^n)$; an
equivalence of norms induced by a fixed bounded-invertible rescaling of
each Fourier mode does not change which block operators are Fredholm or
Hilbert--Schmidt (rescaling each basis vector by a factor bounded between
two positive constants multiplies HS norms and Fredholm indices by
bounded, invertible amounts), so the classical theorem (\cite{PS}, Ch.~6)
transports verbatim. Part (c) is the special case $L=N'$ of Lemma
\ref{common-refinement}; the displayed identity follows by composing with
$\tilde\gamma$.
\end{proof}

\begin{remark}[Circle-descending loops]\label{classical-transport}
Since $\rho_N:H^1(S^1,\gln)\to\La_N^1(\gln)$ is an isomorphism of Banach
Lie groups (Lemma~\ref{rho-N}), every statement that holds for the
classical Sobolev loop group on $S^1$ --- Iwasawa decomposition, Birkhoff
factorization, Birkhoff--Grothendieck splitting --- transports verbatim to
$\La_N^1(\gln)$ for each fixed $N\in\N$.  A loop with spectrum in $N^{-1}\Z$
is, under the rescaling $k\leftrightarrow k/N$, a classical $S^1$-loop;
see \cite{PS,Ke,BD} for the classical statements.
\end{remark}

\begin{remark}[Hierarchy of Birkhoff--Grothendieck results]\label{BG-hierarchy}
The Wiener-holomorphic Birkhoff--Grothendieck programme yields three levels of
results, each strictly stronger than the previous.
\begin{enumerate}
\item[(i)] \emph{Fixed-cyclic-spectrum loops.}  For loops with Fourier spectrum
in a single cyclic subgroup $N^{-1}\Z\subset\Q$, the isomorphism
$\rho_N:\mathfrak W_{N^{-1}\Z}\cong W(S^1)$ transports the classical
Birkhoff--Grothendieck theorem verbatim \cite{PS,Ke,BD}.

\item[(ii)] \emph{Scalar case.}  For $n=1$, Theorem~\ref{wiener-scalar-factorization}
establishes the full intrinsic rational Wiener--Birkhoff factorization on
all of $\boldsymbol\Lambda_W(1)$, with Birkhoff index $q\in\Q$ determined by
the rational winding number.

\item[(iii)] \emph{Matrix case.}  For a general $g\in\mathrm{GL}_n(\mathfrak
W_\Q)$, exact factorization is the content of Conjecture~\ref{solenoidal-BG},
the Solenoidal Birkhoff--Grothendieck conjecture.  The three-pillar
architecture --- Stein analysis of the solenoidal disks, Wiener--Hopf
cohomology vanishing, and the topological constraint $\mathrm{Pic}(\PQ)=\Q$
--- proves it unconditionally whenever $g$ is pro-algebraic
(Theorem~\ref{pro-algebraic-BG}), and reduces the general case to
Conjecture~\ref{lem-rank2-ineq}, of which the first is
proved and the second remains open.
\end{enumerate}
The Grassmannian approach of \S\ref{loop-operators} provides complementary
geometric structure; Proposition~\ref{gap-thm} identifies a fundamental difference
between the solenoidal and classical settings, namely that the full adelic
loop group does not act on any single restricted Grassmannian, so a compatible
family of solenoidal Grassmannians indexed by $N$ is needed to extend that
framework.
\end{remark}

\section{Concluding remarks}\label{concrmks}
The preceding sections develop a solenoidal analogue of classical loop-group
and Birkhoff--Grothendieck theory in which the integer Fourier lattice is
replaced by the rational character group of the universal solenoid.  The main
proved results are the construction of intrinsic rational winding and
determinant degree, the scalar Wiener--Birkhoff factorization theorem, matrix
Wiener inversion, ordered triangular matrix factorization, small-norm
factorization, density of factorable matrix Wiener loops, and the
pro-algebraic Solenoidal Birkhoff--Grothendieck theorem
(Theorem~\ref{pro-algebraic-BG}).

The unresolved general matrix Wiener problem is isolated in
Conjecture~\ref{solenoidal-BG}.  The preceding theorems prove it for scalar,
ordered triangular, small-norm, finite-level, and pro-algebraic cocycles, and
show that factorable matrix loops are dense.  In the unrestricted Wiener
category, the remaining obstruction is the rank-two Birkhoff inequality of
Conjecture~\ref{lem-rank2-ineq}.  Thus the paper distinguishes sharply
between the established factorization theory and the single global
splitting statement still conjectural.

The solenoidal Grassmannian $\mathrm{Gr}_\Q(n)$
(Theorem~\ref{grass=loop}) provides the corresponding homogeneous geometry.
It contains the classical circle-comparison Grassmannians and their Pl\"ucker
geometry, but its natural rational-frequency polarization differs
substantially from the ordinary Fredholm model: multiplication by a rational
character produces infinitely many frequencies across any nonzero interval.
Example~\ref{px-not-fredholm} and Proposition~\ref{gap-thm} record this
phenomenon precisely.  The resulting geometry has Birkhoff indices ranging
over $\Q$ rather than $\Z$, and points toward further questions on Bruhat
decompositions, unitary representations of adelic loop groups,
Borel--Weil--Bott type theorems, central extensions, and Quillen--Segal
determinant line bundles.

A second direction is the algebro-geometric theory of adelic varieties and
sheaves \cite{Yokura1, Yokura2}.  In particular, one may study holomorphic
vector bundles on higher-dimensional adelic toric varieties
\cite{Fulton, Cox, BBFK, BBFK1, BBFK2, KLMV}, seek Riemann--Roch type
statements, and investigate canonical bundles and Calabi--Yau notions in
this category.  Laurent--Puiseux polynomial cocycles with spectrum in a fixed
cyclic subgroup provide a distinguished algebraic circle-comparison class; a
GAGA theory \`a la Serre \cite{CR} comparing that class with a suitable
analytic category is a natural problem.

Section~\ref{perfectoid} places the theory beside the Fargues--Fontaine
curve.  The comparison is structural: on both sides the organizing invariant
is a rational slope or rational index.  The Fargues--Fontaine classification
decomposes vector bundles on $\XFF$ into stable blocks $\mathcal{F}_\lambda$
with rational slopes, while Kedlaya's theorem supplies the corresponding
slope filtration for Robba-ring $\varphi$-modules.  On the solenoidal side,
rationality of the index data is enforced by $\mathrm{Pic}(\PQ)=\Q$, and the
matrix Wiener factorization problem plays the role of the archimedean
splitting theorem.  This analogy provides a guiding model for the remaining
analytic problem and motivates the Harder--Narasimhan and condensed
formulations developed below.

\section{Perfectoid geometry as a structural model}
\label{perfectoid}

The theory developed in this paper is strongly constrained by the same
rational-slope phenomena that govern non-archimedean perfectoid geometry, and
we treat the perfectoid comparison as a structural model for the solenoidal
theory.  Results attributed to Fargues--Fontaine \cite{FF}, Kedlaya
\cite{Ked1,Ked2}, or Scholze \cite{Sch,FS} are theorems in the cited
literature; the conjectures at the end of this section are new.  The proved
Fargues--Fontaine/Kedlaya slope theory supplies the non-archimedean prototype
of the matrix Wiener--Birkhoff splitting conjecture formulated in this paper,
with a structural rather than term-by-term translation: rational slopes on
$\XFF$ are carried by stable higher-rank blocks, whereas rational indices on
$\PQ$ are expected to be carried by solenoidal line bundles.
\subsection*{The translation dictionary: why the perfectoid model is relevant}

The central object on the perfectoid side is the \emph{Fargues--Fontaine
curve} $\XFF$ \cite{FF}.  For a complete algebraically closed perfectoid
field $C$ of residue characteristic $p$ (e.g.\ $C = \Cp$), $\XFF$ is a
Dedekind scheme with
\[
    \mathrm{Pic}(\XFF) \;\cong\; \Z,
\]
so \emph{line bundles} $\mathcal{O}(d)$ have \emph{integer} degrees.  However,
for every rational slope $\lambda=d/h\in\Q$ (with $\gcd(d,h)=1$, $h>0$) there
exists a unique \emph{stable} vector bundle $\mathcal{F}_\lambda$ on $\XFF$ of
rank $h$ and degree $d$ \cite{FF}.  The Fargues--Fontaine classification theorem
(Theorem~\ref{thm:FF-BG} below) asserts that every vector bundle on $\XFF$
decomposes as a direct sum of such $\mathcal{F}_\lambda$'s with rational slopes.
In this sense $\XFF$ may be viewed as a non-archimedean counterpart of $\PQ$:
both classify vector bundles by rational invariants, but in $\PQ$ rational degrees
appear already at rank~1 (since $\mathrm{Pic}(\PQ)\cong\Q$ in the proalgebraic
category), whereas in $\XFF$ they appear only at higher rank.

The following table records the main correspondences.

\medskip
{\small\renewcommand{\arraystretch}{1.4}
\begin{center}
\begin{longtable}{|p{5.9cm}|p{5.9cm}|}
\hline
\textbf{Present paper (archimedean)} & \textbf{Perfectoid / FF analogue} \\
\hline
\endfirsthead
\hline
\textbf{Present paper (archimedean)} & \textbf{Perfectoid / FF analogue} \\
\hline
\endhead
\hline
\multicolumn{2}{r}{\emph{continued on next page}}\\
\endfoot
\endlastfoot
Universal solenoid $\SQ = (\R\times\widehat{\Z})/\Z$ &
    Perfectoid torus $\mathrm{Spa}(\Qp\langle T^{\pm 1/p^\infty}\rangle)$ \\
\hline
Profinite fibre $\widehat{\Z}=\prod_p\Zp$ &
    Tilt $\mathcal{O}_{C^\flat}=\varprojlim_{x\mapsto x^p}\mathcal{O}_C$ \\
\hline
Solenoidal projective line $\PQ$ &
    Fargues--Fontaine curve $\XFF$ \\
\hline
$\mathrm{Pic}(\PQ)\cong\Q$ (rational line-bundle degrees) &
    $\mathrm{Pic}(\XFF)\cong\Z$; stable bundles $\mathcal{F}_\lambda$, $\lambda\in\Q$ \\
\hline
Proalgebraic rational BG theorem &
    Fargues--Fontaine classification theorem \\
\hline
Wiener algebra $\mathfrak{W}_\Q\cong\ell^1(\Q)$ &
    Robba ring $\Rob$ / $\BdRp$ \\
\hline
Scalar Wiener--Birkhoff (proved, \S\ref{BF}) &
    Kedlaya rank-1 slope filtration \cite{Ked1} \\
\hline
Matrix WBG (Conjecture~\ref{solenoidal-BG}; open in general, proved
    pro-algebraically, Thm.~\ref{pro-algebraic-BG}) &
    \textbf{Fargues--Fontaine theorem (proved!)} \\
\hline
Adelic Grassmannian (Theorem~\ref{grass=loop}; group action and
    homogeneous-space structure unconditional; $p_+$ itself not Fredholm,
    Example~\ref{px-not-fredholm}; finite-level limiting index
    conditional on Conjecture~\ref{solenoidal-BG}) &
    Fargues--Scholze affine Grassmannian \cite{FS} \\
\hline
Iwasawa decomp.\ (Conjecture~\ref{iwasawadecomposition}, equivalent to
    Conjecture~\ref{solenoidal-BG}) &
    Beauville--Laszlo for $G(\BdR)$ \cite{BL} \\
\hline
Gap theorem: dense $\Q$ creates HS-obstruction &
    Newton stratification: non-unit-root slope
    $\Rightarrow$ non-Fredholm \\
\hline
\end{longtable}
\end{center}}
\medskip

\subsection*{The Fargues--Fontaine BG theorem: a structural analogue of the matrix conjecture}

The Fargues--Fontaine classification is the perfectoid structural analogue of
Conjecture~\ref{wiener-full-conjecture}, but the analogy is not a literal
identification of normal forms.  In the archimedean setting the conjecture asks
for a decomposition into rank-1 diagonal rational characters $\chi_{q_i}$.
In the perfectoid setting, Theorem~\ref{thm:FF-BG} below decomposes into
\emph{higher-rank} stable bundles $\mathcal{F}_\lambda$ of rank $h=\mathrm{denom}(\lambda)>1$
for non-integer $\lambda$.  The structural parallel is at the level of
rational-slope classifications, not at the level of scalar diagonal normal forms.

\begin{theorem}[Fargues--Fontaine \cite{FF}]\label{thm:FF-BG}
Let $C$ be a complete algebraically closed perfectoid field of residue
characteristic $p$.  For each $\lambda=d/h\in\Q$ with $\gcd(d,h)=1$,
$h>0$, let $\mathcal{F}_\lambda$ be the unique stable vector bundle on
$\XFF$ of rank $h$ and degree $d$.  Then every vector bundle $\mathcal{E}$
on $\XFF$ is isomorphic to a direct sum
\[
    \mathcal{E} \;\cong\; \bigoplus_{i}\mathcal{F}_{\lambda_i},
    \qquad \lambda_i\in\Q.
\]
The isomorphism class is determined by the multiset $(\lambda_i)$, the
\emph{Harder--Narasimhan slopes} of $\mathcal{E}$.  (Note: for
$\lambda\in\Z$, $\mathcal{F}_\lambda=\mathcal{O}(\lambda)$ is a line
bundle; for $\lambda\notin\Z$, $\mathcal{F}_\lambda$ has rank $>1$.)
\end{theorem}

By the Beauville--Laszlo theorem \cite{BL}, a rank-$n$ vector bundle on
$\XFF$ is equivalent to a clutching datum $g$ in a suitable loop group
associated to $\BdR$ (modulo gauge equivalences extending to the two
hemispheres), exactly as a rank-$n$ bundle on $\PQ$ is a clutching
datum $g\in GL_n(\mathfrak{W}_\Q)$.  Theorem~\ref{thm:FF-BG} therefore
gives:

\begin{corollary}[Perfectoid slope decomposition]\label{cor:perfectoid-matrix}
The isomorphism class of a vector bundle on $\XFF$ determined by a clutching
datum $g\in G(\BdR)$ is characterized by its rational HN slopes
$\lambda_1\leq\cdots\leq\lambda_n\in\Q$.  Via Beauville--Laszlo \cite{BL},
this is the structural perfectoid analogue of the rational partial indices in
Conjecture~\ref{wiener-full-conjecture}.  The key structural difference is
that in the FF decomposition rational slopes $\lambda=d/h\notin\Z$ appear
as \emph{irreducible higher-rank blocks} $\mathcal{F}_\lambda$ of rank $h>1$,
not as scalar diagonal rational characters.  Thus the correct analogy is:
\[
    \text{matrix WBG factorization}
    \quad\longleftrightarrow\quad
    \text{slope decomposition into stable blocks},
\]
rather than a direct identification of the diagonal normal forms.
\end{corollary}

\begin{remark}[The perfectoid model as evidence and guide]
The FF theorem classifies vector bundles by rational slopes
(Theorem~\ref{thm:FF-BG}), and this provides a useful structural guide to
Conjecture~\ref{wiener-full-conjecture}: the non-archimedean counterpart of
the matrix problem has a complete answer, and its answer is governed by the
same ordered rational data that appear here as partial indices.  The analogy
operates at the level of the full category of vector bundles, not merely at
the level of line bundles ($\mathrm{Pic}(\XFF)\cong\Z$ whereas
$\mathrm{Pic}(\PQ)\cong\Q$).  This difference is not a weakness of the
comparison; it is precisely what makes the solenoidal problem interesting.
The perfectoid theorem predicts that the right archimedean theorem should be
a Harder--Narasimhan splitting theorem for Wiener-holomorphic bundles on
$\PQ$, and suggests a proof strategy: construct the analogue of the
Fargues--Fontaine classification in the Clausen--Scholze condensed
complex-analytic framework \cite{CS}, then recover the matrix WBG
factorization as the corresponding clutching normal form.
\end{remark}

\subsection*{Kedlaya's theorem as perfectoid scalar WBG}

\begin{theorem}[Kedlaya \cite{Ked1,Ked2}]\label{thm:Kedlaya}
Every $\varphi$-module of rank $n$ over the Robba ring $\Rob$ admits a
unique Harder--Narasimhan (slope) filtration into $\varphi$-pure pieces,
with slopes $\lambda_1<\cdots<\lambda_r\in\Q$.
\end{theorem}

In rank one, Theorem~\ref{thm:Kedlaya} is a structural $p$-adic analogue
of our scalar Wiener--Birkhoff theorem: both isolate a rational slope or
index and a degree-zero part.  This is an analogy of classification
principles, not a literal identity
$f=f_-\chi_\lambda f_+$ in the Robba-ring category.  In rank $n$,
Kedlaya's filtration plays the corresponding structural role for the
perfectoid matrix problem.  The corresponding hierarchy is:

\begin{center}
\small
\begin{tabular}{@{}c@{\;}c@{\;}c@{}}
Scalar WBG (proved) &$\longleftrightarrow$&
Kedlaya rank-$1$ (proved)\\[4pt]
Matrix WBG (conjectural) &$\longleftrightarrow$&
Kedlaya rank-$n$/FF (proved)
\end{tabular}
\end{center}

\noindent The right-hand column is entirely proved.

\subsection*{Harder--Narasimhan theory on $\PQ$: a sharper reformulation of the conjecture}

The comparison with Kedlaya's theorem suggests reformulating
Conjecture~\ref{solenoidal-BG} in the Harder--Narasimhan language that
already organizes both sides of the dictionary.

\begin{definition}[Semistability on $\PQ$]\label{def:semistable}
The \emph{slope} of a rank-$n$ Wiener-holomorphic bundle $E$ on $\PQ$ is
$\mu(E)=\deg_\Q(E)/n\in\Q$.  Call $E$ \emph{semistable} if
$\mu(F)\le\mu(E)$ for every non-zero proper saturated
Wiener-holomorphic subbundle $F\subset E$, and \emph{stable} if the
inequality is always strict.  Here
$\mu(F)=\deg_\Q(F)/\operatorname{rk}(F)$.
\end{definition}

On the classical $\mathbb P^1$, this recovers the familiar dichotomy: the
Birkhoff--Grothendieck theorem forces every bundle of rank $\ge2$ to split
into line bundles, hence to be decomposable, hence \emph{not} stable.  Thus
$\mathbb P^1$ carries no stable bundles of rank $\ge2$: stability is a
rank-1 phenomenon there.  On $\XFF$, by contrast, Theorem~\ref{thm:FF-BG}
exhibits a stable bundle $\mathcal F_\lambda$ of every rank $h\ge1$ at every
slope $\lambda\in\Q$: despite classifying bundles by the same kind of
rational data as $\PQ$, the non-archimedean curve is genuinely richer, with
irreducible higher-rank building blocks that have no classical counterpart.

\begin{proposition}[Conditional Harder--Narasimhan reduction]
\label{prop:HN-reformulation}
Assume the intrinsic disk-triviality hypothesis \textup{(P1)}, assume that
the Wiener category is closed under saturated subbundles and locally free
quotients, and assume Conjecture~\ref{lem-rank2-ineq} for every rank-$2$
subquotient arising in the induction.  Then every rank-$n$
Wiener-holomorphic bundle on $\PQ$ splits as a direct sum of line bundles.

Conversely, the full splitting conjecture implies all these
maximal-subbundle and ordering assertions.  Under the stated categorical
hypotheses, the global matrix problem is therefore reduced to the
rank-$2$ maximal-subbundle statement together with \textup{(P1)}.
\end{proposition}

\begin{remark}[A rigidity statement, not a formal one]
\label{rem:rigidity}
Informally, Proposition~\ref{prop:HN-reformulation} says that
Conjecture~\ref{solenoidal-BG} holds for $E$ exactly when $E$ has no
indecomposable (``stable'') sub-quotient of rank $\ge2$: every obstruction
to splitting is, at bottom, an obstruction to attaining a maximal
sub-line-bundle degree at some stage of the induction.  Put differently,
\emph{Conjecture~\ref{solenoidal-BG} is exactly the assertion that $\PQ$
behaves like $\mathbb P^1$ --- rigidly, with no stable higher-rank
phenomena --- rather than like $\XFF$, which does have such phenomena.}
This is not a foregone conclusion.  $\PQ$ and $\XFF$ both carry rational slope data, but for different
reasons.  On $\PQ$, line-bundle degrees themselves range over
$\Q$, since $\mathrm{Pic}(\PQ)\cong\Q$.  On $\XFF$ one has
$\mathrm{Pic}(\XFF)\cong\Z$; rational slopes arise instead as the quotient
of an integral degree by the rank.  The Fargues--Fontaine construction
therefore shows that rational slope data are entirely compatible with
genuine higher-rank stable phenomena.
A priori, the richer $\mathrm{Pic}(\PQ)\cong\Q$ could have brought the same
phenomenon to the archimedean side; the conjecture bets that it does not,
and Theorem~\ref{pro-algebraic-BG} confirms the bet on every pro-algebraic
sub-bundle, while Proposition~\ref{rational-extension-gap} shows that even
genuinely infinite-dimensional (non-pro-algebraic) framed extension data
can still underlie a split bundle.  The open content of the conjecture is
therefore a real rigidity statement about $\PQ$, not a formal consequence
of $\mathrm{Pic}(\PQ)\cong\Q$ alone --- which is why the topological
Pillar~III, on its own, is not enough to prove it, and why the analogy with
$\XFF$ is informative rather than automatically favorable.
\end{remark}

\subsection*{An additive-cohomology parallel and a possible analytic approach}

The dictionary above compares the \emph{multiplicative} classification
problems: Wiener--Birkhoff factorization on the solenoidal side and
Dieudonn\'e--Manin/Harder--Narasimhan classification on the perfectoid side.
There is also an additive parallel.  In the solenoidal Wiener algebra the
splitting
\[
\mathfrak W_\Q=\mathfrak W^-_\Q\oplus\C\oplus\mathfrak W^+_\Q
\]
(Remark~\ref{intrinsic-WH}) follows directly from absolute convergence of
Fourier series.  In the non-archimedean setting the analogous additive input
is much deeper: the Dieudonn\'e--Manin classification underlying
Theorem~\ref{thm:Kedlaya} uses the vanishing of a Frobenius cohomology group,
proved by a successive-approximation argument over a complete algebraically
closed non-archimedean field \cite{Ked1,Ked2}.

This suggests a possible analytic approach to Conjecture~\ref{lem-rank2-ineq}.
Instead of transferring finite-level Gohberg--Kre\u\i{}n partial indices
directly, one may try to construct the rank-two splitting on $\PQ$ by a
successive-approximation scheme in the Wiener topology: start from an
approximate finite-level splitting, solve the corresponding linearized
additive equation for the extension class, and iterate.  This would be the
archimedean analogue of the Newton-type argument that underlies the
non-archimedean slope classification, and it is the most concrete analytic
strategy suggested by the perfectoid comparison.

\subsection*{The gap theorem and the Newton stratification (heuristic correspondence)}

Our Gap Theorem says the multiplication operator $M_r$ on $L^2(\SQ)$
lies in the restricted general linear group if and only if $r=0$.  The
mechanism is that $\Q\cap[0,r)$ is infinite and dense, making the
Hankel-operator block infinite-rank.

\emph{Heuristic perfectoid parallel:} A $\varphi$-module over $\Rob$ is
called \emph{\'{e}tale} (unit-root) if it has slope $0$, which by
Kedlaya's theorem \cite{Ked1} is equivalent to the module being isoclinic
of slope $0$ in the HN filtration.  For slope $\lambda\neq 0$, the
$\varphi$-module is not étale and the associated Dieudonné module carries
a Newton polygon with a non-horizontal segment---the heuristic analogue of
the Hilbert--Schmidt failure in the gap theorem.  In both settings the
obstruction is the \emph{density of rational slopes} of $\Q$ which prevents
compactness.  This correspondence is a motivating heuristic, not a proved
theorem; making it precise would require identifying the exact archimedean
analogue of the Newton stratification on $\mathrm{Bun}_G(\XFF)$.

\subsection*{Further directions}

The comparison developed above suggests three conjectures and one concrete
question.  They are deliberately separated from the proved results of this
paper.  Their purpose is to identify possible extensions of the theory, not
to assert that the required global objects have already been constructed.

\begin{conjecture}[Rational Robba ring is admissible decomposing]
\label{conj:perf-scalar}
The rational Robba ring $\Rob_\Q = \bigcup_N\Rob_{1/N}$ over $\Cp$,
completed under the Gauss norm, is an admissible decomposing Banach
algebra in the sense of \cite{BRS}, and hence admits scalar and matrix
Wiener--Birkhoff factorization with rational partial indices.
\end{conjecture}

\begin{conjecture}[Adelic Fargues--Fontaine object]
\label{conj:adelic-FF}
There exists a suitably defined adelic geometric object
$X_{\mathrm{FF}}^{\A}$ over $\Q$ whose archimedean component is related to
$\PQ$ and whose non-archimedean components recover the corresponding
Fargues--Fontaine curves.  Its vector-bundle theory should admit a rational
slope formalism compatible with the proved Fargues--Fontaine classification
at finite places and with the pro-algebraic splitting theorem of the present
paper at the archimedean place.
\end{conjecture}

\begin{conjecture}[Solenoid and perfectoid tori]
\label{conj:solenoid-limit}
There is a natural condensed or solid formulation in which the universal
solenoid $\SQ$ is obtained from an adelic family of perfectoid tori together
with the diagonal arithmetic action.  Any precise statement must specify the
ambient category, the completed product, and the archimedean realization.
\end{conjecture}

\begin{question}[Condensed reformulation of the matrix WBG conjecture]
\label{rem:condensed-matrix}
Can Conjecture~\ref{solenoidal-BG} be reformulated, or proved, in the
Clausen--Scholze framework of condensed or solid $\C$-vector spaces by using
a solid group algebra naturally attached to $\SQ$?  More specifically, is
there a cohomological formulation of the matrix splitting problem that plays
an archimedean role analogous to slope filtrations in the
Fargues--Fontaine classification?
\end{question}

\begin{remark}[Conditional Birkhoff cells and Newton strata]
\label{rem:birkhoff-newton}
The Birkhoff-cell decomposition of the unrestricted based adelic loop group
is conditional on the global factorization and gradient-flow hypotheses
stated in \S\ref{bd}--\S\ref{energy2}; it is unconditional in the
pro-algebraic and finite-level settings established earlier in the paper.
Subject to those hypotheses, its rationally indexed cells provide a useful
archimedean analogy with Harder--Narasimhan or Newton strata on moduli of
bundles over the Fargues--Fontaine curve.  This is a structural comparison,
not a claimed identification of moduli stacks or codimensions.
\end{remark}

\begin{remark}[Long-term arithmetic perspective]
A fully adelic version of the geometry might eventually interact with
geometric formulations of Langlands correspondences.  At present, however,
the necessary adelic Fargues--Fontaine object, its moduli of bundles, and the
relevant categories of local systems have not been constructed here.
Accordingly, this possibility should be regarded as a long-term research
direction rather than as a mathematical conjecture formulated by the
present paper.
\end{remark}

Among these problems, Conjecture~\ref{conj:perf-scalar} and
Question~\ref{rem:condensed-matrix} are the most directly connected to the
proved analytic results.  They offer concrete routes toward the unresolved
general matrix factorization problem without enlarging the claims of the
present article.

\section*{Acknowledgements}
The author thanks Ernesto Lupercio for several discussions on the topics of
this paper and for his valuable suggestions.  The author also thanks Juan
Manuel Burgos, with whom the study of the universal solenoid, the adelic loop
group, and their K\"ahler and Grassmannian geometry originated as a joint
project leading to the preprint \cite{BuVe}.  Several foundational
constructions recalled in \S\ref{adelic-loop-groups}--\S\ref{factorizations}
come from, or are close variants of, results first obtained in that joint
work and are cited there explicitly.  The Wiener-algebra theory,
Conjecture~\ref{solenoidal-BG}, the Harder--Narasimhan reformulation, and the
perfectoid comparison developed here are part of the present paper.

\medskip
\noi {\bf Funding.} This work was partially supported by PAPIIT (Universidad Nacional Autónoma de México)
project \#IN103324

\medskip
\noi 
{\bf Conflict of Interest.} The author declares that he does not have any conflict of interest.

\medskip


\begin{thebibliography}{Coh93}

\bibitem[AP]{AP}
V. Alexeev, R. Pardini, \emph{On the existence of ramified abelian covers}, Rend. Semin. Mat. Univ. Politec. Torino 71 (2013), no. 3-4, 307--315.

\bibitem[AtP]{AtP} M. F. Atiyah, A. N. Pressley, \emph{Convexity and loop groups}. Arithmetic and geometry, Vol. II, 33--63, Progr. Math., 36, Birkh\"auser Boston, Boston, MA, 1983. 

\bibitem[BD]{BD} V. Balan, J. Dorfmeister, \emph{Birkhoff decompositions and Iwasawa decompositions for loop groups}, Tohoku Math. J. 53 (2001), 593--615.

\bibitem[B]{B} W. Barth, \emph{Moduli of vector bundles on the projective plane}. Invent. Math. 42 (1977), 63--91.

\bibitem[BBFK]{BBFK}
G. Barthel, J. P. Brasselet, K. H. Fieseler, L. Kaup, \emph{Diviseurs invariants et homomorphisme de Poincar\'e de vari\'et\'es toriques complexes},  Tohoku Math. J. (2) 48 (1996), no. 3, 363--390.

\bibitem[BBFK1]{BBFK1}
G. Barthel, J. P. Brasselet, K. H. Fieseler, L. Kaup, \emph{Combinatorial intersection cohomology for fans}, Tohoku Math. J. (2) 54 (2002), no. 1, 1--41.

\bibitem[BBFK2]{BBFK2}
G. Barthel, J. P. Brasselet, K. H. Fieseler, L. Kaup, \emph{Equivariant intersection cohomology of toric varieties}, Algebraic geometry: Hirzebruch 70 (Warsaw, 1998), 45--68, Contemp. Math., 241, Amer. Math. Soc., Providence, RI, (1999).

\bibitem[BL]{BL}
A. Beauville, Y. Laszlo, \emph{Un lemme de descente},
C.\,R.\ Acad.\ Sci.\ Paris S\'er.\ I Math.\ \textbf{320} (1995), 335--340.

\bibitem[Bir]{Bir} G. D. Birkhoff, \emph{Singular points of ordinary linear differential equations}. Transactions of the American Mathematical Society. 10 (4): 436--470.

\bibitem[Bir1]{Bir1} G.D. Birkhoff, \emph{A theorem on matrices of analytic functions}. Math. Ann. 74(1913),122--133.

\bibitem[BP]{BP} S. Bochner, S,. R. S. Phillips, \emph{Absolutely convergent Fourier expansions for non-commutative normed rings}. Ann. of Math. (2) 43 (1942), 409--418.


\bibitem[Bo]{Bo} A.Borel \emph{Some properties of adele groups attached to algebraic groups}
Bull. Amer. Math. Soc. 67 (1961), 583--585 


\bibitem[Bot]{Bot}  R. Bott, \emph{An application of the Morse theory to the topology of Lie-groups}
Bulletin de la S. M. F., tome 84 (1956), p. 251--281

\bibitem[BuVe]{BuVe} J. M. Burgos, A. Verjovsky, \emph{Adelic toric varieties and adelic loop groups}, arXiv:2001.07997 [math.AG] (2020).

\bibitem[BRS]{BRS} A. Brudnyi, L. Rodman, I. M. Spitkovsky, \emph{Non-denseness of factorable matrix functions}. J. Funct. Anal. \textbf{261} (2011), 1969--1991.

\bibitem[CC]{CC} A. Candel, L. Conlon, \emph{Foliations I}. Graduate Studies in Mathematics, 23. American Mathematical Society, Providence, RI, 2000.

\bibitem[CR]{CR}
E. Clader, Y. Ruan, \emph{Mirror Symmetry Constructions}, {arXiv:1412.1268}.

\bibitem[CG]{CG}
K. F. Clancey, I. Gohberg,  \emph{Factorization of matrix functions and singular integral operators}. Operator Theory: Advances and Applications, 3. Birkh\"auser Verlag, Basel-Boston, Mass., 1981


\bibitem[CLS]{Cox}
D. A. Cox, J. B. Little, H. K. Schenck, \emph{Toric  varieties}, Graduate Studies in Mathematics, vol. 124, American Mathematical Society, Providence, RI, (2011).

\bibitem[Co]{Co} A. A. Condori, \emph{Maximum principles for matrix-valued analytic functions}. Amer. Math. Monthly 127 (2020), no. 4, 331--343.

\bibitem[CS]{CS}
D. Clausen, P. Scholze, \emph{Condensed Mathematics},
lecture notes, University of Bonn (2019).
\url{https://www.math.uni-bonn.de/people/scholze/Condensed.pdf}

\bibitem[DPW]{DPW}  J. Dorfmeister, F. Pedit,H.  Wu, \emph{Weierstrass type representation of harmonic maps into symmetric spaces}. Comm. Anal. Geom. 6 (1998), no. 4, 633--668

\bibitem[E]{E} J. Eells Jr. \emph{A setting for global analysis}. Bull. Amer. Math. Soc. 72 (1966), 751--807.

\bibitem[Fr]{Fr} D. S. Freed, \emph{The Geometry of loop groups} 
Jour. Diff. Geom, 28 (1988) 223--276

\bibitem[FU]{FU} D. S. Freed, K. K. Uhlenbeck, \emph{Instantons and four-manifolds}, Math. Sci. Res. Inst. Publ., Vol. 1, Springer, New York, 1984.

\bibitem[FF]{FF}
L. Fargues, J.-M. Fontaine, \emph{Courbes et fibr\'es vectoriels en
th\'eorie de Hodge $p$-adique}, Ast\'erisque \textbf{406} (2018), xiii+382 pp.

\bibitem[FS]{FS}
L. Fargues, P. Scholze, \emph{Geometrization of the local Langlands
correspondence}, preprint (2021), arXiv:2102.13459.

\bibitem[Fu]{Fulton}
W. Fulton, \emph{Introduction to toric varieties}, vol. 131, Princeton University Press, (1993).

 \bibitem[GR]{GR} H. Garland, M.S, Raghunathan, 
 \emph{A Bruhat decomposition for the loop space of a compact group: A new approach to results of Bott.} Proc.Nat.Acad.Sci.USA
Vol. 72, No. 12, pp. 4716--4717, December 1975. 

\bibitem[Ge1]{Ge1} I.M. Gelfand
\emph{\"Uber absolut konvergente trigonometrische Reihen und Integrale.} (German. Russian summary)
Rec. Math. [Mat. Sbornik] N. S. 9 (51), (1941). 51--66.
 
\bibitem[Ge2]{Ge2} I.M. Gelfand, \emph{Normierte ringe Mat}. Sbornik N. S.9(51), 3--24, 1941.

\bibitem[GK]{GK}  I. C.Gohberg, M. G. Kre\u{i}n, \emph{Systems of integral equations on the half-line with kernels depending on the difference of the arguments}. (Russian) Uspehi Mat. Nauk (N.S.) 13 1958 no. 2 (80), 3--72

\bibitem[Gr1]{Gr1}
M. Gromov, \emph{Endomorphisms of symbolic algebraic varieties}, J. Eur. Math. Soc. (JEMS) 1 (1999), no. 2, 109--97.

\bibitem[Gr2]{Gr2}
M. Gromov, \emph{Topological invariants of dynamical systems and spaces of holomorphic maps I}, Math. Phys. Anal. Geom. 2 (1999), no. 4, 323--415.

\bibitem[MG]{MG} M. A Guest, \emph{From quantum cohomology to integrable systems}. Oxford Graduate Texts in Mathematics, 15. Oxford University Press, Oxford, 2008. 

 \bibitem[MG1]{MG1} M. A Guest, \emph{Harmonic maps, loop groups, and integrable systems}. London Mathematical Society Student Texts, 38. Cambridge University Press, Cambridge, 1997

\bibitem[Grot]{Grot} A. Grothendieck, \emph{Sur la classification des fibr\'es holomorphes sur la sph\`ere de Riemann}. American Journal of Mathematics. 79 (1): 121--138. 

\bibitem[SGA1]{SGA1} A. Grothendieck, M. Raynaud, \emph{Rev\^etement \'Etales et Groupe Fondamental (SGA1)}, Lecture Note in
Math., Springer, Berlin Heidelberg New York, vol.224, (1971).

\bibitem[HM]{HM} K. H. Hofmann, S. A. Morris,  \emph{The Lie theory of connected pro-Lie groups. A structure theory for pro-Lie algebras, pro-Lie groups, and connected locally compact groups}. EMS Tracts in Mathematics, 2. European Mathematical Society (EMS), Z\"urich, 2007.

\bibitem[HMu]{HMu} G. Horrocks, D. Mumford, \emph{A rank 2 vector bundle on $P^4$ with 15,000 symmetries}. Topology 12 (1973), 63--81.

\bibitem[Ka]{Ka} V. G. Kac, \emph{Infinite-dimensional Lie algebras.} Second edition. Cambridge University Press, Cambridge, 1985

\bibitem[KLMV]{KLMV}
L. Katzarkov, E. Lupercio, L. Meersseman, A. Verjovsky, \emph{Quantum (non-commutative) toric geometry: Foundations}, Adv. Math. 391 (2021), Paper No. 107945.

 
\bibitem[K]{K} Y. Katznelson, \emph{An introduction to harmonic analysis}. Second corrected edition. Dover Publications, Inc., New York, 1976

\bibitem[Ked1]{Ked1}
K. S. Kedlaya, \emph{A $p$-adic local monodromy theorem},
Ann.\ of Math.\ (2) \textbf{160} (2004), no.\ 1, 93--184.

\bibitem[Ked2]{Ked2}
K. S. Kedlaya, \emph{Slope filtrations revisited},
Doc.\ Math.\ \textbf{10} (2005), 447--525.

\bibitem[Ke]{Ke} P. Kellersch,
\emph{Eine Verallgemeinerung der Iwasawa Zerlegung in Loop Gruppen}. (German. German summary) [The Iwasawa decomposition for the untwisted loop group of semisimple Lie groups]
Dissertation, Technische Universit\"at M\"unchen, Munich, 1999

\bibitem[Lau]{Lau} L. M. A. Lau,  \emph{On the Sato-Segal-Wilson Grassmannian and the Infinite Grassmannian of Type $I+-$}. MSc thesis, Hong Kong University of Science and Technology, 2015.
\url{http://hdl.handle.net/1783.1/80215}



\bibitem[Le]{Le} P. L\'evy, \emph{Sur la convergence absolue des s\'eries de Fourier}. Compositio Mathematica. 1: 1--14 (1935).

\bibitem[LM]{LM}
M. Lyubich, Y. Minsky, \emph{Laminations in holomorphic dynamics}, J. Differential Geom. 47 (1997), no. 1, 17--94.

\bibitem[Mi]{Mi} J. Milnor, \emph{Remarks on infinite-dimensional Lie groups}, pp. 1007--1057 in: ``Relativit\'e, Groupes et Topologie II,'' B. DeWitt and R. Stora (Eds), North-Holland, Amsterdam, 1983.

\bibitem[Mo]{Mo} C. C. Moore, C. L. Schochet, \emph{Global analysis on foliated spaces}. Second edition. Mathematical Sciences Research Institute Publications, 9. Cambridge University Press, New York, 2006

\bibitem[Od]{Od}
C. Odden, \emph{The baseleaf preserving mapping class group of the universal hyperbolic solenoid}, Trans. Amer. Math. Soc. 357 (2005), 1829-1858.

\bibitem[P]{P} R. S. Palais, \emph{Morse theory on Hilbert manifolds}. Topology 2 (1963), 
299--340.

 \bibitem[Pe]{Pe} J.-P. Penot, \emph{Sur le th\'eor\`eme de Frobenius}, Bull. Soc. Math. France, 98 (1970), 47--80.

\bibitem[Pr]{Pr} A. N. Pressley, \emph{Decompositions of the space of loops on a Lie group}.
Topology 19 (1980), no. 1, 65--79

\bibitem[Pr1]{Pr1} A. N. Pressley,  \emph{The energy flow on the loop space of a compact Lie group}. J. London Math. Soc. (2) 26 (1982), no. 3, 557--566.

\bibitem[Pr2]{Pr2} A. N. Pressley, \emph{Loop Groups}, Encyclopedia of Physical Science and Technology, Elsevier (Third Edition)
2003, Pages 791--798



\bibitem[PS]{PS}
A. Pressley, G. Segal, \emph{Loop groups}. Oxford Mathematical Monographs. Oxford Science Publications. The Clarendon Press, Oxford University Press, New York, 1986. viii+318 pp.

\bibitem[RV]{RV}
D. Ramakrishnan, R. Valenza, \emph{Fourier Analysis on Number Fields}, Graduate Texts in Mathematics {\bf 186}, Springer-Verlag,
New York, (1999).

\bibitem[Sc]{Sc}
W. Scheffer, \emph{Maps between topological groups that are homotopic to homomorphisms}, Proc. Amer. Math. Soc., 33 (1972), 562--567.

\bibitem[Sch]{Sch}
P. Scholze, \emph{Perfectoid spaces},
Publ.\ Math.\ Inst.\ Hautes \'Etudes Sci.\ \textbf{116} (2012), 245--313.


\bibitem[Se]{Se} G. Segal, \emph{Unitary representations of some infinite-dimensional groups}.
Comm. Math. Phys. 80 (1981), no. 3, 301--342. 

\bibitem[Se1]{Se1} G. Segal,  \emph{Loop groups}. Workshop Bonn 1984 (Bonn, 1984), 155--168, Lecture Notes in Math., 1111, Springer, Berlin, 1985.

\bibitem[SW]{SW} G. Segal and G. Wilson, \emph{Loop groups and equations of KdV type}, Pub. Math,
de LH.E.S. 61 (1985), 5--65.

\bibitem[S]{S}J.-P. Serre
\emph{Groupes proalg\'ebriques}
Publications math\'ematiques de l'I.H.\'E.S., tome 7 (1960), p. 5--67

\bibitem [Su1]{Su1}
D. Sullivan, \emph{Linking the universalities of Milnor-Thurston, Feigenbaum and Ahlfors-Bers}, Lisa R. Goldberg et al (ed.), in ``Topological methods in modern mathematics''. Proceedings of a symposium in honor of John Milnor's sixtieth birthday, held at the State University of New York at Stony Brook, USA, June 1991. Houston, TX: Publish or Perish, Inc. 543--564 (1993). 



\bibitem[Su2]{Su2}
D. Sullivan, \emph{Quasiconformal Homeomorphisms and Dynamics I: Solution of the Fatou-Julia Problem on Wandering Domains}, Annals of Mathematics, Second Series, Vol. 122, No. 2 (Sep., 1985), pp. 401--418.

\bibitem[Su3]{Su3}
D. Sullivan, \emph{Quasiconformal homeomorphisms and dynamics II: Structural stability implies hyperbolicity for Kleinian groups}, Acta Math., Volume 155 (1985), 243--260.

\bibitem[Su4]{Su4}
D. Sullivan, \emph{Solenoidal manifolds}, J. Singul. 9 (2014), 203--205.

\bibitem[TT]{TT} T. Tao, \emph{Lecture Notes 3 for 254A},
Week 4 notes: Product estimates, multilinear estimates.
\url{https://www.math.ucla.edu/~tao/254a.1.01w/}

\bibitem[T]{T} J. T. Tate, \emph{Fourier analysis in number fields, and Hecke's zeta-functions}, Algebraic Number Theory (Proc. Instructional Conf., Brighton, 1965), Thompson, Washington, D.C., pp. 305--347


\bibitem[Ve]{Ve}
A. Verjovsky, \emph{Commentaries on the paper `Solenoidal manifolds' by Dennis Sullivan}, [MR3249058]. J. Singul. 9 (2014), 245--251.

\bibitem[Ve1]{Ve1}
A. Verjovsky, \emph{Low-dimensional solenoidal manifolds}, EMS Surv. Math. Sci. 10 (2023), no. 1, 131--178.

\bibitem[Yo1]{Yokura1}
S. Yokura, \emph{Characteristic classes of proalgebraic varieties and motivic measures}, Algebr. Geom. Topol. 12 (2012), no. 1, 601--641. 
 
\bibitem[Yo2]{Yokura2}
S. Yokura, \emph{Characteristic classes of (pro)algebraic varieties}, Singularities in geometry and topology 2004, 299--329, Adv. Stud. Pure Math., 46, Math. Soc. Japan, Tokyo, (2007).

\bibitem[Fors]{Fors} F. Forstneri\v{c}, \emph{Stein Manifolds and Holomorphic
Mappings}. Ergebnisse der Mathematik und ihrer Grenzgebiete (3), 56.
Springer-Verlag, Berlin, 2011; 2nd ed., 2017.

\bibitem[GRem]{GRem} H. Grauert, R. Remmert, \emph{Theory of Stein Spaces}.
Grundlehren der mathematischen Wissenschaften, 236.
Springer-Verlag, Berlin--Heidelberg--New York, 1979.

\bibitem[DG]{DG} F. Docquier and H. Grauert, \emph{Levisches Problem und
Rungescher Satz f\"ur Teilgebiete Steinscher Mannigfaltigkeiten},
Math.\ Ann.\ \textbf{140} (1960), 94--123.

\bibitem[Paty]{Paty} I. Patyi, \emph{On holomorphic Banach vector bundles over Banach spaces},
Math.\ Ann.\ \textbf{341} (2008), no.\ 3, 455--482.

\bibitem[Lem]{Lem} L. Lempert, \emph{The Dolbeault complex in infinite dimensions, III},
Invent.\ Math.\ \textbf{142} (2000), 579--603.

\bibitem[We]{We} A. Weil, \emph{Adeles and algebraic groups}. With appendices by M. Demazure and Takashi Ono. Progress in Mathematics, 23. Birkh\"auser, Boston, Mass., 1982. iii+126 pp.

\bibitem[W]{W}  N. Wiener, \emph{The Fourier integral and certain of its applications}, Cambridge University Press, 1933.
 
\bibitem[Z]{Z} A. Zygmund, \emph{Trigonometric Series}. Cambridge: Cambridge University Press. p. 245.
\end{thebibliography}
\end{document}